%% file: DQexact.tex
\documentclass[11pt,twoside,a4paper]{amsart}

\usepackage{amssymb}
\usepackage{vmargin}
\usepackage{amscd}
  \usepackage{stmaryrd}
\usepackage{mathrsfs}
\usepackage[all]{xy}
\usepackage{enumitem}
\usepackage{xr-hyper}
\usepackage{hyperref}
\hypersetup{colorlinks,
allcolors=black}

\setmargins{32mm}{20mm}{14.6cm}{22cm}{1cm}{1cm}{1cm}{1cm}

\DeclareMathAlphabet      {\mathup}{OT1}{\sfdefault}{m}{n}

\include{newcommands}

\include{newthms}

\begin{document}
 
\begin{abstract} 
 We extend  \cite{poisson,CPTVV} to  establish a correspondence between exact shifted symplectic structures and non-degenerate shifted Poisson structures with formal derivation, a concept  generalising constructions by De Wilde and Lecomte \cite{DeWildeLecomteExactSymplectic}.  Our formulation is sufficiently general to encompass derived algebraic, analytic and $\C^{\infty}$-stacks, as well as Lagrangians and non-commutative generalisations. We also show that non-degenerate shifted Poisson structures with formal derivation
 carry unique self-dual deformation quantisations in any setting where the latter can be formulated.
 
One application is that for (not necessarily exact) $0$-shifted symplectic structures in analytic and $\C^{\infty}$ settings, it follows that the author's earlier parametrisations of quantisations \cite{DQnonneg,DQpoisson} are in fact independent of any choice of associator, and generalise Fedosov's parametrisation of quantisations for classical manifolds. 

Our main application  is to complex  $(-1)$-shifted symplectic structures, showing that our unique quantisation of the canonical exact structure, a sheaf of twisted $BD_0$-algebras with derivation, gives rise to the 
perverse sheaf of vanishing cycles from \cite{BBDJS}, equipped with its monodromy operator.
\end{abstract}

\title[Deformation quantisation of  exact shifted symplectic structures]{Deformation quantisation of  exact shifted symplectic 
structures, 
with an application to vanishing cycles}

\author{J. P. Pridham}

\maketitle

\section*{Introduction}

Shifted symplectic structures in derived algebraic geometry were introduced in \cite{PTVV} as cocycles  in the truncated derived de Rham  complex $ (\prod_{p \ge 2}\oL\Omega^p[p], \delta \pm d)$ satisfying a non-degeneracy condition on $\omega_2$; closely related  concepts  in the  pre-existing graded manifold literature notably include the homotopy symplectic structures of \cite{KhudaverdianVoronov}. An equivalence was established in \cite{poisson,CPTVV} between shifted symplectic structures and non-degenerate shifted Poisson structures, and the formulation of \cite{poisson} readily admits an interpretation as a homotopical Legendre transformation generalising the equivalence of \cite{KhudaverdianVoronov}.

A shifted symplectic structure $\omega$ is called exact if it comes equipped with a homotopy $h$ killing its image in (non-truncated) derived de Rham cohomology. Equivalently, an exact  shifted symplectic structure  can be characterised as a cocycle $\alpha$ in the cone of $d \co  \oL\Omega^0 \to\oL\Omega^1$ such that the associated pre-symplectic structure $d \alpha_1$ satisfies a non-degeneracy condition.  

Negatively shifted symplectic structures are automatically exact, a phenomenon underpinning construction of a perverse sheaf of vanishing cycles  associated to $(-1)$-shifted symplectic structures in \cite{BBDJS}, and further refinements in  work such as \cite{KinjoParkSafronov,HHR1}. On underived differentiable manifolds,  $0$-shifted structures are locally exact, which facilitates the parametrisation of deformation quantisations in
\cite{fedosov,DeWildeLecomte,deligneDefFnsSymplectic}. 

In this paper, we show how to extend the homotopical Legendre transformation of \cite{poisson} to
give an equivalence between  exact  shifted symplectic structures and non-degenerate shifted Poisson structures equipped with certain formal derivations, generalising the data of \cite{DeWildeLecomteExactSymplectic} in the classical $0$-shifted case. Our results are couched in sufficient generality to apply in derived algebraic, $\C^{\infty}$ and analytic geometries, as well as to their non-commutative analogues, namely shifted bisymplectic and double Poisson structures, both algebraic and analytic, and also to shifted Lagrangians and co-isotropic structures.

We establish that the analogue of exactness for  a shifted Poisson structure $\pi = \sum_{i \ge 2} \pi_i$ is a homotopy $D$ killing the element $\sum_{i \ge 2} (i-1)\pi_i$ in Poisson cohomology (Proposition \ref{exactcompprop} and Examples \ref{poissonex}). 
Algebraic interpretations of such structures $D$ are given in \S \ref{interpretalgsn}. When $\pi =\pi_2$, examples are given by closed derivations satisfying  $D\{a,b\}_{\pi}=\{Da,b\}_{\pi} \pm \{a, Db\}_{\pi} - \{a,b\}_{\pi}$ (\S \ref{strictsn}). 

We also address deformation quantisation questions, looking at quantisations of non-degenerate shifted Poisson structures to which the formal derivation lifts. In Corollary \ref{exactquantcor}, Examples \ref{quantex} and \S \ref{interpretDQsn}, we show that every non-degenerate shifted Poisson structure with formal derivation, and hence any exact shifted symplectic structure, has an essentially unique self-dual deformation quantisation with $\hbar\pd_{\hbar}$-derivation, for any shift with a notion of self-dual deformation quantisation (i.e. $n \ge -1$ for Poisson structures and $n\ge 0$ for co-isotropic structures). The same approach allows spaces of shifted Poisson structures with formal derivation to be reinterpreted as spaces of pairs $(\pi_{\hbar}, \sD_{\hbar})$, where $\pi_{\hbar}$ is a Poisson structure over the power series ring and $\sD_{\hbar}$ a Poisson $\infty$-derivation fixing  $\hbar$ (Corollary \ref{exactquantbrhcor} and \S \ref{formalderivationsn}). 

A significant application to shifted symplectic structures without exact data is in Proposition \ref{DWLprop}, Corollary \ref{DWLcor1} and Examples \ref{DWLex}, showing that in $\Cx$-analytic and  $\C^{\infty}$ contexts where the de Rham theorem holds, 
 the parametrisations of  self-dual (weak) quantisations of shifted symplectic ($n \ge 0$) and shifted Lagrangian ($n \ge 1$) structures given by even associators are all equivalent. The preferred  quantisation of a $0$-shifted symplectic structure  agrees with that given by generalising the method of De Wilde--Lecomte, and for manifolds our parametrisation of quantisations agrees with Fedosov's (Example \ref{DWLex}.(\ref{fedosovex})).  
 
 Our main applications are to $(-1)$-shifted symplectic structures, since these are known both to be canonically exact, and to admit a notion of quantisation with duality. 
Theorem \ref{DCritDmodthm} shows that on the derived critical locus of a function $f$, the unique self-dual quantisation  with $\hbar\pd_{\hbar}$-derivation of the canonical exact $(-1)$-shifted symplectic structure gives rise, on inverting $\hbar$, to the canonical bundle with $\sD$-module structure twisted by $f$. This implies (Corollary \ref{DcritAutPervCor}) that the associated monodromic perverse sheaf (via  \cite{GunninghamSafronov} or \cite{SchefersDerivedVFiltrn})  is isomorphic to the vanishing cycles complex. The rigidity of exact self-dual quantisations guarantees compatibility of gluings, ensuring (Theorem \ref{cfHHRthm}) that for any polarised $(-1)$-shifted symplectic derived DM $\Cx$-stack with its canonical exact structure, the unique self-dual quantisation gives rise to the perverse sheaf of vanishing cycles from \cite{BBDJS}, or to its $\mu_2$-twisted generalisation from \cite{HHR1} in the absence of a polarisation. 

\smallskip
The structure of the paper is as follows.
Section \ref{legendre0sn} recalls the underlying features common to homotopical Legendre transformations  and their quantised analogues. We exploit that setup in Section \ref{exactsn} to establish a general Poisson analogue of exact shifted  symplectic structures (Proposition \ref{exactcompprop}), leading to the examples described above, with a quantised version in Corollary \ref{exactquantcor}. \S \ref{DWLsn1} gives applications to canonical quantisation of non-exact symplectic structures.

Section \ref{interpretnsn} shows how to interpret the resulting structures in algebraic terms, primarily focusing on the algebro-geometric setting. Lemmas \ref{epslemma}, \ref{epslemma2} and \ref{epslemmagenl} give deformation-theoretic interpretations of the governing differential graded Lie algebra, with Lemma \ref{epslemmaquant} giving a quantised analogue.  Examples \ref{quantinterpretnex} illustrate these in various quantisation contexts. 

Finally, Section \ref{vanishsn} addresses the case of $(-1)$-shifted symplectic structures in detail. We begin in  \S \ref{Dinvsn} by establishing existence of an essentially unique anti-involutive ring of twisted differential operators. By applying Corollary \ref{exactquantcor} to the resulting construction, in  \S \ref{invBVsn} we establish unique self-dual quantisations for exact $(-1)$-shifted symplectic structures, which lead to twisted $\sD\brh$-modules with $\hbar\pd_{\hbar}$-derivation. On the derived critical locus of a function $f$, Theorem \ref{DCritDmodthm} identifies the unique self-dual quantisation with a right  $\sD$-module given by   twisting the canonical bundle by $f$, which leads to the comparisons with vanishing cycles described above. 

I would like to thank Pavel Safronov for suggesting relevant references  and other helpful comments.

\tableofcontents

\section{Homotopical and quantised Legendre transformations}\label{legendre0sn}



\subsection{Maurer--Cartan spaces}

We briefly recall some essential background theory underpinning our constructions; further details can be found in the references given.

\begin{definition}\label{mcPLdef}
 Given a $\Q$-linear  differential graded Lie algebra (DGLA) $L$, define  the Maurer--Cartan set by 
\[
\mc(L):= \{\omega \in  L^{1}\ \,|\, \delta\omega + \half[\omega,\omega]=0 \in  \bigoplus_n L^{2}\}.
\]

Following \cite{hinstack}, define the Maurer--Cartan space $\mmc(L)$ (a simplicial set) of a nilpotent  DGLA $L$ by
\[
 \mmc(L)_n:= \mc(L\ten_{\Q} \Omega^{\bt}(\Delta^n)),
\]
where 
\[
\Omega^{\bt}(\Delta^n)=\Q[t_0, t_1, \ldots, t_n,\delta t_0, \delta t_1, \ldots, \delta t_n ]/(\sum t_i -1, \sum \delta t_i)
\]
is the commutative dg algebra of de Rham polynomial forms on the $n$-simplex, with the $t_i$ of degree $0$.
\end{definition}

\begin{definition}
Given an inverse system $L=\{L_{\alpha}\}_{\alpha}$ of nilpotent DGLAs, define
\[
 \mc(L):= \Lim_{\alpha} \mc(L_{\alpha}) \quad  \mmc(L):= \Lim_{\alpha} \mmc(L_{\alpha}).
\]
Note that  $\mc(L)= \mc(\Lim_{\alpha}L_{\alpha})$, but $\mmc(L)\ne \mmc(\Lim_{\alpha}L_{\alpha}) $. 
\end{definition}


\subsection{Essentials of the homotopical Legendre transformation}\label{legendresn}

We now summarise the essential features of the  homotopical Legendre transformation between shifted symplectic  and non-degenerate shifted Poisson structures from \cite[\ref{poisson-compsn}]{poisson}, 
 since adapted to many other settings, including \cite{DStein,NCpoisson}. See \cite{KhudaverdianVoronov} for a precursor.

 \subsubsection{Poisson structures}\label{genpoissstrsn}

 In any setting where shifted Poisson structures  are defined over a base $R$,  there is a DGLA $L$ of polyvectors (or multiderivations), satisfying the following.
 
 \begin{setting}\label{Lsetting}
  Let $L$ denote an $R$-linear DGLA $L$ admitting a product decomposition $L=\prod_{m \ge -1}\cW_mL$ by subcomplexes $\cW_mL$, with $[\cW_mL,\cW_nL]\subset \cW_{m+n}L$.
  \end{setting}
In the primary  motivating examples, $\cW_mL$ is a complex of shifted $(m+1)$-vectors. 

 \begin{definition}\label{Fdef}
  The weight decomposition $\cW$ is used to construct a complete decreasing  filtration $F$ by $F^pL:= \prod_{m \ge p-1}\cW_mL$, so $[F^iL,F^jL]\subset F^{i+j-1}L$ and $F^0L=L$.
 \end{definition}
 
 \begin{definition}\label{sigmadef}
  The weight decomposition $\cW$ also gives rise to a Lie algebra derivation $\sigma$ on $L$, given by multiplying $\cW_mL$ by $m$; this is automatically a chain map and respects the filtration $F$.
 
 In particular,  note that $\sigma$ acts as multiplication by $p-1$ on $\gr_F^pL$.
 \end{definition}


\begin{definition}
 The space of Poisson structures associated to $L$ is then taken to be the Maurer--Cartan space $\mmc(F^2L):= \Lim_n \mmc(F^2L/F^{2+n}L)$ of the pro-nilpotent DGLA $F^2L$.
\end{definition}

\subsubsection{Tangent Poisson structures and the derivation \tps{$\sigma$}{sigma}}
For a formal variable $\eps$ of degree $0$ with $\eps^2=0$, we have a morphism $\id + \sigma \eps \co L \to L[\eps]$ of filtered DGLAs, and hence a morphism  $\mmc(F^2L) \to \mmc(F^2L[\eps])$, which is a non-trivial section of the projection map given by $\eps \mapsto 0$. 
\begin{definition}
 We refer to the space $\mmc(F^2L[\eps])$ as being the tangent space of the space of Poisson structures associated to $L$.  
\end{definition}

\subsubsection{(Pre-)symplectic structures}
Shifted symplectic structures are much easier to define, being more linear in nature. 

\begin{setting}\label{dRsetting}
 Let $V$ denote a cochain complex which is complete with respect to a decreasing filtration $F$.
\end{setting}
In our motivating examples, $V$ is a shifted total de Rham complex and $F$ is the Hodge filtration given by brutal truncation in the de Rham direction. 


\begin{definition}
Regarding $V$ as a DGLA with trivial Lie bracket, the space of pre-symplectic structures associated to $V$ is then taken to be  the Maurer--Cartan space $\mmc(F^2V):=\Lim_n \mmc(F^2V/F^{2+n}V)$ of $F^2V$ regarded as a pro-object; this is weakly equivalent to the Dold--Kan denormalisation of the good truncation $\tau^{\le 0}V$. 
\end{definition}

The space of shifted symplectic structures is then a union of path components of $\mmc(F^2V)$ consisting of elements satisfying a non-degeneracy condition.

\subsubsection{The compatibility map}\label{gencompatsn}

The key to the homotopical Legendre transformation is then the existence of an $L_{\infty}$-derivation
\[
 \mu = \{\mu^{(i)}\}_{i \ge 1}  \co V \by F^2L \to L
\]
which is linear in $V$ in the sense that the only non-zero terms are $\mu^{(n+1)} \co V \ten (F^2L)^{\ten n} \to L[-n]$. These respect the 
 filtrations in the sense that 
 \[
 \mu^{(n+1)}(F^pV \ten (F^{2+j_1}L)\ten \ldots (F^{2+j_n}L)  ) \subset F^{p+j_1 +\dots +j_n}L.
\]

 
Moreover, for all $k$ there exists $n_0(k)$ such that  $\mu^{(n+1)}(F^0V \ten (F^2L)^{\ten n}) \subset F^kL$ for all $n \ge n_0(k)$. This ensures that on Maurer--Cartan elements $x$, the sum 
\begin{align*}
x 
\mapsto \sum_{i \ge 1} \mu^{(i)}(x,x, \ldots,x)/i!, \quad\text{i.e.}\quad
(\omega, \pi) 
\mapsto \sum_{n \ge 0} \mu^{(n+1)}(\omega,\pi,\pi, \ldots,\pi)/n! 
\end{align*}
is finite modulo $F^kL$, so converges. The same is true on tensoring with any commutative algebra, so this gives us a map
\begin{align*}
 (\pr_2 + \mu\eps) \co \mmc(F^2V) \by \mmc(F^2L) &\to  \mmc(F^2L \oplus F^2L\eps)\\
(\omega, \pi) &\mapsto  \pi + \eps\sum_{n \ge 0} \mu^{(n+1)}(\omega,\pi,\pi, \ldots,\pi)/n! 
 \end{align*}
of simplicial sets over $\mmc(F^2L)$ fibred in simplicial abelian groups,
which we call the compatibility map.

In practice, $\mu$ is defined on $1$-forms by contraction, then extended multiplicatively to the whole of $V$ using a ring structure on $V$ and a Poisson algebra structure on $L$. Thus $\mu^{(n+1)}$ picks out only the $n$-forms in $V$.

\subsubsection{The space of compatible pairs}

\begin{definition}\label{vanishingdef}
Given a simplicial set $Z$, an abelian group object $A$ in simplicial sets over $Z$, and a section $s \co Z \to A$, define the homotopy vanishing locus of $s$ to be the homotopy limit of the diagram
\[
\xymatrix@1{ Z \ar@<0.5ex>[r]^-{s}  \ar@<-0.5ex>[r]_-{0} & A \ar[r] & Z}.
\]
\end{definition}

We can write this as a homotopy fibre product $Z \by_{(s,0), A \by^h_Z A}^hA$, for the diagonal map $A \to A \by^h_Z A$. When $A$ is a trivial bundle $A = Z \by V$, for $V$ a simplicial abelian group, note that the homotopy vanishing locus is just the homotopy fibre of $s \co Z \to V$ over $0$.

\begin{definition}\label{compdef}
For $L,V,\sigma, \mu$ as above, define the space $\Comp(L,V,\sigma, \mu )$ of compatible  pairs to be the homotopy vanishing locus of
\begin{align*}
 (\id, \mu - \sigma) \co \mmc(F^2V) \by \mmc(F^2L) &\to \mmc(F^2V) \by \mmc(F^2L[\eps])\\
 (\omega, \pi) &\mapsto (\omega, \pi + (\mu(\omega, \pi) -\sigma(\pi))\eps)
\end{align*}
\end{definition}
%

In particular, an element of this space is given by a pre-symplectic structure $\omega$, a Poisson structure $\pi$, and a homotopy $h$ between $\mu(\omega,\pi)$ and $\sigma(\pi)$ in the fibre of $\mmc(F^2L[\eps]) \xra{\eps \mapsto 0} \mmc(F^2L)$ over $\pi$.
%

\subsubsection{The comparisons}\label{compequivsn}

There is a notion of  non-degeneracy for  shifted Poisson structures $\pi \in \mmc(F^2L)$, dependent only on the homotopy class. Write $\mmc(F^2L)^{\nondeg} \subset \mmc(F^2L)$ for the space of non-degenerate elements (a union of path components). 

When $\pi \in \mc(F^2L)$ is non-degenerate, the map 
\begin{align*}
\mu(-,\pi)\co V &\to  L^{\pi}:=(L, \delta + [\pi,-])\\
\omega &\mapsto \sum_{n \ge 0} \mu^{(n+1)}(\omega,\pi,\pi, \ldots,\pi)/n!
\end{align*}
is a filtered quasi-isomorphism.

If we also write $\Comp(L,V,\sigma, \mu )^{\nondeg}:= \Comp(L,V,\sigma, \mu )\by^h_{\mmc(F^2L)}\mmc(F^2L)^{\nondeg}$, the space of compatible pairs $(\omega,\pi)$ with $\pi$ non-degenerate, then it follows immediately from the filtered quasi-isomorphism that the map
\[
 \Comp(L,V,\sigma, \mu )^{\nondeg} \to \mmc(F^2L)^{\nondeg}
\]
is a weak equivalence.

We thus have a  map from non-degenerate Poisson to pre-symplectic structures,  given by the diagram
\[
 \mmc(F^2L)^{\nondeg} \xla{\sim} \Comp(L,V,\sigma, \mu )^{\nondeg} \to \mmc(F^2V).
\]

In fact (as in \cite[Corollary \ref{poisson-compatcor1}]{poisson}), it always turns out in cases of interest that the map  $\Comp(L,V,\sigma, \mu )^{\nondeg} \to \mmc(F^2V)^{\nondeg}$ to the space of shifted symplectic structures is also an equivalence. This comparison is more subtle: it starts from a duality between $\gr_F^2V$ ($2$-forms) and $\gr_F^2L$ (bivectors), and then relies on an obstruction theory argument working up the filtration $F$. The key calculation is that the maps 
\[
 \nu(\omega, \pi) - \sigma \co   \gr_F^pL \to \gr_F^pL
\]
turn out to be quasi-isomorphisms for $(\omega,\pi) \in \Comp(L/F^3L,V/F^3V,\sigma, \mu )^{\nondeg}$, where $\sigma$ is $(p-1)$ here, and 
\[
\nu(\omega, \pi)(b):= \sum_{n \ge 0} \mu^{(n+2)}(\omega,\pi,\pi, \ldots,\pi,b)/n!.
\]

\subsection{Deformation quantisations}\label{quantsn}


Where shifted Poisson structures over a base $R$ admit a notion of deformation quantisation, we have the following setup.

\begin{setting}\label{breveLsetting}
 Let $\breve{L}$ denote an $R$-linear DGLA equipped with an exhaustive increasing filtration $W$ with $W_{-2}\breve{L}=0$, 
 such that $[W_i\breve{L},W_j\breve{L}] \subset W_{i+j}\breve{L}$, together with a graded quasi-isomorphism $\gr^W\breve{L} \to L$ of DGLAs, for $L$ as in Setting \ref{Lsetting}. 
\end{setting}


In Examples \ref{quantex}, we will see that in cases of interest $\breve{L}$ tends to be a cohomological Hochschild complex or a ring of differential operators, with $W$ a reindexed filtration by Hochschild degree or by order.

\begin{definition}\label{tildeLdef}
We then form 
an $\hbar$-adically complete  $R\brh$-linear DGLA $\tilde{L}$ of quantised polyvectors as the complete Rees construction $\widehat{\Rees}(\breve{L},W):= \prod_m \hbar^mW_m\breve{L}$, and set $\tilde{\cW}_m\tilde{L}:= \hbar^mW_m\breve{L}$.
\end{definition}

The pair $(\tilde{L},\tilde{\cW})$ satisfies the conditions for $(L,\cW)$ in Setting \ref{Lsetting} (but not in general those of \S \ref{gencompatsn}). 
\begin{definition}
In particular, let  $\tilde{F}$ be the  complete decreasing filtration on $\tilde{L}$ given by Definition \ref{Fdef}, so  $\tilde{F}^p\tilde{L}= \prod_{m \ge p-1}\tilde{\cW}_m\tilde{L}= \prod_{m \ge p-1} \hbar^mW_m\breve{L} $.

Similarly, let $\tilde{\sigma}$ be the derivation on $\tilde{L}$ given by Definition \ref{sigmadef}, so  $\tilde{\sigma}=\hbar\pd_{\hbar}$. 
\end{definition}

The graded quasi-isomorphism $\gr^W\breve{L} \to L$ is the same as a quasi-isomorphism $\tilde{L}/\hbar\tilde{L} \to L$, where  the induced grading $\cW$ on $\tilde{L}/\hbar\tilde{L}$ is given by $\cW_m(\tilde{L}/\hbar\tilde{L}):= \tilde{\cW}_m\tilde{L}/\hbar \tilde{\cW}_{m-1}\tilde{L} = \hbar^m\gr^W_m\breve{L}$. We thus have a filtered quasi-isomorphism $ \tilde{L}/\hbar \tilde{L} \to L$, equivariant with respect to the respective derivations $\tilde{\sigma}$ and $\sigma$.

 \begin{definition}
The space of quantised Poisson structures associated to $(\breve{L},W)$ is then taken to be the Maurer--Cartan space $\mmc(\tilde{F}^2\tilde{L}):= \Lim_n \mmc(\tilde{F}^2\tilde{L}/\tilde{F}^{2+n}\tilde{L})$ of the pro-nilpotent DGLA $\tilde{F}^2\tilde{L}$. 

The map from the space of quantisations to the space of underlying Poisson structures is the the natural composite $\mmc(\tilde{F}^2\tilde{L}) \to (\tilde{F}^2\tilde{L}/\hbar\tilde{F}^1\tilde{L})\simeq \mmc(F^2L)$.
\end{definition}

\begin{remark*}
Beware that our convention for $\tilde{\sigma}$ differs slightly from that of other papers by the author, where there is a derivation $\sigma$ defined from $\tilde{L}$ to $\hbar\tilde{L}$ (corresponding to our $\hbar\tilde{\sigma}$), rather than to $\tilde{L}$ itself,  in order to simplify expressions for the quantised compatibility map. 
\end{remark*}

The key to quantisation arguments is existence of the following.
\begin{setting}\label{isetting}
 Assume that our filtered DGLA $(\breve{L},W)$ is equipped with an involution $i$, acting as (or at least homotopic to) multiplication by $(-1)^m$ on $\gr^W_m\breve{L}$.
\end{setting}

\begin{definition}
 Given an involution $i$ as in Setting \ref{isetting}, define an involution $*$ of $\tilde{L}=\widehat{\Rees}(\breve{L},W)$
 by setting $v(\hbar)^*:= i(v(-\hbar))$.
 \end{definition}

In particular, $*$ acts as (or is at least homotopic to) the identity on $\tilde{L}/\hbar\tilde{L}$, respects the filtration $\tilde{F}$ on $\tilde{L}$  and satisfies $(\hbar v)^* = -\hbar v^*$ and $\tilde{\sigma}(v^*)=\tilde{\sigma}(v)^*$.  

The DGLA $\tilde{L}$  thus splits  as a direct sum of a sub-DGLA $\tilde{L}^{*=\id}$ and a subcomplex $\tilde{L}^{*=-\id}$.

\begin{lemma}
 
 The inclusion $\hbar^2 (\tilde{F}^{p-2}\tilde{L}^{*=\id}) \subset \tilde{F}^p\tilde{L}^{*=\id} \cap \hbar\tilde{L}$ is a quasi-isomorphism.
 \end{lemma}
\begin{proof} For all $q$, we have
 \begin{align*}
  \tilde{F}^q\tilde{L} \cap \hbar\tilde{L}&=(\prod_{m \ge q-1} \hbar^mW_m\breve{L})\cap (\prod_{r} \hbar^{r+1}W_r\breve{L})\\
  &=\prod_{m \ge q-1} \hbar^mW_{m-1}\breve{L}  = \hbar \tilde{F}^{q-1}\tilde{L}. 
\end{align*}

The inclusion $\tilde{F}^q\tilde{L}^{*=-\id} \cap \hbar\tilde{L} \into \tilde{F}^q\tilde{L}^{*=-\id}$ is a quasi-isomorphism for all $q$, since $*$ is   homotopic to the identity on $\tilde{F}^q( \tilde{L}/\hbar\tilde{L}) \simeq F^qL$.
We thus have a quasi-isomorphism
\[
\hbar^2 (\tilde{F}^{p-2}\tilde{L}^{*=\id})= \hbar( \tilde{F}^{p-1}\tilde{L}^{*=-\id} \cap \hbar\tilde{L}) \into \hbar (\tilde{F}^{p-1}\tilde{L}^{*-=\id})=\tilde{F}^p\tilde{L} \cap \hbar\tilde{L}. \qedhere
\]
\end{proof}

\begin{definition}
The space of self-dual  quantised Poisson structures associated to $(\breve{L},W,i)$ is then taken to be the Maurer--Cartan space $\mmc(\tilde{F}^2\tilde{L})^{sd}:= \mmc(\tilde{F}^2\tilde{L}^{*=\id})$.
\end{definition}

\begin{remark}\label{Lbrhrmk}
One example of a DGLA satisfying the conditions for $\breve{L}$ is $L$ itself, with weight filtration $W_mL:=\bigoplus_{j \le m} \cW_jL$. Then $\widehat{\Rees}(L,W) = \prod_m \hbar^m\cW_m\brh$, with $\tilde{\sigma}= \hbar\pd_{\hbar}$. We can take the involution $i$ of $L$ to be multiplication by $(-1)^j$ on $\cW_jL$, giving an involution $*$ on $\widehat{\Rees}(L,W)$ as $(v\hbar^n)^*= (-1)^{m+n}v\hbar^n$ for $v \in \cW_mL$. 

In fact, there is an isomorphism $(-)_{\hbar} \co L\brh \to \widehat{\Rees}(L,W) $ given by multiplying $\cW_mL$ by $\hbar^m$ (so an $(m+1)$-vector is multiplied by $\hbar^m$). Under this automorphism, the weight decomposition $\tilde{\cW}_m \widehat{\Rees}(L,W)$ corresponds to $ \prod_{i \ge 0} \hbar^i\cW_{m-i}L \subset L\brh$, so $\tilde{F}^p\widehat{\Rees}(L,W)$ corresponds to   $\prod_{i \ge 0} \hbar^i F^{p-i}L \subset L\brh$.

The derivation $\tilde{\sigma}$ on $\widehat{\Rees}(L,W)$ corresponds to the endomorphism  $\sum_n v_n \hbar^n \mapsto \sum_n (nv_n + \sigma(v_n)) \hbar^n$ of $L\brh$, and  the involution $*$ on $\widehat{\Rees}(L,W)$ simply corresponds to  the involution $(\sum_n v_n \hbar^n)^*:=(\sum_n v_n (-\hbar)^n)  $ of $L\brh$. In particular, $(-)_{\hbar}$ then gives an isomorphism $F^2L \by L\brhh \to \widehat{\Rees}(L,W)^{*=\id}$.
 \end{remark}
 
 \begin{remark}
 Formality results tend to be filtered quasi-isomorphisms $(\breve{L},W) \simeq (L,W)$ of DGLAs, but we will not need formality arguments to establish  quantisation for exact structures.
 \end{remark}

\subsubsection{Quantised homotopical Legendre transformations}\label{qcompequivsn}

Self-dual deformation quantisations of non-degenerate Poisson structures can often be parametrised in terms of power series in de Rham cohomology, via a quantised version of the Legendre transformation. The key is the existence of an $L_{\infty}$-derivation
\[
 \tilde{\mu} = \{\tilde{\mu}^{(i)}\}_{i \ge 1}  \co V\brhh \by (\tilde{F}^2\tilde{L})^{*=\id} \to \tilde{L}^{*=\id}
\]
which is linear in $V\brhh$ and preserves the respective filtrations $\tilde{F}$, with $\tilde{F}^p(V\brh) := \prod_{i \ge 0} \hbar^i F^{p-i}V$ similarly to  Remark \ref{Lbrhrmk} and $ \tilde{F}^p(V\brhh)= V\brhh \cap\tilde{F}^p(V\brh)$. For non-negative shifts, $\tilde{\mu}$ tends to  depend on a choice of even  associator.

The method of \S \ref{compequivsn} can then  be applied to 
the data $(\tilde{L}^{*=\id},V\brhh,\tilde{\sigma},\tilde{\mu})$, with the restriction to self-dual data being crucial in ensuring that  $\tilde{\nu}(\omega, \pi) - \tilde{\sigma}$ is a quasi-isomorphism on $\gr_{\tilde{F}}\tilde{L}^{*=\id}$.
We then have weak equivalences
\[
 \mmc(\tilde{F}^2(V\brhh)) \la  \Comp(\tilde{L}^{*=\id},V\brhh,\tilde{\sigma},\tilde{\mu})^{\nondeg} \to \mmc(\tilde{F}^2\tilde{L}^{*=\id})^{\nondeg},
\]
where $\tilde{F}^2(V\brhh)= \tilde{F}^2(V\brh)\cap V\brhh = F^2V \by \hbar^2V\brhh$ and we write $(-)^{\nondeg}$ for $(-)\by^h_{\mmc(F^2L)}\mmc(F^2L)^{\nondeg}$. On the left, this parametrisation corresponds to a shifted symplectic structure together with an $\hbar^2$-power series in de Rham cohomology, while on the right it corresponds to a self-dual deformation quantisation of a non-degenerate shifted Poisson structure. 

\section{Comparisons for exact symplectic structures}\label{exactsn} 
We now consider Poisson counterparts of exact shifted symplectic structures.

\subsection{Maurer--Cartan characterisation}

\subsubsection{Exact symplectic structures}

 In the setup of \S \ref{legendresn}, a shifted pre-symplectic structure $\omega \in \mc(F^2V)$ is exact if its image in $\mmc(V)$ is in the path component of $0$. We will consider exactness as a structure rather than a property, as follows.
 
 \begin{definition}
Define 
the space of  exact shifted pre-symplectic structures associated to $(V,F)$ as in Setting \ref{dRsetting} to be the homotopy fibre of $\mmc(F^2V) \to \mmc(V)$ over $0$.

We usually take this to be the  realisation $\mmc(\cocone(F^2V \to V))$, but the space $\mmc((V/F^2V)[-1])$ is equivalent.
 \end{definition}

The space $\mmc(\cocone(F^2V \to V))^{\nondeg}$ then consists of elements whose images in $\mmc(F^2V)$ are symplectic; this is a union of path components.

\subsubsection{Poisson counterparts}

Take $(L,F,\sigma)$ as in \S \ref{genpoissstrsn}.

\begin{definition}\label{coconesigmadef}
For $\eps^2=0$, define a DGLA structure on the cochain complex $\cocone(\sigma\eps \co F^2L \to L\eps)$ by setting $[a+b\eps,a'+b\eps]:= [a,b]+[a,b']\eps +[a',b]\eps$. We will refer to its Maurer--Cartan space $\mmc$ as the space of Poisson structures with formal derivation (terminology to be justified in  \S \ref{interpretnsn}). 

Then define $\mmc(\cocone(\sigma\eps \co F^2L \to L\eps))^{\nondeg}$ to be 
\[
\mmc(\cocone(\sigma\eps \co F^2L \to L\eps))\by^h_{\mmc(F^2L)}\mmc(F^2L)^{\nondeg},
\]
the locus of points in $\mmc(\cocone(\sigma\eps \co F^2L \to L\eps))$
consisting of elements whose images in $\mmc(F^2L)$ are non-degenerate Poisson structures.
\end{definition}

The following implies that a Poisson structure $\pi$ admits  formal derivations if and only if 
$[\sigma(\pi)]= 0 \in \H^1(L^{\pi})$
, and that the  $i$th homotopy group of the space of choices is then an $\H^{-i}L^{\pi}$%
-torsor.  
\begin{lemma}\label{PDfibrelemma}
The homotopy fibre of $\mmc\left(\cocone(\sigma\eps \co F^2L \to L\eps)\right) \to \mmc(F^2L)$ over an element $\pi \in \mc(L)$ is weakly equivalent to the fibre of
\[
 N^{-1}(\ldots \xra{\delta_{\pi}} L^{-2}   \xra{\delta_{\pi}} L^{-1} \xra{\delta_{\pi}} L^0) \xra{\delta_{\pi}} L^1
\]
over $-\sigma(\pi)$, where $N^{-1}$ denotes Dold--Kan denormalisation and $\delta_{\pi}:= \delta + [\pi,-]$. 
 \end{lemma}
\begin{proof}
 Since $\mmc$ sends central extensions to Kan fibrations, it suffices to take the fibre, which is the simplicial set given in level $n$ by 
 \[
 \{ D \in (L\hten \Omega^{\bt}(\Delta^n))^0 ~:~ \delta D + [\pi,D] + \sigma(\pi)=0\}.
 \]
 If we decompose $D$ as $D^0 + D^{<0}$, for $D^0 \in  (L^0\hten \Omega^0(\Delta^n))$ and $D^{<0} \in  (L^{<0}\hten \Omega^{>0}(\Delta^n))$, we must have $\delta_{\pi}D^{<0}=0$ and $\delta_{\pi}D^0=-\sigma(\pi) \in L^1$.
  In other words, this is the fibre of 
 \[
  \mmc((L^{\le 0}\eps, \delta_{\pi})[1]) \xra{\delta_{\pi}} \mmc(L^1\eps[1])= L^1
 \]
over $-\sigma(\pi)$. Since $\mmc(V[1])$ is weakly equivalent to $N^{-1}V$ for abelian filtered DGLAs $V$, the result follows. 
\end{proof}

\begin{proposition}\label{exactcompprop}
In the setting of \S \ref{compequivsn}, we have  a natural  weak equivalence
 \[
  \mmc(\cocone(F^2V \to V))^{\nondeg} \simeq \mmc(\cocone(\sigma\eps \co F^2L \to L\eps))^{\nondeg}
 \]
of simplicial sets between the space of exact symplectic structures and the space of Poisson structures with formal derivation.
 \end{proposition}
\begin{proof}
The equivalence $\Comp(L,V,\sigma, \mu )^{\nondeg} \to \mmc(F^2V)^{\nondeg}$ assumed in \S \ref{compequivsn} 
implies that the projection from
\[
 D:=\Comp(L,V,\sigma, \mu )^{\nondeg}\by^h_{\mmc(F^2V)} \mmc(\cocone(F^2V \to V))
\]
to  $\mmc(\cocone(F^2V \to V))^{\nondeg}$ is an equivalence,
so it suffices to construct a weak equivalence from $D$ to $\mmc(\cocone(\sigma\eps \co F^2L \to L\eps))^{\nondeg}$.

Since the morphisms $\cocone(F^2V \to V)/F^p \to F^2V/F^p$ are surjections of abelian DGLAs, applying $\mmc$ gives a Kan fibration, so we can replace the homotopy fibre product in $D$ with a fibre product. 

An element of $D_m$ then consists of data $(\omega, \pi, h, g)$ for $\omega \in \mmc(F^2V)_m$, $\pi \in \mmc(F^2L)_m$, $h$ a homotopy between $\mu(\omega,\pi)$ and $\sigma(\pi)$ in $\Lim_n (F^2(L^{\pi}/F^n)\ten \Omega^{\bt}(\Delta^m))$, and a homotopy $g$ between $\omega$ and $0$ in $\Lim_n ((V/F^n)\ten \Omega^{\bt}(\Delta^m)) $. 

Applying the chain map $\mu(-,\pi) \co V \to L^{\pi}$ then gives us a homotopy $\mu(g,\pi)$ between $\mu(\omega,\pi)$ and $0$, and hence a homotopy $\mu(g,\pi)-h$ between $\sigma(\pi)$ and $0$ in $\Lim_n ((L^{\pi}/F^n)\ten \Omega^{\bt}(\Delta^m))$. We have thus constructed an element $\pi +(\mu(g,\pi)-h)\eps$ of   $\mmc(\cocone(\sigma\eps \co F^2L \to L\eps))_m$, sufficiently naturally to give us a map
\[
\theta \co  D \to \mmc(\cocone(\sigma\eps \co F^2L \to L\eps))
\]
of simplicial sets.

It now suffices to show that this map induces weak equivalences of homotopy fibres over any element $\pi \in \mc(F^2L)$ for which the compatibility map $\mu(-,\pi) \co V \to L^{\pi}$ is a filtered quasi-isomorphism. To do this, we use the Hodge filtration on $V$, proving inductively that the maps
\begin{align*}
D(n) := &\Comp(L,V,\sigma, \mu )^{\nondeg}\by^h_{\mmc(F^2V)} \mmc(\cocone(F^2V \to (V/F^nV)))\\ 
&\xra{\theta(n)} \mmc(\cocone(\sigma\eps \co F^2L \to (L/F^nL)\eps))^{\nondeg}
\end{align*}
are weak equivalences, and passing to the homotopy limit.

The base case $n=0$ holds by hypothesis (\S \ref{compequivsn}), and  obstruction theory for central extensions of DGLAs (see for instance \cite[Appendix \ref{DQDG-towersn}]{DQDG}) gives homotopy fibre sequences
\begin{align*}
 &\mmc\left(F^2V \to (V/F^{n+1}V)\right) \to \mmc\left(F^2V \to (V/F^nV)\right) \to \mmc\left(\gr_F^nV\right), \text{ and}\\
  &\mmc\left(\sigma\eps \co F^2L \to (L/F^{n+1}L)\eps\right)_{\pi} \to  \mmc\left(\sigma\eps \co F^2L \to (L/F^nL)\eps\right)_{\pi} \to \mmc\left(\gr_F^nL^{\pi}\eps\right),
\end{align*}
where $(-)_{\pi}$ denotes homotopy fibre product over $\pi$, and $\mmc(L \to M)$ is shorthand for $\mmc(\cocone(L \to M))$. Since $\mu(-,\pi) \co \mmc(\gr_F^nV) \to \mmc(\gr_F^nL^{\pi}\eps)$ is a weak equivalence (the homotopy groups being $\H^{1-i}(\gr_F^nV)$ and $\H^{1-i}(\gr_F^nL^{\pi})$ respectively), it follows that $\theta(n+1)$ is a  weak equivalence whenever $\theta(n)$ is so. 
Thus  $\theta$ gives a  weak equivalence between the non-degenerate loci. 
\end{proof}

\begin{examples}\label{poissonex} 
Specific settings to which Proposition \ref{exactcompprop} applies are the following.
\begin{enumerate}[wide, labelindent=0pt, itemsep=6pt]
 \item\label{poissonexalg} 
 When  $A$ is an $R$-CDGA, or more generally an $R$-algebra in double complexes\footnote{This is called a stacky CDGA in \cite{poisson} and a graded mixed cdga in \cite{CPTVV}, though the differentials are not mixed.}, which is cofibrant\footnote{It suffices for $A$ to be  a filtered colimit of objects which are quasi-smooth in the original sense of \cite[Lecture 27]{Kon}.},
 we can take $L$ to be the complete filtered DGLA $\widehat{\Pol}(A,n)[n+1]$ of shifted multiderivations as in \cite[Definition \ref{poisson-bipoldef}]{poisson}, given by
 \[
 \prod_{j \ge 0} \hatHHom_A(\Co\Symm_A^j((\Omega^1_{A/R})_{[-n-1]}),A). 
\]
setting $\cW_mL:=\hatHHom_A(\Co\Symm_A^{m+1}((\Omega^1_{A/R})_{[-n-1]}),A)$. This definition extends to  strings of morphisms as in \cite[\S\ref{poisson-Artindiagramsn}]{poisson}. 

Globally, for  any derived Artin $N$-stack $\fX$, we can  take $L[-n-1]$ to be $\widehat{\Pol}(\fX,n):=\oR\Gamma((D_*\fX)_{\et}, \oL\widehat{\Pol}(\sO,n))$ in the sense of \cite[\S 4.1.4]{smallet2}, 
and the same expression behaves well for any homogeneous derived $\infty$-stack with bounded below cotangent complexes, by the commutative analogue of \cite[\S \ref{NCstacks-formalrepsn}]{NCstacks}. 
 
 We then take $V$ to be the complete filtered complex $\DR(A)[n+1]$ or $\DR(\fX):=\oR\Gamma(\fX, \oL\DR(\sO)) \simeq \oR\Gamma((D_*\fX)_{\et}, \oL\DR(\sO))$. The conclusion of \S \ref{compequivsn} applies in this setting, giving an equivalence between the space $\Sp(\fX,n):=\mmc(F^2\DR(\fX)[n+1])^{\nondeg}$ of $n$-shifted symplectic structures and the space $\mmc(F^2\widehat{\Pol}(\fX,n)[n+1])^{\nondeg}$ of non-degenerate $n$-shifted Poisson structures (\cite[Theorem \ref{poisson-Artinthm}]{poisson} or \cite[Theorem 3.2.4]{CPTVV} by a less direct method).
 
The new content of Proposition \ref{exactcompprop} is then that we have an equivalence between the space $\Sp^{\ex}(\fX,n)$ of exact $n$-shifted symplectic structures, given by 
\[
\mmc\left(\cocone(F^2\DR(\fX) \to \DR(\fX))[n+1]\right)^{\nondeg}\simeq \mmc(\DR(\fX)[n]/F^2)^{\nondeg}
 \]
 (whose path components correspond to elements of  $\H^{1+n}\oR\Gamma(\fX, \sO_{\fX} \xra{d} \oL\Omega^1_{\fX})$ with non-degenerate image in $\H^n(\fX, \oL\Omega^2_{\fX})$ under the de Rham differential),
and the space
\[
 \cP^D(\fX,n)^{\nondeg}:= \mmc\left(\cocone(\sigma\eps \co F^2\widehat{\Pol}(\fX,n) \to \widehat{\Pol}(\fX,n)\eps)[n+1]\right)^{\nondeg}
\]
of non-degenerate $n$-shifted Poisson structures with formal derivation, terminology we will justify in \S \ref{interpretnsn}.

For negative shifts, symplectic structures  are canonically exact when working over a field, by \cite[Proposition 3.2]{KinjoParkSafronov} or \cite[Proposition 2.2.11]{HHR1}\footnote{Although stated only for $\Cx$, their argument for existence of the retract $ \tau^{\le 0}(\DR/F^2) \to \tau^{\le 0}\DR$ of the canonical map is valid over any characteristic $0$ field, as is the surjectivity of $\tau^{\le 0}(\DR/F^2) \to \tau^{\le 0}(F^2\DR[1])$, since \cite{emmanouil} works in that generality, making the algebraically closed hypothesis in \cite[Proposition 5.6a]{BBBJdarboux} redundant.}.
Explicitly, they give a canonical quasi-isomorphism $ \tau^{\le 0}(\DR/F^2) \simeq \tau^{\le 0}(F^2\DR[1]) \oplus \H^0\DR$ of sheaves, so 
\begin{align*}
 \Sp^{\ex}(\fX,-1) \simeq \Sp(\fX,-1) \by \H^0_{\dR}(\fX), \quad \text{and}\\
 \Sp^{\ex}(\fX,n) \simeq \Sp(\fX,n ) \quad \text{for } n<-1.
\end{align*}

 \item\label{poissonexLag} We can also apply Proposition \ref{exactcompprop} to the equivalence between shifted Lagrangians and non-degenerate shifted co-isotropic structures. Associated to a morphism $f:A \to B$ of (stacky) CDGAs, we take $L$ to be the complete filtered DGLA $\widehat{\Pol}(A,B;0):=Q\widehat{\Pol}(A,B;0)[1]/G^1$ from \cite[Definition \ref{DQLag-QPoldef}]{DQLag} in the $0$-shifted case, generalising to the filtered DGLA $\Pol(f,n)[n+1]$ of \cite[\S 4.2]{MelaniSafronovI}. Taking derived global sections as in (\ref{poissonexalg}), we also obtain a complete filtered DGLA $\widehat{\Pol}(\fY,\fX;0)[n+1]$ for any morphism $f \co \fX \to \fY$ of homogeneous derived $\infty$-stacks with bounded below cotangent complexes. The space $\cP(\fX,\fY;n)$ of co-isotropic structures on $f$ can then be defined as $\mmc(F^2\widehat{\Pol}(\fY,\fX;n)[n+1])$; there is a canonical map $ \cP(\fX,\fY;n) \to \cP(\fY,n)$ whose homotopy fibre over a shifted Poisson structure $\pi$ on $\fY$ gives the space of shifted co-isotropic  (with respect to $\pi$) structures on $\fX$ over $\fY$.
 
 The filtered complex $V$ in this setting is simply $\cocone(\DR(\fY) \to \DR(\fX))[n+1]$, with  \cite[Proposition \ref{DQLag-prop3}]{DQLag} (in the $0$-shifted case) and \cite[Theorem 4.22]{MelaniSafronovII} giving an equivalence between the space $\cP(\fY,\fX;n)^{\nondeg}$ of Lagrangian co-isotropic structures and the space $\Lag(\fY,\fX;n):=\mmc(F^2V[n+1])$ of $n$-shifted Lagrangian structures.
 
 The new content of Proposition \ref{exactcompprop} is then that we have an equivalence between the space 
\begin{align*}
\Lag^{\ex}(\fY,\fX;n):= &\mmc(\Tot(F^2\DR(\fY) \to \DR(\fY) \oplus F^2\DR(\fX) \to \DR(\fX))[n+1])^{\nondeg}\\
&\simeq \mmc((\cocone(\DR(\fY) \to \DR(\fX))[n]/F^2))^{\nondeg}
 \end{align*}
of exact $n$-shifted Lagrangian structures 
and the space
\[
 \cP^D(\fY,\fX;n)^{\nondeg}:= \mmc\left(\cocone(\sigma\eps \co F^2\widehat{\Pol}(\fY,\fX;n) \to \widehat{\Pol}(\fY,\fX;n)\eps)[n+1]\right)^{\nondeg}
\]
of non-degenerate $n$-shifted co-isotropic structures with formal derivation. 

Note that the natural map $\Lag^{\ex}(\fY,\fX;n) \to \Lag(\fY,\fX;n)$ from the space of exact Lagrangians to the space of all Lagrangians will be an equivalence whenever the morphism $\fX \to \fY$ induces an (unfiltered) quasi-isomorphism $\DR(\fY) \to \DR(\fX)$. Examples of this form arise whenever $\fY$ is a formal thickening of $\fX$, such as a relative de Rham stack $(\fX/\fS)_{\dR}$ or the functor $B \mapsto \fX(B^0)\by^h_{\fY(B^0)}D_*\fY(B)$ on stacky CDGAs. 

Over a field, the map $\Lag^{\ex}(\fY,\fX;n) \to \Lag(\fY,\fX;n)$ will also be an equivalence for $n<-1$ as in the symplectic case, and $\Lag^{\ex}(\fY,\fX;-1) \simeq \Lag(\fY,\fX;-1)\by \ker(\H^0_{\dR}(\fY) \to \H^0_{\dR}(\fX))$.

 \item\label{poissonexan} As in \cite[\S 4.4]{DStein}, there are analogues of (\ref{poissonexalg}) and (\ref{poissonexLag}) in analytic and $\C^{\infty}$ settings, working with dg EFC algebras and dg $\C^{\infty}$-algebras, and the corresponding cotangent complexes, in place of abstract CDGAs. The statements about exact shifted symplectic and Lagrangian structures then carry over to derived analytic stacks and derived $\C^{\infty}$ stacks in exactly the same way, substituting in the relevant notions of differential forms and multiderivations. 
 
The retraction argument from \cite[Proposition 3.2]{KinjoParkSafronov} or \cite[Proposition 2.2.11]{HHR1} 
still holds in these settings, giving $\Sp^{\ex}(-,n) \simeq \Sp(-,n)$ for $n<-1$.  We may compute $\DR(\fX)$ by completing $\DR(\fX)$ along the underived truncation $\pi^0\fX$ by the argument of \cite{FeiginTsygan,bhattDerivedDR}, giving the equivalence of derived de Rham cohomology and that of the completed complex $\hat{\Omega}^{\bt}_{\fX}$. In the $\Cx$-analytic setting, the latter is equivalent to  Betti cohomology as in  \cite[Theorem IV.1.1]{hartshorneDRCoho}.
That  supplies the surjectivity part of the argument from \cite[Proposition 2.2.11]{HHR1}, so  $\Sp^{\ex}(\fX,-1) \simeq \Sp(\fX,-1) \by \H^0(\fX,\bK)$. The same expression holds whenever the  de Rham theorem holds for derived de Rham cohomology; in the $\C^{\infty}$ setting, this is a more restricted class of spaces as in \cite[\S 5]{taroyanDRDerivedDG}.




\item There are also non-commutative analogues of (\ref{poissonexalg}) and (\ref{poissonexLag}), as well as of the analytic case of (\ref{poissonexan}). For these, we take $L$ to be the complete filtered DGLA $\widehat{\Pol}^{nc}_{\cyc}(A,n)[n+1]$ of $n$-shifted cyclic multiderivations  on a (stacky) $R$-DGAA $A$, as defined in \cite[Definitions \ref{NCpoisson-polcycdef} and \ref{NCpoisson-bipolcycdef}]{NCpoisson} in the NC algebraic setting,  and on a stacky dg FEFC $R$-algebra as in \cite[Definition \ref{analyticNC-anbipolcycdef}]{analyticNC} in the NC analytic setting.

As in \cite[Definition \ref{NCpoisson-cPhgsdef}]{NCpoisson} or   \cite[Definition \ref{analyticNC-PreBiSphgsdef}]{analyticNC}, this extends to $\widehat{\Pol}^{nc}_{\cyc}(\fX,n):= \oR\Gamma((D_*\fX)_{\rig}, \widehat{\Pol}^{nc}_{\cyc}(-,n))$ for any derived NC prestack (algebraic or analytic)   $\fX$ which is homogeneous and has bounded below cotangent complexes. The  space $\cD\cP(\fX,n)$  of $n$-shifted double Poisson structures on $\fX$ is then $\mmc(F^2\widehat{\Pol}^{nc}_{\cyc}(\fX,n))$. Shifted double co-isotropic structures can be formulated similarly.

In these settings, we take $V[-1-n]$ to be the   cyclic de Rham complex $\DR_{\cyc}(A)$ equipped with its Hodge filtration, or its derived global sections  $\DR_{\cyc}(\fX)$ over a homogeneous derived  NC prestack $\fX$. The space $\BiSp(\fX,n)$ of $n$-shifted bisymplectic structures on $\fX$ is then $\mmc(F^2\DR_{\cyc}(\fX)[n+1]^{\nondeg}) $ as in \cite[Definition  \ref{NCpoisson-PreBiSphgsdef}]{NCpoisson}, with the conclusion of \S \ref{compequivsn} applying in this setting to give the equivalence $\BiSp(\fX,n)  \simeq \cD\cP(\fX,n)^{\nondeg}$ of \cite[Theorem \ref{NCpoisson-Artinthm}]{NCpoisson}. 

The new content of Proposition \ref{exactcompprop} (with analogous analytic and bi-Lagrangian statements) is then that we have an equivalence between the space 
\begin{align*}
\BiSp^{\ex}(\fX,n):= &\mmc(\cocone(F^2\DR_{\cyc}(\fX) \to \DR_{\cyc}(\fX))[n+1])^{\nondeg}\\
&\simeq \mmc(\DR_{\cyc}(\fX)[n]/F^2)^{\nondeg}
 \end{align*}
of exact $n$-shifted bisymplectic structures (whose path components correspond to elements of  $\H^{1+n}\oR\Gamma(X, \oL\sO_X/[\sO_X,\sO_X] \xra{d} \oL\Omega^1_X\ten^{\oL}_{\sO_X^e}\sO_X)$ with non-degenerate image in $\H^n(\fX, \oL\Omega^2_{\fX,\cyc})$ under the Karoubi--de Rham differential),
and the space
\[
 \cD\cP^D(\fX,n)^{\nondeg}:= \mmc(\cocone(\sigma\eps \co F^2\widehat{\Pol}_{\cyc}(\fX,n) \to \widehat{\Pol}_{\cyc}(\fX,n)\eps)[n+1])^{\nondeg}
\]
of non-degenerate $n$-shifted double Poisson structures with formal derivation. 
\end{enumerate}
\end{examples}

\subsection{Deformation quantisations} 

Replacing $L$ with $\tilde{L}$, we can consider the DGLA
 $\cocone(\tilde{\sigma}\eps \co \tilde{F}^2\tilde{L} \to \tilde{L}\eps)$ and its sub-DGLA of fixed points for the involution $*$. We respectively regard their Maurer--Cartan spaces $\mmc$ as the spaces of deformation quantisations with $\hbar\pd_{\hbar}$-derivations and its subspace of self-dual quantisations.
 
 This leads to the following elementary quantisation result for non-degenerate Poisson structures with exact data. 
 
 \begin{corollary}\label{exactquantcor}
  In the setting of \S \ref{qcompequivsn}, 
  the canonical map 
  \[
\mmc(  \cocone(\tilde{\sigma}\eps \co \tilde{F}^2\tilde{L} \to \tilde{L}\eps)^{*=\id})^{\nondeg}  \to \mmc(\cocone(\sigma\eps \co F^2L \to L\eps))^{\nondeg},
  \]
  induced by the filtered quasi-isomorphism $\tilde{L}/\hbar\tilde{L} \simeq L$, is a weak equivalence from the space of self-dual non-degenerate deformation quantisations with $\hbar\pd_{\hbar}$-derivation to the space of non-degenerate Poisson structures with formal derivation.
 \end{corollary}
\begin{proof}
By applying Proposition \ref{exactcompprop} to the data $(L,V,\sigma,\mu)$ and $(\tilde{L}^{*=\id},V\brhh,\tilde{\sigma},\tilde{\mu})$, this reduces to showing that the map
\[
\mmc( \cocone(\tilde{F}^2(V\brhh) \to V\brhh))^{\nondeg} \xra{\hbar^2 \mapsto 0} \mmc(\cocone(F^2V \to V))^{\nondeg}
\]
is a weak equivalence. Since $\tilde{F}^2(V\brhh) = F^2V \by \hbar^2V\brhh$ and is abelian, the space on the left decomposes as 
\[
 \mmc(\cocone(F^2V \to V))^{\nondeg} \by \prod_{i\ge 1} \mmc(\hbar^{2i}\cocone(V \to V)).
\]
The complex $\cocone(V \to V)$ is acyclic, giving the required weak equivalence.
\end{proof}

 \begin{remark}
  This quantisation result is much stronger than those for non-degenerate Poisson structures without additional data, since the quantisation is essentially unique, instead of having a torsor of choices parametrised by Poisson cohomology. Although the data of \S \ref{qcompequivsn} are used to establish equivalence, the canonical map does not depend on them. 
 \end{remark}

 Applying Corollary \ref{exactquantcor} to both $\tilde{L}$ and the DGLA $L\brh$ of Remark \ref{Lbrhrmk} gives:
 \begin{corollary}\label{exactquantbrhcor}
  In the setup of \S \ref{quantsn}, 
  there is a natural weak equivalence 
  \begin{align*}
&\mmc(  \cocone(\tilde{\sigma}\eps \co \tilde{F}^2\tilde{L} \to \tilde{L}\eps)^{*=\id})^{\nondeg}\\
&\simeq  \mmc(\cocone((\sigma+ \hbar\pd_{\hbar}) \eps \co F^2L\by \hbar^2L\brhh \to L\brhh\eps))^{\nondeg}.
  \end{align*}
 \end{corollary}

 \begin{examples}\label{quantex}
 We can apply Corollary \ref{exactquantcor} to the following specific settings, quantising non-degenerate Poisson structures with derivation as considered in Examples \ref{poissonex}.  See \S \ref{interpretDQsn} for justification of the algebraic interpretations given of Maurer--Cartan elements.
 
\begin{enumerate}[wide, labelindent=0pt, itemsep=6pt]

\item\label{quantex0} For $0$-shifted quantisations as  established    in terms of DQ algebroid deformations of the structure sheaf in \cite{DQnonneg,DQpoisson} for derived Artin stacks  $\fX$ (algebraic, analytic or $\C^{\infty}$) with perfect cotangent complexes, we take $(\breve{L},W)$ to be the filtered  DGLA  given by $W_m\breve{L}:= \oR\Gamma((D_*\fX)_{\et}, \hatTot \tau^{\HH}_{m+1} \oL\cD^{\poly}_{\oplus}(\sO))[1]$, where $\cD^{\poly}_{\oplus}$ is the brace algebra of polydifferential operators, $\hatTot$ the sum-product total complex  and $\tau^{\HH}$ good truncation in the Hochschild direction. Equivalent  variants replace filtration $\tau^{\HH}$  with the quasi-isomorphic filtration $\gamma$ of \cite[Definition \ref{DQLag-HHdef0}]{DQLag}, and in algebraic cases replace polydifferential operators with the cohomological Hochschild complex. The involution $i$ on $\cD^{\poly}_{\oplus}(A)$ is given up to Koszul sign by reversing the order of inputs.

Then $\tilde{L}:=\widehat{\Rees}(\breve{L}, W)$ is the complete filtered DGLA
$Q\widehat{\Pol}(\fX,0)[1]:= \oR\Gamma((D_*\fX)_{\et}, \oL Q\widehat{\Pol}(\sO,0))[1]$
of quantised polyvectors from \cite[Definition \ref{DQpoisson-qpoldef2}]{DQpoisson}.

Corollary \ref{exactquantcor} then gives an equivalence between the space $\cP^D(\fX,0)^{\nondeg}$ of non-degenerate $0$-shifted Poisson structures with formal derivation from Examples \ref{poissonex}.(\ref{poissonexalg}) and the space 
\[
 Q\cP^D(\fX,0)^{\nondeg,sd}:= \mmc( \cocone( \eps\hbar \pd_{\hbar} \co \tilde{F}^2Q\widehat{\Pol}(\fX,0) \to Q\widehat{\Pol}(\fX,0) \eps)^{*=\id}[1])^{\nondeg}
\]
of self-dual (or involutive) deformation quantisations with $\hbar\pd_{\hbar}$-derivation, terminology we will justify in \S \ref{interpretDQsn}; these roughly correspond to almost commutative (or $BD_1$) associative algebroid deformations of the structure sheaf, equipped with a connection in the $\hbar$ direction with simple poles. 

In this case, \cite[Theorem \ref{DQpoisson-fildefhochthm1}]{DQpoisson} associates to any even associator an equivalence $\gr^W\breve{L} \simeq \breve{L}$ of filtered involutive DGLAs, and hence,  as in Remark \ref{Lbrhrmk}, an equivalence $\tilde{L} \simeq (\gr^W\breve{L})\brh$ between quantised polyvectors and power series in polyvectors. The only new content of  Corollary \ref{exactquantbrhcor} in this case is that the equivalence of \cite{DQnonneg} is in fact independent of any choice of associator when the structure is non-degenerate and exact.

When $\fX$ is an analytic or smooth manifold, objects of $Q\cP^D(\fX,0)$ are just given  locally by pairs $(\sA,D)$ for $\sA$ an associative deformation of $\sO_{\fX}$ and $D$ an $\hbar\pd_{\hbar}$-derivation on  $\sA$. Isomorphisms $(\sA,D) \to (\sA',D')$ are pairs $(\phi, \gamma)$ for $\phi \co \sA \to \sA'$ an infinitesimal isomorphism  and $\gamma \in \hbar^{-1}\sA$ satisfying $\phi^{-1} \circ  D' \circ \phi - D  =[\gamma,-]$. A $2$-isomorphism from $(\phi,\gamma)$ to $(\psi, \eta)$ is then an element $u \in 1+ \hbar\sA$ satisfying $\psi(a) = \phi (uau^{-1}) $ and  $\eta = u^{-1}\gamma u + u^{-1}D(u)$. 

\item\label{quantexpos} For $n>0$, we can similarly  take $\breve{L}[1-n]$ to be a higher ($E_{n+1}$) Hochschild complex, with $W_m := \tau^{\HH}_{(m+1)(n+1)} $.  Then the complete filtered DGLA $\widehat{\Rees}(\breve{L}, W)$ can be regarded as a  DGLA $Q\widehat{\Pol}(A,n)[n+1]$ of $n$-shifted quantised polyvectors.

Its Maurer--Cartan elements 
similarly correspond to  $BD_{n+1}$-algebra quantisations of $n$-shifted Poisson structures as in \cite[\S 3.5.2]{CPTVV}. Beware that our formulation in terms of $\tilde{F}^2Q\widehat{\Pol}(A,n)[n+1]$ is a weak, rather than strict, quantisation, generalising the notion of DQ algebroid in the $n=0$ case.  Strict quantisations can be formulated by replacing $\tilde{F}^2Q\widehat{\Pol}(A,n)$ with $\ker(\tilde{F}^2Q\widehat{\Pol}(A,n) \to \hbar A\brh)$, which eliminates curvature or inner automorphisms.

Since any choice of associator gives a filtered formality equivalence $E_{n+1} \simeq P_{n+1}$ and hence  
$BD_{n+1} \simeq P_{n+1}\brh$,   such deformation quantisations exist strictly and without restriction on the cotangent complex. In our terms, formality implies $\breve{L} \simeq \gr^W\breve{L}$ and hence  $Q\widehat{\Pol}(A,n) \simeq  \widehat{\Pol}(A,n)\brh$,  intertwining the respective involutions if the associator is even. Unlike the $0$-shifted case, the quasi-isomorphisms $\breve{L} \simeq \gr^W\breve{L}$ respect the augmentation maps to $A[n+1]$, which is why strict quantisation is always  possible for $n>0$. 

Again, Corollary \ref{exactquantcor} gives an equivalence  $Q\cP^D(\fX,n)^{sd,\nondeg}\simeq \cP^D(\fX,n)^{\nondeg}$ between the respective structures with formal derivation, and the only new content of Corollary \ref{exactquantbrhcor} is that the parametrisation is  in fact independent of any choice of associator in non-degenerate exact cases. 

\item\label{quantexcoiso} There are similar formulations and consequences for positively  shifted co-isotropic structures in terms of the Swiss cheese operad as in \cite{MelaniSafronovII}, with a quantisation consisting locally of a $BD_{n+1}$-algebra acting on a $BD_n$-algebra. For $n>1$, they show deformation quantisation of $n$-shifted co-isotropic structures as a consequence of formality, so exactly the same considerations apply as for (\ref{quantexpos}). For $n=1$, deformation quantisation is a consequence of  \cite[Theorem \ref{DQpoisson-fildefhochthm1}]{DQpoisson} as in \cite[\S \ref{DQpoisson-coisosn2}]{DQpoisson}, so the considerations of (\ref{quantex0}) apply instead.

Thus exact $n$-shifted Lagrangians admit essentially unique self-dual quantisations with $\hbar\pd_{\hbar}$-derivation for $n\ge 1$.

\item \label{quantexvanish}  $(-1)$-shifted deformation quantisations are formulated in \cite{DQvanish} for line bundles $\sL$ on derived Artin stacks, and established there for non-degenerate quantisations when there is a right $\sD$-module structure on $\sL^{\ten 2}$ (which can be thought of as $\sL$ being a spin structure or module of half-densities). As explained in \cite[\S 4.4]{DStein}, those results extend verbatim to analytic and $\C^{\infty}$ settings, with the $\C^{\infty}$ case spelt out in \cite[\S \ref{DQDG-quantneg1sn}]{DQDG}. 

In these cases, 
 we take $(\breve{L},W)$ to be the filtered  DGLA  given by $W_m\breve{L}:=\oR\Gamma((D_*\fX)_{\et},\hatTot \oL F_{m+1} \cD(\sL))$, for the order filtration $F$  on the  associative algebra $\cD(\sL)$ of differential operators on $\sL$. Then as in \cite[Definition \ref{DQvanish-qpoldef} and \S \ref{DQvanish-biquantsn}]{DQvanish}, $\tilde{L}:= \widehat{\Rees}(\breve{L}, W)$ is the DGLA 
$Q\widehat{\Pol}(\sL,-1)$
of quantised polyvectors. When $\sL^{\ten 2}$ carries a right $\sD$-module structure, there is a natural anti-involution $i \co \cD(\sL)\simeq \cD(\sL)^{\op}$, which extends to an involution $*$ of $Q\widehat{\Pol}(\sL,-1)$ by also reversing the sign of $\hbar$. 

There is also an obvious generalisation given fixing just a line bundle $\sM$ with a  right $\sD$-module structure and considering $Q\widehat{\Pol}(-,-1)$ as a functor on the $\mu_2$-gerbe of square roots of $\sM$.

Corollary \ref{exactquantcor} then gives an equivalence between the space $\cP^D(\fX,-1)^{\nondeg}$ of non-degenerate $(-1)$-shifted Poisson structures with formal derivation from Examples \ref{poissonex}.(\ref{poissonexalg}) and the space $Q\cP^D(\sL,-1)^{sd,\nondeg}$, defined to be
\[
 \mmc( \cocone( \eps\hbar \pd_{\hbar} \co \tilde{F}^2Q\widehat{\Pol}(\sL,-1) \to Q\widehat{\Pol}(\sL,-1) \eps)^{*=\id})^{\nondeg},
\]
of self-dual  non-degenerate deformation quantisations with $\hbar\pd_{\hbar}$-derivation, terminology we will justify in \S\S \ref{interpretDQsn} and \ref{vanishsn}; 
these roughly correspond to  $BD_0$-algebra (i.e. almost commutative $BV$-algebra) deformations of the structure sheaf twisted by $\sL$, equipped with a connection in the $\hbar$ direction with simple poles. We will see that the connection corresponds to the monodromy operator on the vanishing cycles sheaf.

Deformation quantisation is  established in \cite{DQvanish} for non-degenerate structures,  but Corollary \ref{exactquantcor}  goes further in showing that each exact $(-1)$-shifted symplectic structure has a unique quantisation  admitting an $\hbar\pd_{\hbar}$ derivation, whose existence cannot easily be deduced from the method of \cite{DQvanish}. Since all $(-1)$-shifted symplectic structures are canonically exact by \cite[Proposition 3.2]{KinjoParkSafronov} or \cite[Proposition 2.2.11]{HHR1}, this distinguished choice and additional structure always exist.

\item \label{0shiftedLagex}
Similar considerations to those for $(-1)$-shifted symplectic structures apply to $0$-shifted co-isotropic structures as in \cite{DQLag}, where quantised polyvectors are formulated in terms of a Hochschild complex acting on a ring of differential operators.  The only difference is that the distinguished choice and additional structure only exists when the  $0$-shifted symplectic structure on the base is  exact, which is only guaranteed  in a formal neighbourhood of the Lagrangian. 


\end{enumerate}
\end{examples}
 
\begin{remark}[Limitations]
\begin{enumerate}[wide, labelindent=0pt, itemsep=6pt]
 \item 
There is a principle, originally announced by Costello and Rozenblyum and now proved in \cite{tomicShiftedLagThick}, that every shifted Poisson structure arises as a formal Lagrangian, and every such Lagrangian has an exact form as in Example \ref{poissonex}.(\ref{poissonexLag}). However,  we cannot use this to infer the existence of deformation quantisations for  degenerate $0$-shifted Poisson structures, because our quantisations in Example \ref{quantex}.(\ref{0shiftedLagex}) 
incorporate curvature. Specifically, a quantisation of a Lagrangian morphism $\fX \to \fY$ only gives rise to a (possibly curved)  quantisation of the Poisson structure on $\fX$ when the quantisation of $\fY$ is uncurved, because curvature on $\fY$ is effectively incorporated in  the Maurer--Cartan equation for $\fX$. 

 Indeed, given an $n$-shifted Poisson structure $\pi \in \mmc(\widehat{\Pol}(A,n)[n+1])$, the ambient formal  $(n+1)$-shifted symplectic object roughly corresponds to  $P:= (\widehat{\Pol}(A,n),\delta + [\pi,-])$ with its canonical Poisson structure, and then we can describe the essentially  unique self-dual quantised Lagrangian structure with formal derivation as follows. Writing $Q\widehat{\Pol}$ for   $Q\widehat{\Pol}(A,n)$ when $n\ge 0$ and $Q\widehat{\Pol}(\sL,-1)$ when $n=-1$, pick any element $\Delta \in \tilde{F}^2Q\widehat{\Pol}^{n+2}$ lifting $\pi$. 
 The curvature $\kappa := \delta \Delta + \half [\Delta,\Delta]$ lies in $\ker((\tilde{F}^2 Q\widehat{\Pol})^{*=\id} \to F^2\widehat{\Pol})$, i.e. in $(\hbar^2Q\widehat{\Pol})^{*=\id}$.
 We can then form a curved self-dual quantisation $\tilde{P}:=(\hbar Q\widehat{\Pol}, \delta + [\Delta,-], \kappa)$ of $P$, noting that $\kappa \in \hbar \tilde{P}$, and equip it with the curved $\hbar\pd_{\hbar}$-derivation $(\hbar\pd_{\hbar}, \hbar\pd_{\hbar}(\Delta))$, where $ \hbar\pd_{\hbar}(\Delta) \in \hbar^{-1}\tilde{P}$. 
 The quantised Lagrangian structure is then given  by the identity morphism $\tilde{P} \to \hbar Q\widehat{\Pol}$ with curvature $\Delta$.


\item 
Deformation quantisations  of $(-2)$-shifted  symplectic structures are formulated 
in \cite{DQ-2}. In this case, for any flat right connection $\nabla$ on a CDGA $A$, we take $\breve{L}$ to be the right de Rham complex $\DR^r(A,\nabla)[-1]$ of $A$  associated to $\nabla$, and set $W_m\breve{L}:= F_{p+1}\DR^r(A,\nabla)[-1]$ for the  (increasing) Hodge filtration $F$ on $\DR^r$. Since $\DR^r(A,\nabla)$ is a BV algebra, $\breve{L}$ is a DGLA, and then $\tilde{L}:=\widehat{\Rees}(\breve{L},W)$ is the DGLA $Q\widehat{\Pol}(A,\nabla,-2)[-1]$ of  $(-2)$-shifted polyvectors from \cite{DQ-2}.

Since  $\DR^r(A,\nabla)$ is a BV algebra, $\hbar Q\widehat{\Pol}(A,\nabla,-2)$ is a $BD_0$-algebra. The  Maurer--Cartan condition for an element $S$ of $Q\widehat{\Pol}(A,\nabla,-2)^0$ is then equivalent to the quantum master equation $(\delta + \nabla)(e^S)=0$, for $e^S \in 1+\hbar\DR^(A,\nabla)$, giving rise to a cohomology class $[e^S] \in \H^0\DR^r(A,\nabla)\brh$.
 
 However, in this setting there is no involution $*$ to give a notion of self-duality, and instead the potential  first order obstruction to quantisation is eliminated in \cite{DQ-2} by allowing $\nabla$ to vary, each non-degenerate Poisson structure having an essentially unique connection permitting quantisation. 
 
As in Examples \ref{poissonex}, every non-degenerate $(-2)$-shifted Poisson structure  admits an essentially unique formal derivation, but we cannot use   Corollary \ref{exactquantcor} to lift the formal derivation to the quantisation. Indeed, such a formal derivation would be undesirable, because it would give an element $D \in Q\widehat{\Pol}(A,\nabla,-2)^{-1}\subset \hbar^{-1}\DR^r(A,\nabla)^{-1}\brh$ with $\hbar\pd_{\hbar}e^S= (\delta +\nabla)(e^S D)$, implying that the virtual fundamental class $[e^S]$ equals $[1]$.
\end{enumerate}

\end{remark}

%
%
%
\subsection{Eliminating formal derivations}\label{DWLsn1} 
 
In $\Cx$-analytic and some $\C^{\infty}$ settings, there are local simplifications arising from the Poincar\'e lemma which make it possible to describe shifted symplectic structures and their Poisson and quantised counterparts in terms of exact shifted symplectic structures. We now summarise the method.

Writing $(-)^{\nondeg}$ for $(-)\by^h_{\mmc(F^2L)}\mmc(F^2L)^{\nondeg}$, we have the following.
\begin{proposition}\label{DWLprop} 
 If in the setup \S \ref{qcompequivsn}  we have an inclusion $ \iota\co U \into L$ of a central subalgebra, lifting to a morphism $\tilde{\iota} \co U \to \tilde{L}^{*=\id}$ with central image, then the
 canonical map 
  \begin{align*}
&\mmc(  \cocone(\tilde{\sigma}\eps \co \tilde{F}^2\tilde{L}^{*=\id} \to (\tilde{L}^{*=\id}/\tilde{\iota}U)\eps))^{\nondeg}\\
&\to \mmc(\cocone(\sigma\eps \co F^2L \to (L/\iota U)\eps))^{\nondeg},
  \end{align*}
  induced by the filtered quasi-isomorphism $\tilde{L}/\hbar\tilde{L} \simeq L$, is a weak equivalence. 
  
If there is moreover a quasi-isomorphism $j \co U \to V$, with $\iota  = \mu(-,\pi) \circ j$ for all $\pi$, then  there is a canonical weak equivalence
\[
\rho_{\tilde{L}} \co  \mmc( \tilde{F}^2\tilde{L}^{*=\id})^{\nondeg} \simeq \mmc(F^2L)^{\nondeg} \by \mmc(\hbar^2U\brhh),
\] 
independent of the quantised compatibility map $\tilde{\mu}$.

\end{proposition}
\begin{proof}
 The first statement follows by modifying the proof of Corollary \ref{exactquantcor}, since we still have homotopy fibre sequences
 \begin{align*}
&\mmc(  \cocone(\tilde{\sigma}\eps \co \tilde{F}^2(\tilde{L}/\hbar^{2n+2})^{*=\id} \to \tilde{L}^{*=\id}/(\tilde{\iota }U +(\hbar^{2n+2}))\eps))^{\nondeg}\\
&\to   \mmc(  \cocone(\tilde{\sigma}\eps \co \tilde{F}^2(\tilde{L}/\hbar^{2n})^{*=\id} \to \tilde{L}^{*=\id}/(\tilde{\iota }U +(\hbar^{2n}))\eps))^{\nondeg}\\
&\phantom{\to} \to \mmc(\hbar^{2n}\cocone(V \to V)[1])\simeq 0
 \end{align*}
for $n \ge 1$. 
 
 For the second statement, writing $(-)^{\nondeg}$ for $(-)\by^h_{\mmc(F^2L)}\mmc(F^2L)^{\nondeg}$, observe that the proof of Proposition \ref{exactcompprop} adapts to give an equivalence 
 \[
 \mmc(\cocone(F^2V \to (V/U)))^{\nondeg} \simeq \mmc(\cocone(\sigma\eps \co F^2L \to (L/\iota U)\eps))^{\nondeg},
 \]
 with the former equivalent to $\mmc(F^2V)^{\nondeg}$ (since $V/U$ is acyclic), and hence by hypothesis from \S \ref{compequivsn}, to $\mmc(F^2L)^{\nondeg}$. We thus have an equivalence  $\mmc(\cocone(\sigma\eps \co F^2L \to (L/\iota U)\eps))^{\nondeg} \to\mmc(F^2L)^{\nondeg}$. 
%

Combined with the first statement, this implies that the map marked $f_2$ in the commutative diagram below induces an equivalence on applying $\mmc(-)^{\nondeg}$.
\[
 \xymatrix{
\cocone( \tilde{F^2}\tilde{L}^{*=\id} \xra{\tilde{\sigma}\eps}  (\tilde{L}^{*=\id}/\tilde{\iota}U)\eps)  \ar[d]_{f_1} \ar[dr] \ar[drr]^{f_2}\\
 \tilde{F^2}\tilde{L}^{*=\id} \ar[r]^-{f_3} & \tilde{F}^2\tilde{L}^{*=\id}/\hbar^2\tilde{\iota}U\brhh \ar[r]_-{f_4} &  F^2L.}
\]
The proof of Proposition \ref{exactcompprop} also adapts to show that the map marked $f_4$ gives an  an equivalence on applying $\mmc(-)^{\nondeg}$, since filtering $\tilde{F}^2\tilde{L}^{*=\id}/\hbar^2\tilde{\iota}U\brhh$ by powers of $\hbar^2$ leads us to consider  graded pieces $\hbar^{2n}(L^{\pi}/\iota U)$ for $n \ge 1$, which are acyclic when $\pi$ is non-degenerate.

Since the bottom row $f_4 \circ f_3$ of the diagram is a central extension, it gives a homotopy fibre sequence
\[
  \mmc( \tilde{F}^2\tilde{L}^{*=\id})^{\nondeg} \to \mmc(F^2L)^{\nondeg}\to \mmc(\hbar^2U\brhh[1]).
\]
Because $f_2$ and $f_4$ are  equivalences, we have a canonical  homotopy section $f_1 \circ f_2^{-1}$ of $f_4 \circ f_3$, so the equivalence 
\[
 \mmc( \tilde{F}^2\tilde{L}^{*=\id})^{\nondeg} \simeq \mmc(F^2L)^{\nondeg}\by \mmc(\hbar^2U\brhh)
\]
follows because $\hbar^2U\brhh$ is the loop space of  $\hbar^2U\brhh[1]$. 
\end{proof}

 \begin{lemma}\label{DWLparamlemma}
 In the setting of Proposition \ref{DWLprop}, assume that  $\tilde{\iota} U \subset \hbar^{-1}\cW_{-1}\breve{L} $ and that  the map $\tilde{\mu}$ of \ref{qcompequivsn} also satisfies $\tilde{\mu}(\omega, \pi + \tilde{\iota} u)= \tilde{\mu}(\omega, \pi)$ for all $\omega \in V\brh$,  $\pi \in \tilde{L}$ and $u \in U\brh$.  
 
 Then if $\psi \co \mmc( \tilde{F}^2\tilde{L}^{*=\id})^{\nondeg} \simeq \mmc(F^2V \by \hbar^2V\brhh)^{\nondeg}$ denotes the equivalence of \S \ref{qcompequivsn},
  we have
\[
 \psi(\rho_{\tilde{L}}^{-1}( \pi,\hbar^2 u)) \simeq  \psi(\rho_{\tilde{L}}^{-1}( \pi,0)) + j\frac{\pd (\hbar u)}{\pd \hbar} \hbar^2.
\]
\end{lemma}
\begin{proof}
Writing $\varpi:= \rho_{\tilde{L}}^{-1}( \pi,0)$,   by construction    we have $\rho_{\tilde{L}}^{-1}( \pi,\hbar^2u) \simeq \varpi + \hbar^2\tilde{\iota}  u$, since the parametrisation is defined by the action of the central subalgebra $\hbar^2\tilde{\iota} U\brhh$. 
Because $\tilde{\iota} U \subset \hbar^{-1}\cW_{-1}\breve{L}$, we have $\tilde{\sigma}(\hbar \tilde{\iota}u) = (\hbar\tilde{\iota}) \hbar \frac{\pd u}{\pd \hbar}= 
\tilde{\iota}\frac{\pd u}{\pd \hbar} \hbar^2$ for $u \in U \brh$. 

The map $\psi$ is defined by the property that  $\tilde{\mu}(\psi(a),a) \simeq \tilde{\sigma}(a)$, so we wish to show 
\begin{align*}
\tilde{\mu}\left(\psi(\varpi)+ j\frac{\pd (\hbar u)}{\pd \hbar} \hbar^2,\varpi + \hbar^2\tilde{\iota}  u  \right) &= \tilde{\sigma}(\varpi + \hbar^2\tilde{\iota}  u).
\end{align*}
The hypothesis on $\tilde{\mu}$ reduces the left-hand side to $\tilde{\mu}(\psi(\varpi),\varpi)  + \tilde{\iota}\frac{\pd (\hbar u)}{\pd \hbar} \hbar^2 $, as required. 
%
%
\end{proof}

\begin{corollary}\label{DWLcor1}
 In the setup of Proposition \ref{DWLprop}, any two filtered sections $s$ of $\tilde{L}^{*=\id} \to L$ for which $s \circ \iota = \tilde{\iota}$ 
 induce homotopic maps $\mmc(F^2L)^{\nondeg}\to  \mmc( \tilde{F}^2\tilde{L}^{*=\id})^{\nondeg}$ and $\mmc(F^2L\by \hbar^2L\brhh)^{\nondeg}\simeq  \mmc( \tilde{F}^2\tilde{L}^{*=\id})^{\nondeg}$.
 \end{corollary}
\begin{proof}
It suffices to show that first the map in question is homotopic to  the zero section of $\rho_{\tilde{L}}$, which implies that the second map  must be homotopic to $  \rho_{\tilde{L}}^{-1} \circ\rho_{L\brh}$, since everything is equivariant with respect to the action of $ \mmc(\hbar^2U\brhh)$.

Since the filtered section $s$ of $h$ in the diagram below  preserves $U$, it also induces a section $s'$ of the other  horizontal map $h'$, with $v_1 \circ s'= s \circ v_2$.  
\[
 \begin{CD}
  \cocone( \tilde{F^2}\tilde{L}^{*=\id} \xra{\tilde{\sigma}\eps}  (\tilde{L}^{*=\id}/\tilde{\iota}U)\eps)@>{h'}>> \cocone(F^2L \xra{\sigma\eps} (L/\iota U)\eps)\\
 @V{v_1}VV @VV{v_2}V \\
 \tilde{F}^2\tilde{L}^{*=\id} @>{h}>> F^2L
 \end{CD}
\]
The maps $h'$ and $v_2$ both induce weak equivalences on $\mmc(-)^{\nondeg}$ (by Proposition \ref{DWLprop} and its proof, respectively), with $\rho_{\tilde{L}}$ corresponding to $v_1 \circ (h')^{-1} \circ (v_2)^{-1}$ in the homotopy category. Thus $\rho_{\tilde{L}}$ is given up to homotopy by $v_1 \circ s' \circ (v_2)^{-1} \simeq s$.
\end{proof}

\begin{examples}\label{DWLex}
The main sources of examples to which Corollary \ref{DWLcor1} can be applied are 
$\Cx$-analytic and   $\C^{\infty}$ settings where a de Rham theorem holds. That condition holds in the $\Cx$-analytic case  by \cite[Theorem IV.1.1]{hartshorneDRCoho}, and is known as being algebraically locally acyclic (ala) in  \cite{taroyanDRDerivedDG},  with 
\cite[Proposition 4.11 and Theorem 5.12]{taroyanDRDerivedDG} giving a large class of $\C^{\infty}$ examples.


The  equivalence between derived de Rham and Betti cohomology means that when $V=\oL\DR(\sO_X)[m]$ we can take $U[-m]$ to be the base field $\bK$  locally. There are natural central subalgebras $\hbar^{-1}\bK \subset Q\widehat{\Pol}(\sO_X,n)$ and  $\hbar^{-1}\bK\brh \subset Q\widehat{\Pol}(\sO_X,n)$ in all cases, ensuring the conditions of Proposition \ref{DWLprop} are satisfied.

The space $\mmc(  \cocone(\tilde{\sigma}\eps \co \tilde{F}^2\tilde{L}^{*=\id} \to (\tilde{L}^{*=\id}/\tilde{\iota }U)\eps))^{\nondeg}$ of preferred quantisations from that proposition then corresponds to self-dual quantisations $\sA$ equipped with an $\hbar\pd_{\hbar}$-derivation $D$, but  with curvature in $\hbar^{-1}(\sA/\bK)$ (rather than $\hbar^{-1}\sA$ as in Examples \ref{quantex}).
Specifically:

\begin{enumerate}[wide, labelindent=0pt, itemsep=6pt] 
 \item For $n>0$, a choice of formality equivalence  $E_{n+1}\simeq P_{n+1}$ intertwining the respective involutions  gives an equivalence  $Q\widehat{\Pol}(-,n)\simeq \widehat{\Pol}(-,n)\brh$ of the respective centres, acting as multiplication by $\hbar$ on   $\hbar^{-1}\bK\brh[n+1] \subset Q\widehat{\Pol}(-,n)$. For $n>1$, such formality equivalences as Hopf dg operads are all homotopic, by \cite{FresseWillwacherIntrinsic}, so Corollary \ref{DWLcor1} does not add much.
 
For $n=1$, such formality equivalences are not homotopic, and are given by even $\bK$-linear  associators $w$. For an ala homogeneous $\C^{\infty}$ or $\Cx$-analytic derived stack $\fX$,  
Corollary \ref{DWLcor1} (with  $U=\oR\Gamma(\fX,\bK)[n+1]$)   then implies that the resulting map  $s_w \co \cP(\fX,n) \to Q\cP(\fX,n)^{sd}$ to self-dual (weak) quantisations  becomes independent of the choice of associator on restriction to the space   $\cP(\fX,n)^{\nondeg}\simeq \Sp(\fX,n)$ of non-degenerate Poisson structures. The map $s_w$ corresponds to the deformation quantisation considered in \cite{CPTVV}.
 
 Proposition \ref{DWLprop} also gives a parametrisation 
 \[
 Q\cP(\fX,n)^{sd,\nondeg} \simeq \Sp(\fX,n) \by \mmc(\oR\Gamma(\fX,\hbar^2\bK\brhh)[n+1]),
 \] 
 which  Corollary \ref{DWLcor1} and Lemma \ref{DWLparamlemma} imply must correspond under the transformation $(\omega, \hbar^2f(\hbar)) \mapsto (\omega, \hbar^2 \frac{\pd (\hbar f)}{\pd \hbar}(\hbar))$ to the parametrisation $Q\cP(\fX,n)^{sd,\nondeg} \simeq \Sp(\fX,n) \by \mmc(\hbar^2\DR(\fX)\brhh[n+1])$ arising from any even associator by extending the Legendre transformation of \cite{poisson} as in Remark \ref{Lbrhrmk}.
 
  \item For $n>1$, degeneration of a filtration on  the Swiss cheese operad as in \cite{MelaniSafronovII} gives an equivalence $Q\widehat{\Pol}(-,-;n)\simeq \widehat{\Pol}(-,-;n)\brh$ of the respective centres, so the same reasoning shows that for a morphism $\fX \to \fY$ of ala homogeneous $\C^{\infty}$ or $\Cx$-analytic derived stacks, 
 the choice of even associator $w$ does not affect the quantisation map  $s_w \co \Lag(\fY,\fX;n) \to Q\cP(\fY,\fX;n)^{sd}$ from $n$-shifted Lagrangian structures to self-dual $n$-shifted (weakly)  quantised co-isotropic structures, or the corresponding parametrisation of self-dual quantisations $Q\cP(\fY,\fX;n)^{sd,\nondeg}$ as
 \[
  \Lag(\fY,\fX;n) \by \mmc\left(\cocone(\oR\Gamma(\fY,\hbar^2\bK\brhh) \to \oR\Gamma(\fX,\hbar^2\bK\brhh))[n+1]\right),
 \]
 
 \item For $n=0$, \cite{DQpoisson} uses formality of the $E_2$-operad to associate to any even associator an involutive equivalence $Q\widehat{\Pol}(-,0)\simeq \widehat{\Pol}(-,0)\brh$  when the cotangent complex is perfect, and similarly for $1$-shifted co-isotropic structures. For non-degenerate structures on a homogeneous  $\Cx$-analytic, or algebraically locally acyclic $\C^{\infty}$, derived $\infty$-stack,  
 Corollary \ref{DWLcor1} implies that the resulting parametrisation 
 \[
  Q\cP(\fX,0)^{sd,\nondeg} \simeq \Sp(\fX,0) \by \mmc(\oR\Gamma(\fX,\hbar^2\bK\brhh)[1]), 
 \]
of  non-degenerate self-dual quantisations is independent of the choice of associator, by the reasoning above. By Lemma \ref{DWLparamlemma}, a parameter  $(\omega, \hbar^{2}f(\hbar))$ here corresponds to the parameter $(\omega, \hbar^2 \frac{\pd (\hbar f)}{\pd \hbar}(\hbar))$ 
 in the parametrisation of \cite[Theorem 3.13]{DQnonneg}.
  
 \begin{enumerate}[wide]
\item\label{fedosovex}  
    For real symplectic manifolds $M$, the approach of Proposition \ref{DWLprop} agrees with that of \cite{DeWildeLecomte}, 
 as described in \cite[\S 4]{deligneDefFnsSymplectic}, so  Corollary \ref{DWLcor1} implies that the preferred quantisation given by any even associator yields  De Wilde and Lecomte's quantisation, and hence also Fedosov's canonical quantisation \cite{fedosov},
  by \cite[Proposition 4.12]{deligneDefFnsSymplectic}. 
 
The action of  $\hbar^{-1}\oR\Gamma(M,\R)[1]$ on the space of quantisations in Proposition \ref{DWLprop}  then corresponds precisely to the description \cite[\S 2.3 and Proposition 2.6]{deligneDefFnsSymplectic}, for $\pi$ non-degenerate, of     $\pi_0(Q\cP(M,0)_{\pi})$ as a torsor for $\H^2(M, \R\brh)$, except that we are restricting to fixed points of the involution to characterise $\pi_0(Q\cP(M,0)_{\pi}^{sd})$ as a torsor for $\H^2(M, \hbar\R\brhh)$, by identifying $\exp(\hbar\R\brhh)$ with  $\{ a \in \R\brh^+ ~:~ a(\hbar)a(-\hbar)=1\}$. 
 The parameter $(\omega, \hbar^2f(\hbar^2))$ of a self-dual quantisation given by Proposition \ref{DWLprop} thus corresponds to $\hbar^{-1}\omega +\hbar f(\hbar^2)$ in  Fedosov's parametrisation \cite{fedosov} as described in \cite[\S 3.7]{deligneDefFnsSymplectic}.\footnote{The parametrisation in those papers of non-self-dual structures in those terms is possible because $\H^1(M,\sO_M)$ and $\H^2(M,\sO_M)$ both vanish, so the morphism $\tilde{F}^2 Q\widehat{\Pol}(\sO_M,0) \to \tilde{F}^2 Q\widehat{\Pol}(\sO_M,0) \oplus \sO_M[1]$ induces an isomorphism on $\pi_0\mmc$. Corollary \ref{DWLprop} then adapts to describe that  space as a torsor for $\H^2(M, \hbar\R\brh)$.}
 
\item  For complex symplectic manifolds $(X,\omega)$, our preferred quantisation from Proposition \ref{DWLprop} is given by the DQ algebroid  $\mathsf{W}_X$ of \cite{PoleselloSchapira}. This follows because \cite[Appendix A.1]{DAgnoloKashiwaraQuantnCxSympMfds} gives $\mathsf{W}_X$ the same characterisation as the unique anti-involutive quantisation carrying an $\hbar\pd_{\hbar}$-derivation. 

In this setting, curvature vanishes for degree reasons, so the distinction between the constructions of Corollary \ref{exactquantcor} and Proposition \ref{DWLprop} arises locally in the $1$-morphisms as described at the end of Example \ref{quantex}.\ref{quantex0}, with $\gamma$ lying in $\hbar^{-1}(\mathsf{W}_X/\Cx)$ rather than $\hbar^{-1}\mathsf{W}_X $. The class $[\omega] \in \H^2(X,\Cx)$ measures the obstruction to lifting the associated gluing cocycle in $\mathsf{W}_X/\Cx$ to one in $ \mathsf{W}_X$.
\end{enumerate}


\end{enumerate}
 
\end{examples}

\begin{remark}
 In \cite{DQnonneg}, a parametrisation similar to that  of Proposition \ref{DWLprop} was established for $0$-shifted symplectic structures in the algebraic setting, using de Rham rather than Betti cohomology, and dependent on a choice of even associator. The formality equivalence given by the associator is a key ingredient in constructing  the resulting Legendre transformation, but it should be possible to establish an algebraic analogue of Proposition \ref{DWLprop} by working with relative (quantised) polyvectors over the de Rham stack. 
\end{remark}

\begin{remark}\label{GTrmk}
In general, understanding the dependence on an even associator $w$ of  the  involutive equivalence 
 $\alpha_w \co Q\widehat{\Pol}(-,0)\simeq \widehat{\Pol}(-,0)\brh$ from \cite{DQpoisson} amounts to understanding actions of the pro-unipotent Grothendieck--Teichm\"uller group $\GT^1$ and its graded analogue $\mathrm{GRT}^1$. The set of associators (equivalently of Levi decompositions of $\GT$) is a torsor for both $\GT^1$ and $\mathrm{GRT}^1$, with even associators corresponding to torsors for the subgroups $(\GT^1)^P$ and $(\mathrm{GRT}^1)^P$ commuting with the canonical element $P \in \GT(\Q)$ of order $2$.
 
The rigidity property of $\widehat{\Pol}(-,0)$ established in \cite[Corollary \ref{DQpoisson-fildefcor1}]{DQpoisson} implies that for stacks with perfect cotangent complexes,   $(\mathrm{GRT}^1)^P$ acts as a group of $L_{\infty}$-automorphisms of $\widehat{\Pol}(-,0)\brhh$, acting trivially modulo $\hbar^2$ and fixing $\bK\brhh$. Poisson data which quantise independently of $w$ are then  homotopy fixed points of the $(\mathrm{GRT}^1)^P$-action on  
\[
\mmc(\tilde{F}^2 \widehat{\Pol}(-,0)\brhh)= \mmc(F^2\widehat{\Pol}(-,0) \by \hbar^2\widehat{\Pol}(-,0) \brhh).
\]

In the setting of Remark \ref{Lbrhrmk}, Proposition \ref{DWLprop} constructs a space over $\mmc(\tilde{F}^2 \widehat{\Pol}(-,0)\brhh)$, equivalent to the zero section $\mmc(F^2\widehat{\Pol}(-,0))^{\nondeg}$, to which the $(\mathrm{GRT}^1)^P$-action lifts. That   implies that  non-degenerate $\hbar$-constant structures form a space of homotopy fixed points of $(\mathrm{GRT}^1)^P$ inside $\mmc(\tilde{F}^2 \widehat{\Pol}(-,0)\brhh)$, giving an alternative perspective on the first part of Corollary \ref{DWLcor1} in this case. The $\bK\brhh$-torsor characterisation of the second part of the corollary then implies that the $(\mathrm{GRT}^1)^P$-action on the whole of  $\mmc(\tilde{F}^2 \widehat{\Pol}(-,0)\brhh)^{\nondeg}$ is homotopy trivial.

%
 
\end{remark}

\section{Interpretation as algebras}\label{interpretnsn}

\subsection{Shifted Poisson algebras with formal derivation}\label{interpretalgsn}

The simplest example to consider is  the setting of Example \ref{poissonex}.(\ref{poissonexalg}), when $A$ is a cofibrant\footnote{or just quasi-smooth in the original sense, i.e. cofibrant over ind-smooth}  $R$-CDGA  so $(L[-n-1],F)$ is the cochain complex of $n$-shifted multiderivations  on $A$, given by
\[
F^p \widehat{\Pol}(A,n):= \prod_{i \ge p}\HHom_A(\mathrm{CoSymm}^i_A((\Omega^1_{A/R})[n+1]),A), 
\]
and $L$ is equipped with the Schouten--Nijenhuis bracket $[-,-]$, which extends the
 Lie bracket on $\Hom_A(\Omega^1_{A/R},A)$ multiplicatively. The derivation $\sigma$ is then defined to act as multiplication by $(p-1)$ on $p$-derivations $\HHom_A(\mathrm{CoSymm}^{p}_A((\Omega^1_{A/R})[n+1]),A)$

Thus a Maurer--Cartan element $\pi \in \mc(F^2L)$ consists of $p$-derivations $\pi_p$ of chain degree $(p-1)(n+1)-1$ for all $p \ge 2$, such that the data $(\delta,\pi)$ gives $A[n+1]$ the structure of an $L_{\infty}$-algebra. A Maurer--Cartan element $\pi+D\eps \in \mc(\cocone(\sigma\eps \co F^2L \to L\eps))$ lying over $\pi \in \mc(L)$ is then given by $p$-derivations $D_p$ of chain degree $(p-1)(n+1)$ for all $p \ge 0$, such that $\delta D + [\pi,D] =-\sigma(\pi)$, so 
\[
 \delta D_p + \sum_i [\pi_i,D_{p+1-i}] = -(p-1)\pi_p.
\]

Writing $\cP(A,n):= \mmc(F^2L)$ and $\cP^D(A,n):= \mmc(\cocone(\sigma\eps \co F^2L \to L\eps))$,  Lemma \ref{PDfibrelemma} implies that the homotopy fibre of $\cP^D(A,n) \to \cP(A,n)$ over $\pi$ is non-empty if and only if $[\sigma(\pi)]= 0$ in the $(n+2)$th Poisson cohomology group $\H^{n+2}_{\pi}(A)$, and then its $i$th homotopy group is a torsor for $\H^{n+1-i}_{\pi}(A)$, where  $\H^*_{\pi}(A):= \H^*(\widehat{\Pol}(A,n), \delta + [\pi,-])$.  

\subsubsection{Strict structures}\label{strictsn}

The data $D$ are most easily understood when both $\pi$ and $D$ are strict, in the sense that $\pi=\pi_2$, a bivector, and $D=D_0+D_1$, the sum of a constant and a vector. The condition on $D$ then simplifies to
\[
\delta D_0 =0,  \quad  \delta D_1 + [\pi,D_0] =0, \quad [\pi,D_1] =-\pi;
\]
for the Lie bracket $\{-,-\}_{\pi}$ defined by $\pi$,
the first condition says that $D_0$ is closed,  the second that $D_1$ is a homotopy from $0$ to $\{D_0,-\}_{\pi}$, and the third that $D_1\{a,b\}_{\pi}=\{D_1a,b\}_{\pi} \pm \{a, D_1b\}_{\pi} - \{a,b\}_{\pi}$.

\begin{definition}
 Given a dg operad $\cQ$, define $\cQ[n]$ to be the operad $(\cQ[n])(i):= \cQ(i)[n(i-1)]$; this is often denoted $s^{-n}\cQ$.
\end{definition}

Every shifted Poisson algebra with formal derivation is weakly equivalent to a strict one, with the following reasoning. As in \cite[Remark 1.6]{poisson}, an element $\pi \in \cP(A,n)$ gives $A$ the structure of a $\Com \circ (L_{\infty}[-n])$-algebra  (via a distributive law, cf. \cite[\S 8.6]{lodayvalletteoperads})
extending the given $\Com$-algebra structure on $A$. The bar-cobar adjunction as in \cite[\S 11.3]{lodayvalletteoperads} then gives a cofibrant  $P_{n+1}= \Com \circ (\Lie[-n])$-algebra $(A',\pi')$ equipped with a quasi-isomorphism $ (A',\pi') \to (A,\pi)$; in particular, $\pi'=\pi'_2$ is a strict $n$-shifted Poisson structure. Since $(A',\pi')$ is cofibrant as a $P_{n+1}$-algebra, inclusion gives a quasi-isomorphism
\[
 \DDer_{P_{n+1}}(A'_{\pi'},A'_{\pi'})[-n-1] \into  (F^1 \widehat{\Pol}(A',n), \delta + [\pi',-]),
\]
where the former is the subcomplex of strict $P_{n+1}$-algebra derivations, given by
\begin{align*}
&\ker( [\pi',-]\co \HHom_{A'}(\Omega^1_{A'/R}[n+1],A')\to \HHom_{A'}(\mathrm{CoSymm}^2_{A'}((\Omega^1_{A'/R})[n+1]),A')).
\end{align*}
Thus  inclusion gives a quasi-isomorphism
\[
 \cocone(A' \xra{[\pi',-]} \DDer_{P_{n+1}}(A'_{\pi'},A'_{\pi'})[-n])\into  (\widehat{\Pol}(A',n), \delta + [\pi',-]),
\]
so Lemma \ref{PDfibrelemma} implies that the homotopy fibre of $\cP^D(A',n) \to \cP(A',n)$ over the strict structure $\pi'$ is weakly equivalent to its subspace of strict derivations $D'$ with $D'_i=0$ for all $i \ge 2$.

\subsubsection{Formal derivations}\label{formalderivationsn}

Applying Corollary \ref{exactquantcor} to  the DGLA $L\brh \cong \widehat{\Rees}(L,W)$ of Remark \ref{Lbrhrmk} gives a weak  equivalence 
\begin{align*}
 &\mmc\left(\cocone((\sigma+ \hbar\pd_{\hbar}) \eps \co F^2L\by \hbar^2L\brhh \to L\brhh\eps)\right)^{\nondeg}\\
 &\xra{\sim}
\mmc\left(\cocone(\sigma\eps \co F^2L \to L\eps)\right)^{\nondeg},
\end{align*}
and this has a section given by inclusion of terms constant in $\hbar$.

Elements of  $\mc(\cocone((\sigma+ \hbar\pd_{\hbar}) \eps \co F^2L\by \hbar^2L\brhh \to L\brhh\eps))$ take the form $\pi + D\eps$, for $\pi \in \mc(F^2L\by \hbar^2L\brhh)$ and $D \in L^0\brhh $ satisfying $\delta D +[\pi,D]= -(\sigma+ \hbar\pd_{\hbar}) \pi$.
%
The isomorphism  $(-)_{\hbar} \co L\brh \to \widehat{\Rees}(L,W) $ from Remark \ref{Lbrhrmk}, which multiplies $m$-vectors by $\hbar^{m-1}$,  induces an isomorphism 
\begin{align*}
&\mmc(\cocone((\sigma+ \hbar\pd_{\hbar}) \eps \co F^2L\by \hbar^2L\brhh \to L\brhh\eps))\\
&\cong \mmc(\cocone(\hbar\pd_{\hbar} \eps \co \tilde{F}^2\widehat{\Rees}(L,W) \to \widehat{\Rees}(L,W)\eps)^{*=\id}).
\end{align*}

Thus $\sD_{\hbar}:= D_{\hbar} +  \hbar\pd_{\hbar}$ is a curved $L_{\infty}$ derivation with respect to the $L_{\infty}$ structure $\pi_{\hbar}$ on $A\brh$. Its terms are  multiderivations with respect to the commutative multiplication, and it agrees with $\hbar\pd_{\hbar}$ on the subalgebra $R\brh$. It respects the anti-involution $a(\hbar) \mapsto a(-\hbar)$ on $A \brh$ and satisfies $(\sD_{\hbar})_p \co A\brh \to \hbar^{p-1}A\brh$. These conditions are equivalent to the constraints on $D$. Beware that the curvature term $\sD_0$ lies in  $\hbar^{-1}A\brh$, rather than $A\brh$ itself, so strictly speaking $\sD_{\hbar}$ only defines a curved derivation of $(A(\!(\hbar)\!),\pi_{\hbar})$.

These properties are most easily understood when applied to  strict data $\pi=\pi_2$ and $D=D_0+D_1$ from \S \ref{strictsn}, with $\sD_{\hbar,1}= D_1  +\hbar\pd_{\hbar}$ being a  derivation on $A\brh$ with respect to both the commutative product and the Lie bracket $\hbar\{-,-\}_{\pi}$, respecting the anti-involution and agreeing with $\hbar\pd_{\hbar}$ on $R\brh$, such that $[\delta, \sD_{\hbar,1}] = \hbar\{\sD_{\hbar,0},-\}_{\pi}$, for $\sD_{\hbar,0}= \hbar^{-1}D_0 \in \hbar^{-1}A$, which in particular means that $ \hbar\{\sD_{\hbar,0},-\}_{\pi}=  \{D_0,-\}_{\pi} $.

When $n=0$ and $A$ is concentrated in degree $0$, the derivation $\sD_{\hbar}$ precisely corresponds to the additional data considered in \cite{DeWildeLecomteExactSymplectic,deligneDefFnsSymplectic}. 

\subsubsection{Weighted operadic interpretation}\label{weightoperad}

Curved $P_{n+1}$-algebras are governed by a curved operad (in the sense of \cite{DrummondColeBellierMillesHtpyThCurved}) $cP_{n+1}= \Com \circ (c\Lie[n])$, whose algebras $B$ have a (non-unital) graded commutative product 
of degree $0$, a graded Lie bracket $[-,-]$ of chain degree $n$ and a curvature element $c \in B_{-2-n}$ such that:
\begin{itemize}
 \item $[-,-]$ is a biderivation with respect to 
the product,
\item 
$\delta$ is a derivation with respect to both product and bracket,
\item $\delta c=0$ and  $\delta \circ \delta = [c,-]$.
\end{itemize}
In order for the structure to behave reasonably, the ideal generated by $c$ has to be pro-nilpotent.

The curved operads $c P_{n+1}$ and  $c \hat{P}_{n+1}:= \Com \circ (cL_{\infty}[n])$
have $\bG_m$-actions (or weight gradings) given by setting $\Com$ to have weight $0$ and putting the negative parity grading  on $c\Lie$ and $cL_{\infty}$, so  $c\Lie(r)$ and $cL_{\infty}(r)$ have weight $1-r$. 
\begin{definition}
Extending \cite[Definition \ref{DQpoisson-involutivePacdef}]{DQpoisson} to the curved setting, we denote these curved operads equipped with their $\bG_m$-actions by  $c P_{n+1}^{ac}$ and  $c \hat{P}_{n+1}^{ac}$. 

We define derivations $\sigma$ on  $c P_{n+1}$ and  $c \hat{P}_{n+1}$ to be multiplication by the weight of the $\bG_m$-action. 
\end{definition}

The map  $c \hat{P}_{n+1} \to c P_{n+1}$ is a weak equivalence of curved operads, and these also admit a weak equivalence from the curved operad $cP_{n+1,\infty}$ given by the cobar construction $\Omega(c P_{n+1}^!)^*$ of the co-unital co-operad  $(c P_{n+1}^!)^*$  Koszul dual to $c P_{n+1}$. The linear dual $ (c P_{n+1})^!$  of that co-unital co-operad is just the  operad  $P_{n+1}^{u}$ encoding  unital $P_{n+1}$-algebras, up to  shifts of both degree, generalising \cite[Lemma 1.16]{DQpoisson} in the obvious way. 

\begin{remark}
Beware that we are following the convention that Koszul dual co-operads are linear duals of Koszul dual operads, necessitating incorporation of degree shifts in the bar and cobar constructions, whereas \cite[\S 7.2.3]{lodayvalletteoperads} shifts the co-operad.
\end{remark}

\begin{definition}\label{weightshiftdef}
Given a $\bG_m$-equivariant chain complex $V$, let $V\<n\>$ be the chain complex $V$ with $\bG_m$-weight increased by $n$. Given a $\bG_m$-equivariant dg operad $\cP$, let $\cP\<n\>$ be the dg operad given by $\cP\<n\>(i):= \cP\<n(i-1)\>$, with the obvious operad operations.
\end{definition}

\begin{definition}
 We put a $\bG_m$-action on $cP_{n+1,\infty}$, denoted  $cP_{n+1,\infty}^{ac}$, by applying the cobar construction $\Omega$ to the co-unital dg co-operad  $(P_{n+1}^{u,ac}[-n]\<1\>)^*$. 
\end{definition}
Incorporating curvature in \cite[Lemma 1.16]{DQpoisson} then gives $\bG_m$-equivariant weak equivalences $cP_{n+1,\infty}^{ac} \to c \hat{P}_{n+1} \to c P_{n+1}$.

Now, the complete graded DGLA $L=\widehat{\Pol}(A,n)[n+1]$ of polyvectors is a quasi-isomorphic sub-DGLA of the deformation complex $L':=(\HHom_{\bS}((c P_{n+1}^!)^*, \End_A),\pd_A)$ in the terminology of \cite[\S 12.2.4]{lodayvalletteoperads}. Here, were are taking the convolution DGLA twisted by the differential $\pd_A$ corresponding the the commutative algebra structure on $A$ inducing a $P_{n+1}$-algebra structure with trivial bracket.  Taking the weight decomposition $P_{n+1}^{u,ac}[-n]\<1\>$ of $c P_{n+1}^!=  P_{n+1}^{u}[-n]$  induces a decomposition on that deformation complex as a product of weights, extending  the decomposition on polyvectors.

For the derivation $\sigma$ given by multiplication by weight, and $\eps^1$ a central variable of cochain degree $1$ (so $\eps^1\eps^1=0$), the following is then an immediate consequence of the definitions.
\begin{lemma}\label{epslemma}
The complete graded DGLA $\cocone(\sigma \eps \co L' \to L'\eps)$ is isomorphic to the  $\Q[\eps^1]$-linear  deformation complex of  
$A[\eps^1]$ as an algebra   over the curved dg operad $(cP_{n+1,\infty}^{ac}\ten_{\Q}\Q[\eps'], \delta + \sigma \eps^1)$.
\end{lemma}

Strict data $(\pi,D)$ on $A$ as in \S \ref{strictsn} then correspond to an algebra $A'_{\eps^1}:=(A[\eps^1],\pi,\delta + \eps^1 D_1, \eps^1 D_0)$  over the curved operad $(cP_{n+1}[\eps^1], \delta +\eps^1\sigma)$, for $\eps^1$ a central variable of cochain degree $1$ (so $\eps^1\eps^1=0$), 
such that $A'_{\eps^1}\ten_{\Q[\eps^1]}\Q$ is a $P_{n+1}$-algebra (specifically, $(A,\pi)$) with underlying commutative algebra $A$. More generally, elements of $\cP^D(A,n)$ correspond to  $(c\hat{P}_{n+1}[\eps^1], \delta +\eps^1\sigma)$-algebras $A'_{\eps^1}$ for which $A'_{\eps^1}\ten_{\Q[\eps^1]}\Q$ is a $\hat{P}_{n+1}$-algebra with underlying commutative algebra $A$. 

Applying the Rees construction to Lemma \ref{epslemma}, we have:
\begin{lemma}\label{epslemma2}
The complete DGLA $ (\cocone(\hbar\pd_{\hbar} \eps \co \widehat{\Rees}(L',W) \to \widehat{\Rees}(L',W)\eps)$ is isomorphic to the  $(\Q\brh[d\log \hbar], d_{\hbar})$-linear deformation complex of $(A\brh[d\log \hbar], d_{\hbar})$ as an algebra over the $\hbar$-adically complete
curved dg operad $(\widehat{\Rees}(cP^{ac}_{n+1,\infty})[d\log \hbar], \delta +d_{\hbar})$. 
Its involution is induced by the involutions $\hbar \mapsto -\hbar$ on either side, fixing $c\hat{P}^{ac}_{n+1}$ and $A$, and the filtration $\tilde{F}$ is determined by the grading on $cP_{n+1,\infty}^{ac}$.
\end{lemma}

Elements $\pi_{\hbar}, D_{\hbar}\eps$ of $\mc(\cocone(\hbar\pd_{\hbar} \eps \co \tilde{F}^2\widehat{\Rees}(L,W) \to \widehat{\Rees}(L,W)\eps)^{*=\id}$ as in  \S \ref{formalderivationsn} then correspond to  anti-involutive algebras $A'_{\hbar}$ over the complete
curved operad $(\widehat{\Rees}(c\hat{P}^{ac}_{n+1})[d\log \hbar], \delta +d_{\hbar})$ over $(Q\brh[d\log \hbar], d_{\hbar})$, where $\hbar$ and $d\log\hbar =\hbar^{-1}d\hbar$ are central variables of cochain degrees $0,1$ respectively.  The anti-involution  is an involutive isomorphism $e\co (A'_{\hbar})^o \cong A'_{\hbar}$, where the algebra $(A'_{\hbar})^o$ is the algebra defined via the involution of $(\widehat{\Rees}(c\hat{P}^{ac}_{n+1})[d\log \hbar], \delta +d_{\hbar})$ which sends  $\hbar$ to $-\hbar$ and fixes $d\log\hbar$ and $c\hat{P}^{ac}_{n+1}$. 

Explicitly, $c\hat{P}_{n+1}$ is generated by the commutative multiplication element $m \in \Com(2)_0$, the   $L_{\infty}$ operations $\ell_i \in (cL_{\infty}[n])(i)_{i-2+(i-1)n}$ for $i \ge 2$, and the curvature term $\theta \in (cL_{\infty}[n])(0)_{-n-2}$,  so $(\widehat{\Rees}(c\hat{P}^{ac}_{n+1})[d\log \hbar]$ is generated over $(Q\brh[d\log \hbar], d_{\hbar})$ by $m$, $\hbar \theta$ and $\{\hbar^{1-i}\ell_i\}_{i \ge 2}$.  In particular, the curvature of the differential is given by $[\hbar^{-1}\ell_2,\hbar \theta]= [\ell_2,\theta]$.

The graded-commutative algebra underlying $A'_{\hbar}$ is then just $A\brh[\d\log\hbar]$, with differential $\delta + \pi_1 + (d\log\hbar) D_1 + d_{\hbar}$, and  the other generators $\hbar^{1-i}\ell_i$ and $\hbar\theta$ of $\widehat{\Rees}(c\hat{P}^{ac}_{n+1})$ act as   $\{\pi_i + (d\log\hbar) D_i\}_{i \ge 2} $ and $(d\log\hbar) D_0$, respectively. The square of the differential is thus $[\pi_2+ (d\log\hbar) D_2, (d\log\hbar) D_0]= [\pi_2,(d\log\hbar) D_0]= [D_0,-]_{\pi_2}d\log \hbar$. The anti-involution is simply given by fixing  $A[\d\log\hbar]$ and sending $\hbar$ to $-\hbar$.

\begin{remarks}
If the weights of a dg operad $\cP$ are all non-positive, then $\cP\brh \subset \widehat{\Rees}(\cP)$, giving a $\cP$-algebra underlying any $\widehat{\Rees}(\cP)$-algebra. That condition is not quite satisfied by $cP^{ac}_{n+1}$ and $c\hat{P}^{ac}_{n+1}$ since the curvature has weight $1$, but $\pi_{\hbar} \in \mc(\tilde{F}^2\widehat{\Rees}(L,W))$ does still define a $c\hat{P}^{ac}_{n+1}\brh$-algebra structure since we have $\pi_0 \in \hbar^2A\brh$, so $ \hbar^{-1}\pi_0 \in \hbar A\brh$.

Because $\cP(\!(\hbar)\!) \subset \widehat{\Rees}(\cP)[\hbar^{-1}]$ in general, it always follows that $A'_{\hbar}[\hbar^{-1}]$ is naturally a $c\hat{P}^{ac}_{n+1}$-algebra, with operations $\pi_{\hbar} + (d\log\hbar)\sD_{\hbar}$.
\end{remarks}

\begin{remark}
We can think of the dg algebras 
 $\Q[\eps^1]$ and $(Q\brh[d\log \hbar], d_{\hbar})$ as functions on formal completions of $B\bG_m$ and $[\bA^1/\bG_m]$, respectively at the unique point and at $\{0\}$. Explicitly, they are given by applying the left adjoint $D^*$ of the denormalisation functor to the cosimplicial rings of functions on the associated simplicial schemes, as in \cite[Example \ref{poisson-DstarBG}]{poisson}, then completing at the origin. 
 
 Derivations with respect to $\sigma$  (resp.  $\hbar\pd_{\hbar}$-derivations) give structures over $\Q[\eps^1]$ (resp. $(Q\brh[d\log \hbar], d_{\hbar})$), so they interpolate between graded structures (resp. filtrations) and the underlying object (resp. the associated graded object), forgetting the grading. 
 Proposition \ref{exactcompprop} associates non-degenerate Poisson structures with formal derivation to exact symplectic structures.
 Shifted cotangent bundles $T^*[n]X$ then give an example where the formal derivation  integrates to a $\bG_m$-action, by giving cotangent forms weight $-1$ and  $X$ weight $0$. Derived critical loci give an example where the formal derivations  do not integrate in general.  
\end{remark}

\subsubsection{Generalisations}\label{genPDsn}

The considerations of \S \ref{weightoperad} all generalise in much the same way to strings of CDGAs, to analytic and $\C^{\infty}$ settings, to stacky CDGAs, to co-isotropic structures and to double Poisson structures. Whereas curvature terms in deformations of the non-negatively shifted Poisson structures on the CDGAs considered in derived geometry only affect automorphisms (via the polynomial de Rham complex on simplices), on stacky CDGAs they give rise to more objects, corresponding formally to line bundles (when $n=-1$) or more generally to  $K(\bG_m,n+2)$-torsors.

The general pattern is as follows.
\begin{setting}\label{cooperadsetting}
Assume that we have a $\bG_m$-equivariant, i.e. weight graded, coloured dg co(pr)operad (or occasionally, $\infty$-co-(pr)operad) $\C=\bigoplus_{m\le 1} \cW_m\C$ with $\cW_1\C$ being the  arity $0$ component of $\C$, and  the arity $1$ part of  $\cW_0\C$ consisting only of co-identities (so $\cW_0\C(x,x)=I$ for all colours $x$ and $\cW_0\C(x,y)=0$ for $x \ne y$). 
\end{setting}

\begin{definition}\label{cPpoperaddef}
Since $\cW_{\ge 0}\C$ and $\cW_0\C$ are co-augmented, so the cobar construction $\Omega$ on co-augmented co-(pr)operads gives us  a  $\bG_m$-equivariant dg (pr)operad  $\cP:= \Omega(\cW_{\le 0}\C)$ and a dg (pr)operad  $\Omega\cW_0\C= \cW_0\cP$, on the same colours as $\C$. 
The curved cobar construction gives us a curved $\bG_m$-equivariant dg (pr)operad  $c\cP:= \Omega(\C)$ on those colours.
\end{definition}

Note that since the quotient maps $\C \to \cW_{\le 0}\C \to \cW_0\C$ are morphisms of dg co(pr)operads, they give rise to morphisms $c\cP \to \cP \to \cW_0\cP$, and the latter morphism has a section  $\cW_0\cP \to \cP$. 

\begin{definition}
Define an  involution $*$ on each of $\C$, $\cP$ and $c\cP$  as the action of $-1 \in \bG_m$ (so $(-1)^m$ on $\cW_m\cP$). Let $(-)^{\op}$ be the corresponding involution of the category of algebras. 

Let $\sigma$ be the coderivation on $\C$ (resp. derivation on $\cP$ or $c\cP$) given by multiplying $\cW_m\C$ (resp. $\cW_m\cP$) by $m$. 
\end{definition}
 
 \begin{examples}\label{CPgradedex}
 \begin{enumerate}[wide, labelindent=0pt, itemsep=6pt]
  \item  In the setting of \S \ref{weightoperad}, we have $\C= (P_{n+1}^{u,ac}[-n]\<1\>)^*$, so $\cP= P_{n+1,\infty}^{ac}$, while the curved cobar construction $\Omega(\C)$ is $cP_{n+1,\infty}^{ac}$, and $\cW_0\cP=\Com_{\infty}$. 
  
  \item\label{CPgradedex:coiso} For $n$-shifted co-isotropic structures, we take $\C$ to be the coloured $\infty$-co-operad $(P_{n+1,n}^{u,ac}[-n]\<1\>)^*$, where $P_{[n+1,n]}^u$ is the operad $\bP_{[n+1,n]}$ of \cite[\S 3.4]{MelaniSafronovI}.  
  Thus $\cP$ is a cofibrant resolution of the non-unital operad $P_{[n+1,n]}$ (i.e. $\bP^{nu}_{[n+1,n]}$), by  \cite[Proposition 3.8]{MelaniSafronovI}, so a $\cP$-algebra is a pair $(A,B)$ where $A$ is a $P_{n+1,\infty}$-algebra acting on the $P_{n, \infty}$-algebra $B$. The operad $\cW_0\cP$  simply encodes pairs $(A,B)$ of commutative (or $\Com_{\infty}$) algebras equipped with a morphism $A \to B$.
  \end{enumerate}
 \end{examples}

\begin{setting}\label{PQalgsetting} 
 The setup is then 
 that we have a $\cW_0\cP$-algebra $A$   in some dg (pr)operad $\cQ$. 
\end{setting}
In algebraic settings, $\cQ$  just consists of multilinear maps, but in analytic and $\C^{\infty}$ settings $\cQ$ has to be a dg operad of polydifferential operators.

\begin{definition}
 In Settings \ref{cooperadsetting} and \ref{PQalgsetting}, 
we take $L$ to be the complete graded DGLA
given by the deformation complex
$(\HHom_{\bS}(\C, \cQ),\pd_A)$ in the terminology of \cite[\S 12.2.4]{lodayvalletteoperads} (or its coloured and properadic generalisations), i.e. the convolution DGLA twisted by the differential $\pd_A$ corresponding the $\Omega(\cW_0\C)=\cW_0\cP$-algebra $A$, regarded as an $\Omega(\C)=c\cP$-algebra. The weight decomposition $L=\prod_{m \ge -1}\cW_mL$ is induced by that on $\C$, so its sub-DGLA $\prod_{m \ge 0}\cW_mL$ is the deformation complex
$(\HHom_{\bS}(\cW_{\le 0}\C, \cQ),\pd_A)$ of $A$ as a $\cP$-algebra.
\end{definition}

Thus $\mmc(F^2L)$ is the space of $\cP$-algebras   with underlying $\cW_0\cP$-algebra $A$ (defined via the section $\cW_0\cP \to \cP$), or equivalently the space of morphisms $\cP \to \cQ$ of dg (pr)operads extending $A \co \cW_0\cP \to \cQ$.

Adapting Lemmas \ref{epslemma} and \ref{epslemma2}, we then have:

\begin{lemma}\label{epslemmagenl}
The complete graded DGLA $\cocone(\sigma \eps \co L \to L\eps)$ is isomorphic to the  $\Q[\eps^1]$-linear  deformation complex of  
$A[\eps^1]$ as an algebra   over the curved dg (pr)operad $(c\cP\ten_{\Q}\Q[\eps'], \delta + \sigma \eps^1)$.

The complete DGLA $ (\cocone(\hbar\pd_{\hbar} \eps \co \widehat{\Rees}(L,W) \to \widehat{\Rees}(L,W)\eps)$ is isomorphic to the  $(\Q\brh[d\log \hbar], d_{\hbar})$-linear deformation complex of $(A\brh[d\log \hbar], d_{\hbar})$ as an algebra over the $\hbar$-adically complete
curved dg operad $(\widehat{\Rees}(c\cP)[d\log \hbar], \delta +d_{\hbar})$.
Its involution is induced by the involutions $\hbar \mapsto -\hbar$ on either side, fixing $c\cP$ and $A$, and the filtration $\tilde{F}$ is determined by the grading on $c\cP$.
\end{lemma}

Thus $\mmc(\cocone(\sigma \eps \co F^2L \to L\eps))$ can be thought of as  the space of $(c\cP\ten_{\Q}\Q[\eps'], \delta + \sigma \eps^1)$-algebras $B$ flat over $\Q[\eps^1]$,  with curvature divisible by $\eps^1$, for which the $\cP$-algebra $B/(\eps^1)$ having underlying $\cW_0\cP$-algebra $A$. Specifically, for the contraderived construction $\cQ^{\ctr}$ as in \cite[Definition \ref{coddt-Qctrdef}]{coddt}, it is the space of $\Q[\eps^1]$-linear morphisms $(c\cP\ten_{\Q}\Q[\eps'], \delta + \sigma \eps^1) \to \cQ^{\ctr}(\Q[\eps^1])$ for which reduction modulo $\eps^1$ factors through a morphism $\cP \to \cQ$ extending $A \co \cW_0\cP \to \cQ$.

Similarly, $\mmc (\cocone(\hbar\pd_{\hbar} \eps \co \tilde{F}^2\widehat{\Rees}(L,W) \to \widehat{\Rees}(L,W)\eps)^{*=\id})$ can be thought of as the space of involutive  $(\widehat{\Rees}(c\cP)[d\log \hbar], \delta +d_{\hbar})$-algebras $B$ flat and complete over $(\Q\brh[d\log \hbar], d_{\hbar})$, with curvature (i.e. the image of $\hbar\C(0)= \hbar\cW_1\C$) lying in $(\hbar, d\log \hbar)B$ and the $\cP$-algebra $B/(\hbar, d\log\hbar)$ having underlying $\Omega(\cW_0\C)$-algebra $A$. In particular, $B[\hbar^{-1}]$ is a $(c\cP(\!(\hbar)\!) [d\log \hbar], \delta +d_{\hbar})$-algebra. 

\subsection{Deformation quantisations with \tps{$\hbar\pd_{\hbar}$}{hd}-derivations}\label{interpretDQsn}

Deformation quantisations as in \S \ref{quantsn} are associated to certain filtered DGLAs, and these often arise from filtered operads. The setup is very similar to \S \ref{genPDsn}, but with operads only being filtered, rather than graded. The general pattern as follows. 

\begin{setting}\label{filtcooperadsetting}
Assume that we have a coloured dg co(pr)operad $\C$  (or occasionally, $\infty$-co-(pr)operad) equipped with an 
 increasing filtration $W$, compatible with the co-operad operations, with $\C=W_1\C$. Unless stated otherwise, it is  equipped with an involution $*$ respecting the filtration and homotopic to $(-1)^m$ on $\gr^W_m\C$. We assume moreover that $W_0\C$ has no component of arity $0$, that $\gr^W_1\C$ is concentrated in arity $0$ and that $\gr^W_0\C$  consists only of co-identities (so $\gr^W_0\C(x,x)=I$ for all colours $x$ and $\gr^W_0\C(x,y)=0$ for $x \ne y$). 
\end{setting}
 
In particular, writing $\C(0)$ for the arity $0$ component $\bigoplus_x \C(x) $, we have a decomposition $\C = (W_0\C) \oplus \C(0)$,  and the quotient co-operad $\C/\C(0)$ is co-augmented, as is its quotient $\gr^W_0(\C/\C(0))$.  
\begin{definition}
The cobar construction $\Omega$ on co-augmented co-(pr)operads  thus gives us a filtered dg (pr)operad  $\cP:= \Omega(\C/\C(0))$ with $W_0\cP=\cP$ and a dg (pr)operad  $\gr_W^0\cP=\Omega(\gr^W_0(\C/\C(0)))$, on the same colours as $\C$. The curved cobar construction gives us a curved filtered dg (pr)operad  $c\cP:= \Omega(\C)$ on those colours. 
\end{definition}
Note that since the quotient maps $\C \to \C/\C(0) \to \gr^W_0(\C/\C(0))$ are morphisms of dg co(pr)operads, they give rise to morphisms $c\cP \to \cP \to \gr^W_0\cP$. Note that the involution $*$ induces involutions on $c\cP$ and $\cP$, homotopic to $(-1)^m$ on $\gr^W_m$.

\begin{examples}\label{quantinterpretnex}
\begin{enumerate}[wide, labelindent=0pt, itemsep=6pt]
 \item \label{BD1ex} The prototypical example is deformation quantisation of $0$-shifted Poisson structures, where we  filter the unital associative operad $\Ass^u$  as in \cite[\S 3.5.1]{CPTVV}, taking the smallest filtration $W$ such that the commutator bracket lies in $W_{-1}\Ass^u(2)$ and $W_0\Ass^u=\Ass^u$. The Rees construction of this filtration is the $BD_1$ operad of \cite[Definition 13.2.2]{costelloNotesSupersymm} (due to Ed Segal), so we denote the filtered operad by $BD_1^{u,\pre}$ and write $BD_1^{\pre}$ for the induced filtration on the non-unital associative operad $\Ass \subset \Ass^u$.
 Thus a $BD_1^{\pre}$-algebra $A$ in filtered complexes is a dg associative algebra for which:
 the multiplication respects the filtration and the commutator lowers filtration index, or equivalently the associated graded $\gr^WA$ is graded commutative; a $BD_1^{u,\pre}$-algebra also has a unit $1 \in \z_0W_0A$. 
 
 There is also a curved enhancement $cBD_1^{\pre}$ with an element $\theta$ generating $W_1cBD_1^{\pre}(0)_{-2}$ such that curvature is the commutator of $\theta$ with the bracket. A $BD_1^{\pre}$-algebra $A$ in filtered curved complexes has an associative multiplication respecting the filtration such that $\gr^WA$ is graded-commutative, together with  an element  $c \in \z_{-2}W_1A$ such that the differential is a derivation with curvature $[c,-]$.
  
Adapting the notation of Definition \ref{weightshiftdef} from gradings to filtrations, 
an argument similar to \cite[Lemma \ref{DQLag-barcobarprop1}]{DQLag} shows that 
the Koszul dual of $BD_1^{\pre}$  is the filtered operad $BD_1^{\pre}\<1\>$, which has filtration $W_m( BD_1^{\pre}\<1\>(i))= W_{m-i+1} BD_1^{\pre}(i)$. This is the smallest filtration for which the product lies in $W_1$ and the  commutator bracket lies in $W_0$, while the filtration of $BD_1^{u,\pre}\<1\>$ also stipulates $1 \in W_{-1}$. Given a vector space $V$, the free $BD_1^{u,\pre}\<1\>$-algebra generated by $V$ is the tensor algebra $T(V)$ filtered by powers of the free Lie algebra $\Lie(V)$,  with  $W_{n-1}T(V)= \im(\Lie(V)^{\ten n} \to T(V))$ for $n \ge 0$. 

In this case, we thus take $\C$ to be  $ BD_1^{u,\pre}\<1\>^*$, equipped with the involution which sends an associative algebra to its opposite algebra, so  $BD_{1,\infty}^{\pre} :=\Omega(\C/\C(0))$ is a cofibrant replacement of $BD_1^{\pre}$, and $cBD_{1,\infty}^{\pre} :=\Omega(\C)$ is a replacement for the curved associative operad $cBD_1^{\pre}$ equipped with the  filtration above. We have $\gr^WBD_1^{\pre} \cong P_1^{ac}$ and $\gr^WBD_1^{u,\pre} \cong P_1^{u,ac}$, so $\gr^WBD_{1,\infty}^{\pre} \cong P_{1, \infty}$ and $\gr^WcBD_{1,\infty}^{\pre} \cong cP_{1, \infty}$.

\item \label{BD2ex} To quantise $1$-shifted Poisson structures, we can filter the brace operad by a decreasing filtration $\gamma$ as in \cite[Definition \ref{DQLag-acbracedef}]{DQLag}, with the cup product in $\gamma^0$ and $r$-braces in $\gamma^r$. Alternatively, we can take the quasi-isomorphic filtration given by good truncation $\tau_{\ge r}$, and we can replace the brace operad with chains on the $E_2$ operad. Setting $W_{-r}$ to be $\gamma^r$ or $\tau_{\ge r}$, we thus obtain a filtered operad $BD_2^{u,\pre}$, with the non-unital sub-operad $BD_2^{\pre}$ being Koszul dual up to a shift. 

These operads carry involutions   corresponding to the operation sending an $E_2$-algebra to its opposite, corresponding to the element 
$P \in \GT(\Q)$ of  \cite[\S 4.1]{barnatan}, following  \cite[Proposition 4.1]{drinfeldQuasiTriangular}.
  
 \item  \label{BDnex} The  characterisation via chains on the $E_2$ operad generalises to all positively shifted structures, defining  as in \cite[\S 3.5.1]{CPTVV}  a filtration on chains on the $E_{n+1}$-operad by setting $W_{-r}$ to be good truncation $\tau_{\ge nr}$. We thus obtain a filtered operad $BD_{n+1}^{u,\pre}$, with the non-unital sub-operad $BD_{n+1}^{\pre}$ being Koszul dual up to a shift. Again, these carry a canonical involution.

We then set $\C$ to be $ BD_{n+1}^{u,\pre}[-n]\<1\>^*$, so   $BD_{n+1,\infty}^{\pre} :=\Omega(\C/\C(0))$ is a cofibrant replacement of $BD_{n+1}^{\pre}$, and $cBD_{n+1,\infty}^{\pre} :=\Omega(\C)$ is a replacement for the curved associative operad $cBD_{n+1}^{\pre}$ equipped with the good truncation filtration. We have $\gr^WBD_{n+1}^{\pre} \simeq P_{n+1}^{ac}$ and $\gr^WBD_{n+1}^{u,\pre} \simeq P_{n+1}^{u,ac}$, so $\gr^WBD_{n+1,\infty}^{\pre} \simeq P_{n+1, \infty}$ and $\gr^WcBD_{n+1,\infty}^{\pre} \cong cP_{n+1, \infty}$.

For $n>1$, \cite{FresseWillwacherIntrinsic} can be interpreted as giving canonical filtered quasi-isomorphisms $BD_{n+1}^{u,\pre} \simeq P_{n+1}^{u,\pre}$ intertwining the respective involutions. For $n=1$, such equivalences exist but depend on a choice of even associator.

\item For non-negatively shifted co-isotropic structures, the Swiss cheese construction of \cite[\S 3.3]{MelaniSafronovI} gives us a filtered $\infty$-co-operad $P( BD_{n+1}^{\pre}[-n-1]\<1\>^*, BD_{n}^{\pre}\<1\>^*)$ on two colours, and we can take $\C$ to be its co-unital enhancement $ P( BD_{n+1}^{u,\pre}[-n-1]\<1\>^*, BD_{n}^{u,\pre}\<1\>^*)$. Then $BD_{[n+1,n]\infty}^{\pre} :=\Omega(\C/\C(0))$ is a model for the operad encoding pairs $(A,B)$ with $A$ a $BD_{n+1}$-algebra acting on a $DB_n$-algebra $B$, while $\Omega \C$ incorporates curvature. Again, for $n \ge 1$ these carry involutions corresponding to taking the opposite algebra. The associated graded dg operad $\gr^WBD_{[n+1,n]\infty}^{\pre}$ is a resolution of the dg operad $P_{[n+1,n]}$ from Example \ref{CPgradedex}.(\ref{CPgradedex:coiso}), and similarly for the  curved analogues.

\item\label{BVex} For $(-1)$-shifted deformation quantisations, consider the dg operad $BV$ whose algebras in chain complexes are equipped with a graded-commutative product 
of chain degree $0$ and a graded Lie bracket $[-,-]$ of chain degree $-1$ such that the bracket is closed under the differential and a derivation with respect to the multiplication, together with the key property that $\delta(a\cdot b) -(\delta a)\cdot b \mp a\cdot (\delta b) =[a,b]$. The Rees construction of this filtration is the $BD$ operad of \cite[Definition 2.2.5]{gwilliamThesis}, so we then let $BD_0^{\pre}$ be the dg operad $BV$ equipped with the filtration which places the product in $W_0$ and the Lie bracket in $W_{-1}$.

The BV operad is again Koszul self-dual, with  the twisting cochain arising since for any pair $A,B$ of BV algebras, the tensor product $A \ten B$ carries a natural BV algebra structure, and hence a DGLA structure on $A\ten B[-1]$. The Koszul dual of $BD_{0}^{\pre}$ is $BD_{0}^{\pre}\<1\>$, with the shift in filtration corresponding to the weight of the Lie bracket. That the twisting cochain induces a quasi-isomorphism then follows because $\gr^WBD_0^{\pre}\cong P_0$, so we know we have Koszul duality on the associated graded dg operads.

We can then take  $\C$ to be $ BD_{0}^{u,\pre}\<1\>^*$, with very similar considerations to those for the higher shifts. However, beware that $BD_0^{\pre}$ does not carry an involution with the required properties, which is why square roots of the dualising line bundle (also known as spin structures or the space of Berezinian half-densities) enter the picture. 

Specifically, the notion of curved $BD_{0}$-algebra deformations generalises to allow morphisms to have curvature in $A\brh^{\by}$ rather than just $1+ \hbar A\brh = \exp(\hbar A\brh)$, with the image modulo $\hbar$ in $A^{\by}$ giving rise to an underlying line bundle $\sL$. These algebras do have a duality theory coming from  the equivalence $\sD(\sL)^{\op} \simeq \sD(\sL^{\vee})$ of rings of differential operators for the  Grothendieck--Verdier dual $\sL^{\vee}:=\sHHom_{\sO}(\sL,\sM)$  for any line bundle $\sM$ with a right $\sD$-module structure. The opposite algebra of a deformation quantisation of $(\pi,\sL)$ is a  deformation quantisation of $(\pi, \sL^{\vee})$, so there is a notion of self-dual deformation quantisations of $(\pi, \sqrt{\sM})$; see \S \ref{vanishsn} for more details.


\end{enumerate}

 \end{examples}

\begin{setting}
 The setup is then
that we have a $\gr^W_0\cP$-algebra $A$   in some dg (pr)operad $\cQ$. 
\end{setting}
In algebraic settings, $\cQ$  just consists of multilinear maps, but in analytic and $\C^{\infty}$ settings $\cQ$ has to be a dg operad of polydifferential operators.

\begin{definition}
We then  take $\breve{L}$ to be the complete filtered DGLA
given by the deformation complex
$(\HHom_{\bS}(\C, \cQ),\pd_A)$.
\end{definition}

Lemma \ref{epslemma2} then adapts to give:
\begin{lemma}\label{epslemmaquant}
The complete DGLA $ (\cocone(\hbar\pd_{\hbar} \eps \co \widehat{\Rees}(\breve{L},W) \to \widehat{\Rees}(\breve{L},W)\eps)$ is isomorphic to the  $(\Q\brh[d\log \hbar], d_{\hbar})$-linear deformation complex of $(A\brh[d\log \hbar], d_{\hbar})$ as an algebra over the $\hbar$-adically complete
curved dg operad $(\widehat{\Rees}(c\cP)[d\log \hbar], \delta +d_{\hbar})$.
Its involution is induced by the involutions $\hbar \mapsto -\hbar$ on either side, fixing $c\cP$ and $A$, and the filtration $\tilde{F}$ is determined by the filtration on $c\cP$.
\end{lemma}

Similarly to the interpretation of Lemma \ref{epslemmagenl}, $\mmc (\cocone(\hbar\pd_{\hbar} \eps \co \tilde{F}^2\widehat{\Rees}(\breve{L},W) \to \widehat{\Rees}(\breve{L},W)\eps)^{*=\id})$ can be thus be thought of as the space of involutive  $(\widehat{\Rees}(c\cP)[d\log \hbar], \delta +d_{\hbar})$-algebras $B$ flat and complete over $(\Q\brh[d\log \hbar], d_{\hbar})$, with curvature (i.e. the image of $\hbar\C(0)= \hbar\cW_1\C$) lying in $(\hbar, d\log \hbar)B$ and the $\gr^W\cP$-algebra $B/(\hbar, d\log\hbar)$ having underlying $\gr^W_0\cP$-algebra $A$. In particular, $B[\hbar^{-1}]$ is a $(c\cP(\!(\hbar)\!) [d\log \hbar], \delta +d_{\hbar})$-algebra. 

\section{Vanishing cycles}\label{vanishsn}

As explained in Examples \ref{poissonex}, when working algebraically  over a field or in the $\Cx$-analytic setting, there is an equivalence
$ \Sp^{\ex}(\fX,-1) \simeq \Sp(\fX,-1) \by \H^0_{\dR}(\fX)$   by \cite[Proposition 3.2]{KinjoParkSafronov} or \cite[Proposition 2.2.11]{HHR1}, so that in particular every $(-1)$-shifted symplectic structure is  exact in a canonical way. As in Examples \ref{poissonex}, we thus have
$\cP^D(\fX,n)^{\nondeg} \simeq \Sp(\fX,-1) \by \H^0_{\dR}(\fX)$.

Deformation quantisation for $(-1)$-shifted symplectic structures is formulated in \cite{DQvanish}  for line bundles 
on derived Artin stacks, and then established for  self-dual non-degenerate quantisations of Gorenstein derived DM stacks when the determinant line  bundle
has a square root 
(which can be thought of as 
a spin structure or module of half-densities). As explained in \cite[\S 4.4]{DStein}, those results extend verbatim to analytic and $\C^{\infty}$ settings, with the $\C^{\infty}$ case spelt out in \cite[\S \ref{DQDG-quantneg1sn}]{DQDG}.

In this section, we will look at the distinguished choice of quantisation associated to the canonical exact structure on a $(-1)$-shifted symplectic structure, and  relate it to the perverse sheaf of vanishing cycles from \cite{BBDJS}, but first in \S \ref{Dinvsn} we will give an equivalent formulation of self-dual quantisations not requiring square roots.   

\subsection{The anti-involutive ring of twisted differential operators}\label{Dinvsn}

Since the BV operad does not carry an involution,  Example \ref{quantinterpretnex}.(\ref{BVex}) does not recover the setting to which we can apply Corollary \ref{exactquantbrhcor}. Instead, elaborating on Example \ref{quantex}.(\ref{quantexvanish}), the solution in \cite[Definition \ref{DQvanish-selfdualdef}]{DQvanish} is to take a line bundle $\sL$ on $\fX$ equipped with a right $\sD$-module structure on $\sL^{\ten 2}$. 

Then $\oR\sHHom_{\sO}(-,\sL^{\ten 2})$ gives an anti-involution (up to coherent homotopy) of the filtered dg category of perfect $\sO_{\fX}$-modules and differential operators. If we write $\sD(\sE)$ for the ring of differential operators on a perfect complex $\sE$ (quasi-isomorphic to $\sE\ten_{\sO}\sD\ten_{\sO}\oR\sHHom_{\sO}(\sE,\sO)$), we thus have an anti-involution $t$ on $\sD(\sE)$ whenever $\sE$ is equipped with a non-degenerate symmetric pairing $\sE\ten\sE \to \sL^{\ten 2}$. In particular, we have a homotopy involution $t \co \sD(\sL)^{\op} \simeq \sD(\sL)$, compatible with the respective order filtrations $F$.

We now digress a little,  generalising that setup slightly to show that the formulation and existence of self-dual $(-1)$-shifted deformation quantisations does not need such line bundles $\sL$ to exist. However, a choice of $\sL$ is needed to associate $\sD(\!(\hbar)\!)$-modules, and hence perverse sheaves, to such quantisations.

\subsubsection{Constructing self-dual twisted differential operators canonically}\label{Dinvconstrnn}

When the complete  graded DGLA $(L,\cW)$ of  \S \ref{genpoissstrsn} is a DGLA $\widehat{\Pol}(-,-1)$ of $(-1)$-shifted polyvectors, we have the following canonical construction of a  filtered involutive DGLA  $(L^{\smile,sd} ,V,i)$ satisfying the conditions for  $\breve{L}$ as in \S \ref{quantsn}, generalising the construction of the dg algebra $\sD(\sqrt{\sK_X})$ of differential operators on a square root of the determinant bundle.

\begin{setting}
All we need is for the graded DGLA $\bigoplus_{j \ge 0} \cW_{j-1}L$ to underlie a $\bG_m$-equivariant unital $P_1^{ac}$-algebra structure (where $j$ indexes the new weight grading, corresponding to $j$-vectors when $L= \widehat{\Pol}(-,-1)$), such that  $\cW_0L$ is cofibrant  as a module over the CDGA $\cW_{-1}L$, and such that  the natural maps 
$
\Symm^j_{\cW_{-1}L}\cW_0L \to \cW_{j-1}L
$
induced by the commutative multiplications  are quasi-isomorphisms.
\end{setting}

\begin{definition}
We can then form a DG Lie--Rinehart algebra structure $(\cW_{-1}L, V_0L^{\smile,\inv})$  on the pair  $(\cW_{-1}L, \cW_{-1}L \oplus \cW_0L)$ as follows.

The Lie bracket on  $V_0L^{\smile,\inv}$ is given by regarding it as a sub-DGLA of $L$, and there is an anchor map $V_0L^{\smile,\inv} \to \DDer(\cW_{-1}L, \cW_{-1}L)$ given by  taking the projection  $V_0L^{\smile,\inv} \to \cW_0 L$
composed with the adjoint map $\cW_0L \to \DDer(\cW_{-1}L, \cW_{-1}L)$ associated to the Lie bracket. The left $\cW_{-1}L$-module structure on $V_0L^{\smile,\inv}$ is then given by
\[
 f\star (g,u):= (fg + \half [f,u], fu).
\]
\end{definition}

Note that the unit $1 \in \cW_{-1}L$ is  a central element in $V_0L^{\smile,\inv}$ with respect to the Lie bracket.

\begin{definition}
We then adapt \cite[Definition 2.2.1]{ginzburgLectDmods} by  defining the associative algebra  $L^{\smile,\inv}$ to be the quotient of the universal enveloping algebra of the dg Lie-Rinehart algebra $(\cW_{-1}L,V_0L^{\smile,\inv})$ by the relation identifying $1 \in \cW_{-1}L$ with the unit of the universal enveloping algebra. Explicitly, this means that $L^{\smile,\inv}$ is  the quotient of the universal enveloping algebra $\cU$ of the Lie algebra  $ V_0L^{\smile,\inv}$ by the relations $1_{\cW_{-1}L}= 1_{\cU}$ and  $f\ten \theta = f\star \theta$ for $f \in \cW_{-1}L$, $\theta \in   V_0L^{\smile,\inv}$.
\end{definition}

\begin{definition}
We then define an increasing filtration on $L^{\smile,\inv}$  given by setting $ V_{-1}L^{\smile,\inv}:= \cW_{-1}L$ and  defining $V_{j-1}L^{\smile,\inv}$ to be the image of $(V_0L^{\smile,\inv})^{\ten j}$ under multiplication for $j\ge 1$.
\end{definition}
This satisfies $(V_pL^{\smile,\inv})\cdot  (V_ qL^{\smile,\inv})\subset V_{p+q+1}L^{\smile,\inv}$, and the cofibrancy condition on the  $\cW_{-1}L$-module $\cW_0L$ ensures that $\gr^V_pL^{\smile,\inv} \cong \Symm^{p+1}_{\cW_{-1}L}\cW_0L \simeq \cW_pL$. In particular, $\gr^VL^{\smile,\inv}$ is commutative, so the commutator bracket satisfies $[V_pL^{\smile,\inv}, V_qL^{\smile,\inv}]\subset V_{p+q}L^{\smile,\inv}$.

The properties so far are common to many other constructions; for instance, replacing the product $\star$ with the left action $f\cdot (g,u):= (fg , fu)$ generalises the construction of untwisted differential operators. What sets $L^{\smile,\inv}$ apart is that it carries an anti-involution $t$. 

\begin{definition}
Specifically, define $t$ on $V_0L^{\smile,\inv}$ by setting $(g,u)^t:=(g,-u)$ for  $(g,u) \in  \cW_{-1}L \oplus \cW_0L$.  
We then have $[\theta,\phi]^t=-[\theta^t,\phi^t]$ for $\theta,\phi \in V_0L^{\smile,\inv}$, 
and $(f\star \theta)^t= f^t\star \theta^t -[f^t,\theta^t] = (-1)^{\deg f\deg \theta} \theta^t\star f^t$  for $f \in \cW_{-1}L$, where the product $ \theta^t\star f^t$ refers to multiplication inside  $L^{\smile,\inv}$.
Since we also have $1^t=1$, the composition of $t$ with the inclusion $V_0L^{\smile,\inv} \into L^{\smile,\inv}$ thus extends to an associative algebra homomorphism $t \co L^{\smile,\inv} \to (L^{\smile,\inv})^{\op}$.
\end{definition}

The map $t$ is necessarily an involution, because its restriction to generators $V_0L^{\smile,\inv}$ is so. On the associated graded pieces $\gr^V_pL^{\smile,\inv}$, multiplicativity of $t$ ensures that it acts as multiplication by $(-1)^{p+1}$.
We thus have a filtered involutive DGLA $(\breve{L},V,i)$ satisfying the conditions of \S \ref{quantsn}, taking $i$ to be $-t$ and $[-,-]$ to be the commutator bracket.

\begin{examples}\label{Dinvex}
 For a cofibrant $R$-CDGA $A$ with perfect cotangent complex, the conditions imposed on $\bigoplus_{j\ge 0} \cW_{j-1}L$ above are satisfied by
 the dg Poisson algebra  $\Pol(A,-1)$ given by  $\bigoplus_{j\ge 0}\HHom_A(\CoS^j_A\Omega^1_{A/R},A)$, where $\CoS$ denotes cosymmetric powers. Then $L=\prod_{j\ge 0} \cW_{j-1}L$ is the algebra $\widehat{\Pol}(A,-1)$ of multiderivations from  \cite[Definition \ref{poisson-bipoldef}]{poisson}. Since $\mmc(\prod_{j\ge 2} \cW_{j-1}L)$ is the space of $(-1)$-shifted Poisson structures on $A$, applying the construction of \S \ref{quantsn} to $\sD^{\inv}_A:=\widehat{\Pol}(A,-1)^{\smile,\inv}$ gives us  
 a canonical space of self-dual twisted quantisations.
 
The conditions are also satisfied if we take $A$ to be an $[m]$-diagram (a string of morphisms) which is cofibrant and fibrant in the injective diagram model structure, by \cite[Definition \ref{poisson-calcTOmegalemma}]{poisson}. That gives sufficient functoriality to 
 construct 
 $\sD(\fX)^{\inv}$
 for derived DM $n$-stacks $\fX$ with perfect cotangent complexes, as derived global sections of $\widehat{\Pol}(-,-1)^{\smile,\inv}$ over a category of \'etale diagrams as in \cite[\S \ref{poisson-hgpdsn}]{poisson}.
 
 Replacing $\Omega^1$ with EFC and $\C^{\infty}$ cotangent complexes gives the corresponding constructions for cofibrant EFC and $\C^{\infty}$ DGAs, and hence for derived analytic and $\C^{\infty}$  DM $n$-stacks in the sense of \cite{DStein}. 
 
 The construction also works in more general tensor categories, including double complexes. Using stacky CDGAs (a.k.a. graded mixed CDGAs) then taking sum-product total complexes as in \cite[\S \ref{poisson-Artinsn}]{poisson} similarly gives a $BD_1^{\pre}$-algebra (in the notation of Example \ref{quantinterpretnex}.(\ref{BD1ex})) $\sD^{\inv}(\fX)$
 for derived Artin $n$-stacks $\fX$ (algebraic, analytic or $\C^{\infty}$) with perfect cotangent complexes.
\end{examples}

\subsubsection{Characterisations of \tps{$\sD^{\inv}$}{Dinv} and its modules}

The following extends \cite[Definition 2.2.1]{ginzburgLectDmods} to the derived setting:
\begin{definition}
A presheaf $(\sD',F)$ of twisted differential operators (TDO) on a derived DM $n$-stack $X$ (algebraic, analytic or $\C^{\infty}$)  is an exhaustively  filtered
presheaf $\sD'$  of $BD_1^{\pre}$ algebras (see Example \ref{quantinterpretnex}.(\ref{BD1ex})) on the site of affine objects \'etale over $X$, with $F_{-1}\sD'=0$, together with a quasi-isomorphism of graded $P_1^{ac}$-dgas 
$\phi \co \gr^F \sD' \to  \Pol(\sO_X,-1)$,  the complex 
of $(-1)$-shifted multiderivations as in Examples \ref{Dinvex} 
(algebraic over a chosen base, EFC or $\C^{\infty}$, respectively).
\end{definition}

\begin{remarks}
The presheaf $\sD'$ is automatically a hypersheaf, since $\gr^F\sD'$ is quasi-coherent.

  When $\oL\Omega^1_{\sO_X}$ is perfect, note that  $\gr^F\sD'$ is quasi-isomorphic to the  symmetric algebra of the tangent complex. 
  
  Although the definition of TDOs depends on the relevant notion of cotangent complex in analytic and smooth settings, note that $\sD'$ is just assumed to be a hypersheaf of abstract $BD_1^{\pre}$-algebras, with no functional calculus. Although $\gr^F_0\sD' \simeq \sO$ is then unwieldy as an abstract commutative algebra,  $\sD'$  still behaves sufficiently well to establish   Proposition \ref{Dinvuniqueprop} below because it only relies on the behaviour of $\gr^F\sD'$ as a commutative algebra over $\gr^F_0\sD'$. 
\end{remarks}

\begin{examples}
 The prototypical example of a TDO is the ring $\sD_X$ of differential operators, equipped with the order filtration $F$. Another example is given by the ring $\sD^{\inv}_X:= \widehat{\Pol}(\sO_X,-1)^{\smile,\inv}$ of Examples \ref{Dinvex}, equipped with the filtration $F_p=V_{p-1}$.
\end{examples}

\begin{definition}\label{Addef}
The adjoint action $\Ad$ of $\oR\map(X,\bG_m)$ on the Dold--Kan denormalisation of
any TDO $\sD'$ on $X$ is induced by the homomorphisms $\Ad_{\lambda}(\theta):= \lambda\theta \lambda^{-1}$ for $\lambda \in \z_0\sO_X^{\by}$ and $\theta \in \sD'$. 

Explicitly, a model for the simplicial group $\oR\bG_m(\sO_X)$ is given by $n \mapsto \z_0( (F_0\sD') \ten \Omega^{\bt}(\Delta))^{\loc,\by}$, where $\loc$ denotes localisation over $\H_0\sO_X$. This group has an adjoint action on the corresponding localisation $\z_0(\sD' \ten \Omega^{\bt}(\Delta))^{\loc}$  along  $\z_0(F_0\sD' \ten \Omega^{\bt}(\Delta))  \to \H_0\sO_X$, which is weakly equivalent to the Dold--Kan denormalisation of $\sD'$. (Left and right localisations agree in this case, because commutators act nilpotently.)  
\end{definition}

\begin{definition}
 Given a TDO $(\sD',F,\phi)$, define the opposite TDO $(\sD',F,\phi)^{\op}$ to be given by the opposite algebra $((\sD')^{\op},F)$ together with the composite  quasi-isomorphism $(\sD')^{\op} \xra{\phi^{\op}} \Pol(\sO_X,-1)^{\op} \xra{(-1)^w} \Pol(\sO_X,-1)  $, where $(-1)^w$ is the Poisson algebra isomorphism which multiplies terms of weight $m$ (i.e. $m$-vectors) by $(-1)^m$.
 \end{definition}

Simplicially localising the category of TDOs on $X$ at the class of filtered quasi-isomorphisms gives an $\infty$-category, and the space $\TDO(X)$ of TDOs on $X$ is then the simplicial set given by the nerve of its maximal $\infty$-subgroupoid. We then have:
\begin{proposition}\label{Dinvuniqueprop}
For a derived DM $n$-stack  $X$ (algebraic, analytic or $\C^{\infty}$) with perfect cotangent complex:
 \begin{enumerate}
  \item\label{one} the space $\TDO(X)$ is a torsor for $\mmc(F^1\DR(X)[1])$; 
  
  \item the action of $\mmc(F^1\DR(X)[1])$ on $\TDO(X)$ is $C_2$-equivariant, for the actions generated respectively by  $- 1$  acting multiplicatively on $F^1\DR(X)$ and by the endofunctor $\op$ of the category of TDOs;
  
  \item the action $\Ad$ of $\oR\map(X,B\bG_m)$ on $\TDO(X)$ is homotopic to the composite of the group homomorphism $d\log \co \sO_X^{\by} \to   F^1\DR(X)[1]$ with the torsor action (\ref{one}).
  \end{enumerate} 
\end{proposition}
\begin{proof}
Via the Rees construction $\Rees(\sD',F) :=  \bigoplus F_i\sD'\hbar^i$, $BD_1^{\pre}$-algebras over $R$ with non-negatively indexed increasing filtrations correspond to non-negatively  graded $\hbar$-torsion-free $BD_1$-algebras over $R[\hbar]$, for $\hbar$ having weight $1$. The space $\TDO(X)$ thus corresponds to the space of graded $BD_1$-algebras $E$ over $R[\hbar]$ equipped with a graded $P_1^{ac}$-algebra quasi-isomorphism $E\ten^{\oL}_{R[\hbar]}R \simeq \Pol(\sO_X,-1)$.

Since $F_{-1}\sD'=0$, the filtration is complete, so $\Rees(\sD',F)$ is the $\bG_m$-equivariant limit of the $BD_1/\hbar^k$-algebras $\Rees(\sD',F)/\hbar^k$, and we can regard this as a deformation problem. We can then adapt \cite[\ref{DQpoisson-defPprop} and Remark \ref{DQpoisson-filteredoperadrmk2}]{DQpoisson},  relaxing the connectivity property using the argument of \cite[Corollary \ref{coddt-cofalgdefcor}]{coddt}; together, these give us a sequence of homotopy fibre sequences
\[
 \Alg^{\bG_m}(BD_1/\hbar^{k+1})_P \to \Alg^{\bG_m}(BD_1/\hbar^k)_P \to \mmc(\oR\DDer_{P_1^{ac}}^{\bG_m}(P,\hbar^kP)[1])
\]
for any $\bG_m$-equivariant $P_1^{ac}$-algebra $P$ satisfying $\H_{\ll 0}\cW_iP=0$ for all $i$. Here $\Alg^{\bG_m}(BD_1/\hbar^k)_P$ denotes the space of $\bG_m$-equivariant $BD_1/\hbar^k$-algebras $E$ with a quasi-isomorphism $E\ten^{\oL}_{R[\hbar]/\hbar^k}R \simeq P$, and the complex of derivations is regarded as an abelian DGLA.

When $P=\Pol(\sO_X,-1)$, the calculations of \cite[Proposition \ref{DQpoisson-weightprop}]{DQpoisson} imply that the complex of derivations is acyclic for $k \ge 2$, so this reduces to a first order deformation problem. In other words, 
\[
 \TDO(X) \simeq \oR\Gamma(X, \Alg^{\bG_m}(BD_1/\hbar^2)_{\Pol(\sO_X,-1)}.  
\]
Since $\TDO(X)$ is non-empty (elements include the ring of differential operators and $\sD^{\inv}$), the fibre sequence for $k=1$ then implies that $\TDO(X)$ is a torsor for $\oR\Gamma(X,\mmc(\oR\DDer_{P_1^{ac}}^{\bG_m}(P,\hbar P))$. 

We have a morphism from $F^1\DR(\sO_X)[1]$ to the complex of  Poisson $L_{\infty}$-derivations of $\Pol(\sO_X,-1)$ with $\bG_m$-weight $-1$ (cf. \cite[Definition \ref{DQpoisson-PoissonLinftydef}]{DQpoisson}) given by sending $adb_1\ldots db_p$ to the $p$-derivation $(v_1, \ldots,v_p) \mapsto a[b_1,v_1]\cdots[b_p,v_p]$. By \cite[Lemma \ref{DQpoisson-polyvectorlemma}]{DQpoisson}, the resulting map 
\[
\mu \co  F^1\DR(\sO_X)[1] \to \oR\DDer_{P_1^{ac}}^{\bG_m}( \Pol(\sO_X,-1),\hbar \Pol(\sO_X,-1) )
\]
gives quasi-isomorphisms on associated graded pieces  $\gr_F^p$, so is a quasi-isomorphism. Thus $\TDO(X)$ is a torsor for $\mmc(F^1\DR(X)[1])$.

It remains to prove the second and third statements.

The action of 
$\hbar\alpha \in \oR\Gamma(X,\mmc(\oR\DDer_{P_1^{ac}}^{\bG_m}(P,\hbar P))$ on  $\Alg^{\bG_m}(BD_1/\hbar^2)_P$ sends $E$ to $E^{\alpha}:= (E, \delta_E +\hbar\alpha)$. Then $(E, \delta_E +\hbar\alpha)^{\op}= (E^{\op}, \delta_{E^{\op}} +\hbar\alpha^{\op})$, with  $\alpha^{\op}$ the derivation from $P^{\op}$ to $P^{\op}$ corresponding to $\alpha$. Since $\alpha$ has odd weight, $\alpha^{\op}$ corresponds to $-\alpha$ under the isomorphism $(-1)^w \co P \to P^{\op}$, so  $(E^{\alpha})^{\op}= (E^{\op})^{-\alpha}$, giving the second statement.

For the final statement, observe that for $E \in  \Alg^{\bG_m}(BD_1/\hbar^2)_P$ and an invertible element $\lambda$ in $E$ of $\bG_m$-weight $0$, we have $\lambda \theta \lambda^{-1} = \theta + [\lambda,\theta]\lambda^{-1}$ for all $\theta \in E$, with the operation $[\lambda,-]\lambda^{-1}$ squaring to zero, so $\Ad_{\lambda}= \exp([\lambda,-]\lambda^{-1})$. 
The derivation $[\lambda,-]\lambda^{-1}$ is given by $\mu(d\log \lambda)$ under the quasi-isomorphism $\mu$ above, which implies the third statement.
\end{proof}

\begin{corollary}\label{Dinvuniquecor}
For a derived DM $n$-stack  $X$ with perfect cotangent complex, the action of $\mmc(F^1\DR(X))$ on $\sD^{\inv} \in \TDO(X)$ gives a $C_2$-equivariant equivalence $\mmc(F^1\DR(X)[1]) \simeq \TDO(X)$ (intertwining the actions of $-1$ and $\op$).

In particular, up to coherent homotopy, $\sD^{\inv}_X$ is unique TDO $(\sD',F,\phi)$ with anti-involution $t \co (\sD',F,\phi) \to (\sD',F,\phi)^{\op}$.
\end{corollary}
\begin{proof}
 The first statement follows because $\sD^{\inv}$ carries an anti-involution, so is a fixed point for the $C_2$ action. The second statement then follows because  $\oR\Gamma(C_2,F^1\DR(X))\simeq 0$ for the multiplication action by $-1$.
 \end{proof}


\begin{definition}
 Given a line bundle $\sM$ on $X$, define the TDO $\sD_X(\sM)$ to be the twist of $\sD_X$ by the corresponding $\bG_m$-torsor $P(\sM)$ via the adjoint action $\Ad$ of Definition \ref{Addef}, i.e.   $(\sD_X\by P(\sM))/\bG_m$. This is equivalent to the ring of differential operators on $\sM$.
 
\end{definition}

\begin{definition}
 Roots of unity act trivially under the  action $\Ad$ of Definition \ref{Addef}, in particular giving a square root $\sqrt{\Ad}$ of the action $\Ad$ via the \'etale Kummer sequence $1 \to \mu_2 \to \bG_m \xra[{[2]}]{} \bG_m \to 1$.
 
 We then define $\sD_X(\sqrt{\sM})$ to be the twist of $\sD_X$ by $P(\sM)$ via the square root action $\sqrt{\Ad}$.
 \end{definition}
Note that if $\sL^{\ten 2}\simeq \sM$ then $\sD_X(\sL) \simeq \sD_X(\sqrt{\sM})$, but that the latter is defined irrespective of the choice or existence of such a square root. We can think of $\sD_X(\sqrt{\sM})$ as the ring of differential operators on the $\mu_2$-twisted line bundle $\sqrt{\sM}$, and the construction generalises to more general $\mu$-twisted line bundles since $\Ad$ factors through $\bG_m/\mu_{\infty}$.

\begin{corollary}\label{Dinvsqrtcor}
If  $c(\sM) \in \H^2(F^1\DR(X))$ is the Chern class of a line bundle $\sM$  on $X$ equipped with  a right $\sD$-module structure, then under the action of Proposition \ref{Dinvuniqueprop} we have $\sD_X^{\inv} \simeq \sD_X^{c(\sM)/2} \simeq \sD_X(\sqrt{\sM})$. 
\end{corollary}
\begin{proof}
 Since $\sM$ has a right $\sD$-module structure, we have $\sD_X^{\op} \simeq \sD_X(\sM)$, which by the last part of Proposition \ref{Dinvuniqueprop} corresponds to $\sD_X^{c(\sM)}$. We have $\sD_X\simeq (\sD^{\inv}_X)^{\alpha}$ for some $\alpha \in \H^2(F^1\DR(X))$ by the first part of that proposition, and then $\sD_X^{\op}\simeq (\sD^{\inv}_X)^{-\alpha}$ by Corollary \ref{Dinvuniquecor}, so $c(\sM)=-2\alpha$ and $\sD_X^{\inv} \simeq \sD_X^{c(\sM)/2}$. The second equivalence then follows by applying the last part of Proposition \ref{Dinvuniqueprop} again.
 \end{proof}

\begin{definition}\label{Bmu2actDmoddef}
 Roots of unity in $A$ lie in the centre of $\sD_A$, so multiplication by elements of $\mu_2(A)$ defines a natural isomorphism of the identity endofunctor of the dg category $\Mod_{\dg}(\sD_A)$ of $\sD_A$-modules. Hypersheafifying this gives an action of $B\mu_2$ on $\Mod_{dg}(\sD_X)$, and given a class $[\sP] \in \H^2(X,\mu_2)$ we define the hypersheaf  $\Mod_{dg}(\sD_X)^{[\sP]}$ of dg categories on $X$ to be the twist $(\Mod_{\dg}(\sD_X)\by \sP)/^hB\mu_2  $ of  $\Mod_{dg}(\sD_X)$ by $\sP$.
 \end{definition}
 
\begin{definition}
 Given a line bundle $\sM$ on $X$, define the orientation class $o(\sM) \in \H^2(X,\mu_2)$ to be the image of $[\sM]$ under the connecting homomorphism $\H^1(X,\bG_m) \to \H^2(X,\mu_2)$ of the Kummer sequence.
\end{definition}

\begin{corollary}\label{moritaDinvcor}
For a derived DM $n$-stack  $X$ with perfect cotangent complex, and a line bundle $\sM$ on $X$ with right $\sD$-module structure,    the dg category $\Mod_{dg}(\sD^{\inv}_X)$ is quasi-equivalent to the twist $\Mod_{dg}(\sD_X)^{o(\sM)}$ of the dg category of $\sD$-modules by the orientation class of $\sM$. 
\end{corollary}
\begin{proof}
 By Corollary \ref{Dinvsqrtcor}, $\sD_X^{\inv} \simeq \sD_X(\sqrt{\sM})$. The adjoint action of $\bG_m$ on $\Mod_{dg}(\sD_X)$, induced by the action $\Ad$ on $\sD_X$, is homotopically trivial because multiplication by $\lambda$ gives a natural isomorphism from the identity functor to the endofunctor $\Ad_{\lambda}^*$.  The action $\sqrt{\Ad}$ of $\bG_m$ on  $\Mod_{dg}(\sD_X)$ thus factors through the homotopy quotient $B\mu_2$ of $[2] \co \bG_m \to \bG_m$ via its action from Definition \ref{Bmu2actDmoddef}, so $\Mod_{dg}( \sD_X(\sqrt{\sM}))\simeq \Mod_{dg}( \sD_X)^{o(\sM)}$, as required. 
\end{proof}
Explicitly, the proof implies that the derived Morita equivalence is given by the $\mu_2$-twisted $\sD^{\inv}_X -\sD_X$-bimodule $\sDiff(\sO_X,\sqrt{\sM}) \simeq \sqrt{\sM}\ten_{\sO_X}\sD_X$ of differential operators from $\sO_X$ to the $\mu_2$-twisted line bundle $\sqrt{\sM}$. 

There is a dg functor $\oR\sHHom(-,\sD')_{\sD'}\co \Mod_{dg}(\sD')^{\op} \to \Mod_{dg}((\sD')^{\op})$ for any TDO $\sD'$, which  gives rise to a contravariant endofunctor of $\Mod_{dg}(\sD^{\inv}_X)$ when combined with the anti-involution of $\sD^{\inv}_X$. Under the equivalence of Corollary \ref{moritaDinvcor}, this corresponds to the endofunctor $\oR\sHHom(-,\sDiff(\sO_X,\sM))_{\sD}$ of $\Mod_{dg}(\sD_X)$, defined using the equivalence $\sD_X^{\op} \simeq \sD_X(\sM)$.

\subsection{Self-dual \tps{$(-1)$}{(-1)}-shifted deformation quantisations}\label{invBVsn}

As in Examples \ref{Dinvex}, we have a filtered involutive DGLA $\sD^{\inv}(X):= \oR\Gamma(X,\widehat{\Pol}(\sO_X,-1)^{\smile,\inv})$ satisfying the conditions for $\breve{L}$ from Setting \ref{breveLsetting}. Defining $Q^{\inv}\widehat{\Pol}(X,-1)$ to be the corresponding involutive DGLA $\tilde{L}$ of Definition \ref{tildeLdef} with its filtration $\tilde{F}$, and writing $Q^{\inv}\widehat{\Pol}(X,-1)^{sd}:=Q^{\inv}\widehat{\Pol}(X,-1)^{*=\id}$, we have
\[
\tilde{F}^pQ^{\inv}\widehat{\Pol}(X,-1)^{sd} = \{\alpha(\hbar) \in \prod_{i \ge p-1} \hbar^iF_{i+1}\sD^{\inv}(X) ~:~ \alpha(\hbar)^t=- \alpha(-\hbar) \}.
\]
The corresponding space of self-dual quantisations is then
$Q\cP(X,-1)^{sd}:= \mmc(\tilde{F}^2Q^{\inv}\widehat{\Pol}(X,-1)^{sd})$.

As in Example \ref{quantex}.(\ref{quantexvanish}),
the content of Corollary \ref{exactquantbrhcor} in this case is thus to give an equivalence between the space of exact $(-1)$-shifted symplectic structures on $X$ and the space  $Q\cP^D(X,-1)^{sd,\nondeg}$ given by
 \[
\mmc(\cocone( \eps\hbar \pd_{\hbar} \co \tilde{F}^2Q^{\inv}\widehat{\Pol}(X,-1)^{sd} \to \tilde{F}^0Q^{\inv}\widehat{\Pol}(X,-1)^{sd}\eps)).
 \]

When $X$ admits a line bundle $\sM$ with a  right $\sD$-module structure, Corollary \ref{Dinvsqrtcor} allows us to interpret this along similar lines to \S \ref{interpretDQsn} as a  twisted space of $BD_0$-algebras with $\hbar\pd_{\hbar}$-derivation satisfying a self-duality property. 

\subsubsection{\tps{$\sD\brh$}{D[[h]]}-modules associated to quantisations}

\begin{definition}
  Given a TDO $\sD'$,  an element $\Delta \in \mc(\prod_{i \ge 1} \hbar^iF_{i+1}\sD')$ and a left $\sD'$-module $M$ define ${}_{\Delta}\! M\brh$ to be the complex  $M\brh$ with differential $\delta_{\Delta}^l\co m \mapsto \delta m + \Delta \cdot m$.
  
In particular, since $\sD'$ is a $\sD'$-bimodule, ${}_{\Delta}\!\sD'\brh$ is a   right $\sD'\brh$-module.
\end{definition}
 
\begin{definition}
 Define a formal deformation $\tilde{M}$ of a right $\sD'$-module $M$ to be an inverse system $\{\tilde{M}(n)\}_n$, for $\tilde{M}(n)$ a right  $\sD'[\hbar]/\hbar^{n+1}$-module compatibly with the structure maps, such that $M(0)=M$ and the maps $M(n)\ten^{\oL}_{\sD']\hbar]/\hbar^{n+1}}\sD' \to \sD'$ are quasi-isomorphisms. 
\end{definition}
 
 \begin{definition}
 Write  ${}^t\!\sD^{\inv}$ for $\sD^{\inv}$ equipped with the right  $(\sD^{\inv})^{\ten 2}$-module structure given by $a\cdot (b\ten c):=(-1)^{\deg a \deg b}b^tac$.
 \end{definition}
 
 \begin{definition}
 Given a right $(\sD')^{\ten 2}$-module $N$, write  ${}^{-\hbar}\!N$ for $N\brh$ equipped with the right  
 $(\sD^{\inv}\brh)^{\ten 2}$-module structure given by $n(\hbar)\cdot (a(\hbar)\ten b(\hbar)):= n(\hbar) \cdot_N (a(-\hbar)\ten b(\hbar))$, where $\cdot_N$ denotes the $\hbar$-linear extension of the multiplication on $N$.
 \end{definition}

 \begin{definition}
 Given a right $\sD'$-module $M$ equipped with a symmetric right  $(\sD')^{\ten 2}$-linear morphism $\tau \co M\ten M \to N$ for a $(\sD')^{\ten 2}$-module $N$, we say that a sesquilinear formal deformation of $(M,\tau)$ is a formal deformation $\tilde{M}= \{\tilde{M}(n)\}_n$ of the right $\sD'$-module underlying  $M$, together with a compatible system of right  $(\sD'[\hbar]/\hbar^{n+1})^{\ten 2}$-linear maps $\tilde{\tau}\co \tilde{M}(n)\ten \tilde{M}(n)\to {}^{-\hbar}\!N[\hbar]/\hbar^{n+1}$.

 We say that an $\hbar^2\pd_{\hbar}$-connection $\phi \co \tilde{M} \to \tilde{M}$ (i.e. $\phi(ma) = \phi(m)a +\hbar^2 m \frac{\pd a}{\pd \hbar})$) is orthogonal with respect to $\tilde{\tau}$ if
 \[
  \hbar^2\pd_{\hbar}\tilde{\tau}(u,v)=  -\tilde{\tau}( \phi u,v) + \tilde{\tau}(u, \phi v)  \in N\brh. 
 \]
 \end{definition}
 Note that if $\tau$ is non-degenerate in the sense that the induced map $M \to \oR\HHom(M,N)_{\sD'}$ is a quasi-isomorphism, then so are the maps $ \tilde{M}(n) \to \oR\HHom(\tilde{M}(n),{}^{-\hbar}\!N[\hbar]/\hbar^{n+1} )_{\sD'[\hbar]/\hbar^{n+1}}$.

\begin{lemma}\label{quantDmodlemma} 
If $X$ is a derived DM stack (algebraic, analytic or $\C^{\infty}$)  with perfect cotangent complex generated in chain degrees $0,1$, then sending $\Delta$ to ${}_{\Delta}\sD^{\inv}_X\brh$ gives a natural map from the space $Q\cP(X,-1)^{sd}$ to the space of formal sesquilinear deformations of $\sD_X^{\inv}$ as a right $\sD_X^{\inv}$-module with non-degenerate pairing $\tau \co (\sD_X^{\inv})^{\ten 2} \to {}^t\!\sD^{\inv}_X$  given by $(a,b) \mapsto a^tb$.

The same construction gives a natural map from the space $Q\cP^D(X,-1)^{sd}$ to the space of triples $(\sE,\upsilon,D)$ for  formal sesquilinear deformations $(\sE,\upsilon)$ of $(\sD_X^{\inv},\tau)$ equipped with orthogonal $\hbar^2\pd_{\hbar}$-connections.
\end{lemma}
\begin{proof}
 We have a natural inclusion $(\prod_{i \ge 1} \hbar^iF_{i+1}\sD^{\inv})^{sd} \to (\hbar \sD^{\inv}\brh)^{sd}$ of pro-nilpotent DGLAs and hence $Q\cP(X,-1)^{sd} \to \mmc(\hbar \sD^{\inv}\brh)^{sd}$. The conditions ensure that the DGA $\sD^{\inv}_X$ is homologically bounded below (by the number of generators of $\bL^X$ in degree $1$), so a simple case of \cite[Corollary \ref{coddt-cofalgdefcor}]{coddt} allows us to interpret $\mmc(\hbar \sD^{\inv}\brh)$ as the space of formal deformations of the right  $\sD_X^{\inv}$-module $\sD_X^{\inv}$, since  multiplication on the left identifies   $\sD^{\inv}$  with  the complex of right $\sD^{\inv}$-module derivations (i.e. endomorphisms) of $\sD^{\inv}$.
 
Now,  $(\hbar \sD^{\inv}\brh)^{sd}= \{f(\hbar) \in \hbar \sD^{\inv}\brh~:~ f(\hbar)^t= -f(-\hbar)\}$, and 
sending $f$ to the map $a \mapsto  f a$ allows us to regard this as the complex of  linear endomorphisms $\theta$ of the right $\sD^{\inv}\brh$-module $\sD^{\inv}\brh$ which are trivial modulo $\hbar$ and orthogonal with respect to  the sesquilinear pairing $\tau_{\hbar} \co (a(\hbar),b(\hbar)) \mapsto a(-\hbar)^tb(\hbar)$
in the sense that 
\[
 \tau_{\hbar}(\theta(a),b)+ 
   \tau_{\hbar}(a,\theta(b))=0. 
\]

A similar argument thus gives an equivalence between $\mmc(\hbar \sD^{\inv}\brh)^{sd}$ and the space of formal sesquilinear deformations of $(\sD_X^{\inv},\tau)$.

For $Q\cP^D(X,-1)^{sd}$, we then proceed as in Lemma \ref{epslemmaquant}, identifying 
\[
\mmc(\cocone(\hbar\pd_{\hbar} \eps \co \hbar \sD^{\inv}\brh \to \hbar^{-1} \sD^{\inv}\brh\eps))^{sd}
\]
with the space of triples $(\sE,\upsilon,\phi)$ for $(\sE,\upsilon)$ a formal sesquilinear deformation  of $(\sD_X^{\inv},\tau)$, and $\phi \co \sE \to \sE$  an orthogonal $\hbar^2\pd_{\hbar}$-connection. Explicitly, a Maurer--Cartan element $\Delta + D\eps$ is sent to the module   ${}_{\Delta}\!\sD'\brh$ with connection $\phi(a):= \hbar^2\frac{\pd a}{\pd \hbar}+\hbar Da$.
%
 \end{proof}

As an immediate consequence of Lemma \ref{quantDmodlemma} and Corollary \ref{moritaDinvcor}, we have the following, writing $f \mapsto f^t$ for the isomorphism $\sD_X^{\op} \to \sD_X(\sM)$ and regarding $\sDiff(\sO_X, \sM) $ as a $(\sD_X(\sM),\sD_X)$-bimodule.

\begin{corollary}\label{quantDmodcor}
Given $X$ as in  Lemma \ref{quantDmodlemma} and a line bundle $\sM$ on $X$ with right $\sD_X$-module structure, 
sending $\Delta$ to ${}_{\Delta}\sDiff(\sO_X, \sqrt{\sM})\brh$ gives a natural map from the space $Q\cP(X,-1)^{sd}$ to the space of formal sesquilinear deformations of $\sDiff(\sO_X, \sqrt{\sM}) $ as a $\mu_2$-twisted right $\sD_X$-module with non-degenerate pairing $\tau \co (\sDiff(\sO_X, \sqrt{\sM}))^{\ten 2} \to \sDiff(\sO_X, \sM)$  given by $(a,b) \mapsto a^tb$.

The same construction gives a natural map from the space $Q\cP^D(X,-1)^{sd}$ to the space of triples $(\sE,\upsilon,\phi)$ for  formal sesquilinear deformations $(\sE,\upsilon)$ of $(\sDiff(\sO_X, \sqrt{\sM}),\tau)$ equipped with orthogonal $\hbar^2\pd_{\hbar}$-connections $\phi\co \sE \to \sE$.
\end{corollary}

Writing $\sK_X:= \det \oL\Omega^1_X$, we then have:
\begin{corollary}\label{quantDmodcor2}
On a derived DM $n$-stack $X$  (algebraic, analytic or $\C^{\infty}$), to each exact $(-1)$-shifted symplectic structure there is a canonically associated 
formal sesquilinear deformation  with  orthogonal $\hbar^2\pd_{\hbar}$-connection of  the  $\mu_2$-twisted right $\sD_X$-module $\sDiff(\sO_X, \sqrt{\sK}_X)$ with non-degenerate pairing $\tau(a,b) = a^tb$.
\end{corollary}
\begin{proof}
Proposition \ref{exactcompprop} and Corollary \ref{exactquantcor} give us equivalences $\Sp^{\ex}(X,-1) \simeq \cP^D(X,-1)^{\nondeg} \simeq Q\cP^D(X,-1)^{sd}$. Since $X$ is DM and carries a $(-1)$-shifted symplectic structure, the conditions of Lemma \ref{quantDmodlemma} are satisfied and $\sK_X$ is naturally a right $\sD$-module. Thus Corollary \ref{quantDmodcor} associates a sesquilinear deformation with orthogonal connection to each exact  $(-1)$-shifted symplectic structure. 
\end{proof}
As in Examples \ref{poissonex}, note that \cite[Proposition 3.2]{KinjoParkSafronov} or
\cite[Proposition 2.2.11]{HHR1} give a canonical section of the map $\Sp^{\ex}(X,-1) \to \Sp(X,-1)$ whenever the base is a field, so the corollary canonically associates a  formal sesquilinear deformation  with  orthogonal connection to every $(-1)$-shifted symplectic structure.
 
\begin{remark}\label{pushforwardpairingrmk}
Under the conditions of Lemma \ref{quantDmodlemma}, if $i \co X \to U$ is a derived closed immersion of codimension $d$, with $U$ smooth, then $\sK_X \simeq i^!\sK_U[d] := \oR\sHHom_{i^{-1}\sO_U}(\sO_X,i^{-1}\sK_U)[d]$. Pushforward of right $\sD$-modules is given by $i_!\sE:= i_*(\sE\ten^{\oL}_{\sD_X}\sDiff_{i^{-1}\sO_U}(i^{-1}\sO_U,\sO_X))$, and then any sesquilinear pairing $ (\sE)^{\ten 2} \to \sDiff_{\sO_X}(\sO_X, \sK_X)\brh$ gives rise to a sesquilinear pairing
\[
(i_!\sE)^{\ten 2} \to  \sDiff_{\sO_U}(\sO_U,\sO_X) \ten^{\oL}_{\sD_X^{\op}}\sDiff_{\sO_X}(\sO_X, \sK_X) \ten^{\oL}_{\sD_X}\sDiff_{\sO_U}(\sO_U,\sO_X)\brh
\]
where we omit $i_*,i^{-1}$ to lighten the notation. 

The left action of $\sD_X^{\op}$ on $\sDiff_{\sO_X}(\sO_X, \sK_X)$ comes from the identification $\sD_X^{\op} \simeq \sD_X(\sK_X)$, and the contravariant functor $\oR\sHHom_{\sO_U}(-,\sK_U)$ similarly gives $\sDiff_{\sO_U}(\sO_U,\sO_X) \simeq \sDiff_{\sO_U}(\sK_X[-d],\sK_U)$, intertwining the respective $(\sD_X,\sD_U)$- and $(\sD_U(\sK_U),\sD_X(\sK_X))$-bimodule structures. We thus have a pairing
\[
 (i_!\sE)^{\ten 2} \to \sDiff_{\sO_U}(\sK_X[-d],\sK_U)\ten^{\oL}_{\sD_X(\sK_X)}\sDiff_{\sO_X}(\sO_X, \sK_X) \ten^{\oL}_{\sD_X}\sDiff_{\sO_U}(\sO_U,\sO_X)\brh,
\]
which maps naturally under composition to give a $\sDiff_U(\sO_U,\sK_U)\brh[d]$-valued sesquilinear pairing $\pi$ on $i_!\sE$. Equivalently, we can write this as an $\hbar$-semilinear morphism
\[
 i_!\sE \to \oR\sHHom(i_!\sE, \sDiff_U(\sO_U,\sK_U)\brh)_{\sD_U\brh}[d]
\]
of right $\sD_U\brh$-modules. Note that if the pairing on $\sE$ is non-degenerate, then this duality morphism is a quasi-isomorphism.

The resolution 
\[
\sF\ten_{\sO_U}\sD_U \la \sF\ten_{\sO_U}\sT_U\ten_{\sO_U}\sD_U \la \sF\ten_{\sO_U}(\L^2\sT_U)\ten_{\sO_U}\sD_U \la \ldots
\]
of any right $\sD_U$-module $\sF$ gives a morphism
\[
\oR\sHHom_{\sO_U}(\sF, \sN\ten_{\sO_U}\sK_U)[-d] \to 
 \oR\sHHom(\sF, \sDiff_U(\sO_U,\sN))_{\sD_U} 
\]
which is a quasi-isomorphism when $\sF$ is perfect as an $\sO_U$-module.

%

If $\sE$ is equipped with an orthogonal $\hbar^2\pd_{\hbar}$-connection $\phi\co \sE \to \sE$, then the induced $\hbar^2\pd_{\hbar}$-connection on $i_!\sE$ is automatically orthogonal with respect to $\pi$. 



\end{remark}

\subsubsection{Derived critical loci: twisted de Rham complexes appear}

 Given a smooth  scheme, or analytic or $\C^{\infty}$ space $Y$  of dimension $m$, and a function $f \co Y \to \bA^1$, we can consider the derived critical locus $\oR\Crit(Y,f)$ of $f$, which is equipped with a canonical exact $(-1)$-shifted symplectic structure $(\omega, \lambda)$. 
 
\begin{definition}
 The structure sheaf of $\oR\Crit(Y,f)$ is the CDGA given by the alternating algebra $\sO_Y[\sT_Y[1]]$, with differential $\delta$ given by contraction with $df$. The term $\lambda \in \z^0(\DR(\oR\Crit(Y,f))/F^2)$ is given by the sum of $-f \in \sO_Y$ and  the generator $\nu$ of the canonical copy of $\Omega^1_Y\ten_{\sO_Y}\sT_Y$ inside $\Omega^1_{\oR\Crit(Y,f)}$. The $(-1)$-shifted symplectic structure $\omega$ is then given by $d\nu$, for the de Rham differential $d$.
 
%

Explicitly, in local co-ordinates $(y_i)$ for $\sO_Y$, with  $(\eta_i)$ the basis for $\sT_Y$ dual to $(dy_i)$, we have $\nu= \sum_i \eta_idy_i$ and so  $\omega= \sum_i dy_i d\eta_i$
and $\lambda = \sum_i \eta_idy_i  -f$, noting that $d\lambda= \sum_i d\eta_idy_i  -df$ and $\delta \lambda = \sum_i \frac{\pd f}{\pd y_i} dy_i =df$.
\end{definition}

For $X:=\oR\Crit(Y,f)$ and $i \co X \to Y$ the natural map, the determinant bundle $\sK_X$ is given by $(i^*\sK_Y)^{\ten 2}$, so has a natural square root $\sL:= i^*\sK_Y$.  Proposition \ref{exactcompprop} and Corollary \ref{exactquantcor} then give us a contractible space $ Q\cP^D(X,-1)^{sd} \simeq Q\cP^D_X(\sL,-1)^{sd}$ of self-dual quantisations with $\hbar\pd_{\hbar}$-derivation.

\begin{proposition}\label{DCritprop}
 On the exact $(-1)$-shifted symplectic scheme $\X:=\oR\Crit(Y,f)$, the essentially unique self-dual quantisation $(\Delta, D)$ with $\hbar\pd_{\hbar}$-derivation 
 is given on the line bundle $\sL:= i^*\sK_Y \cong (\Omega^*_Y,?\wedge df)[\dim Y]$ by setting $\Delta :=\hbar d \in \hbar F_2\sD_X(\sL)$, for the de Rham differential $d$,
 and defining the  $\hbar\pd_{\hbar}$-connection $D \in F_1\sD_X(\sL) \by \hbar^{-1}F_0\sD_X(\sL)$ to act  on $\Omega^p_Y$ as multiplication by  $(\dim Y/2 -p) -\hbar^{-1}f$. 
 \end{proposition}
\begin{proof}
 It suffices to show that we indeed have $(\Delta, D) \in Q\cP^D_X(\sL,-1)^{sd}$. For $\delta \alpha =\alpha \wedge df$ on $\Omega^*_Y$, we have the following relations:
 \[
[\delta, \delta]=0, \quad [\delta, d]=0, \quad [d,d]=0, \quad [\delta,D]= \delta, \quad [d,D]=  d- \hbar^{-1}\delta, 
 \]
 so $[\delta +\hbar d, \delta +\hbar d]=0$ and 
 $
 [\delta + \hbar d,D] 
 = \hbar d = \hbar\pd_{\hbar}(\hbar d). 
 $
Thus  $(\hbar d ,D\eps)$ is indeed a Maurer--Cartan element of 
 \[
 \cocone(\prod_{i \ge 1} \hbar^i F_{i+1}\sD_X(\sL) \xra{\hbar \pd_{\hbar}} \prod_{i \ge -1} \hbar^i F_{i+1}\sD_X(\sL)\eps),
 \]
 so it only remains to show self-duality.

Via the identification $\sHom_{i^{-1}\sO_Y}(\sO_X,i^{-1}\sK_Y)[m]\cong \sK_X $ (where $m=\dim Y$),  the anti-involution $t$ on $\sD_X(\sL)$ is induced by the pairing
\[
 \sL \cong \sHom_{\sO_Y}(\sL,\sK_Y)[m]
\]
given by combining the isomorphisms $\Omega^p_Y \cong \sHom_{\sO_Y}(\Omega^{m-p}_Y,\sK_Y)$ coming from the cup product. Up to sign, $t$ thus combines the corresponding isomorphisms $\sDiff_Y(\Omega^p_Y, \Omega^q_Y) \cong \sDiff_Y(\Omega^{m-q}_Y, \Omega^{m-p}_Y)$.

Now $\delta$ is $\sO_Y$-linear, so $\delta^t$ is its $\sO_Y$-linear dual, up to sign.  Since  $(\alpha \wedge df) \wedge \beta = \pm \alpha \wedge (\beta \wedge df)$, we thus have $\pm \delta^t = \delta$. We know that $t$ acts as multiplication by $(-1)^j$ on $\gr^F_j \sD_X(\sL)$, so by considering principal symbols the signs must resolve as $\delta^t = -\delta$, since $\delta$ has order $1$ as a differential operator of $\sO_X$-modules.

The degree operator $\deg_{\Omega}$ multiplying $\Omega^p_Y$ by $p$ is also $\sO_Y$-linear, and the dual of the operator $p=\deg_{\Omega}$ on $\Omega^p_Y$ is the operator $p = (m-\deg_{\Omega})$ on $\Omega^{m-p}_Y$, so we similarly have $\deg_{\Omega}^t= m-\deg_{\Omega}$.

The de Rham differential $d \co \Omega^p_Y \to \Omega^{p+1}_Y$ has order $1$ as a differential operator of $\sO_Y$-modules. 
Its dual $d^t \co \Omega^{m-p-1}_Y \to \Omega^{m-p}_Y$ is therefore  determined by the property that for $\alpha \in \Omega^p_Y$ and $\beta \in  \Omega^{m-p-1}_Y$, we have
\[
  d^t(\beta) \wedge \alpha = \pm \beta \wedge d\alpha   \pm  (\beta\wedge  \sigma_1(d)(\alpha)_{(1)})^{ \sigma_1(d)(\alpha)_{(2)}} \in \sK_Y
\]
in Sweedler notation, 
where $\sigma_1(d) \co \Omega^p_Y \to \Omega^{p+1}_Y\ten_{\sO_Y}\sT_Y$ is the principal symbol of $d$, and $(-)^{\xi}$ denotes the right action of $\xi \in \sT_X$ on $\sK_Y$. In this case, $\sigma_1(d)$ is  dual to the cup product $\Omega^p_Y \ten_{\sO_Y} \Omega^1_Y \to \Omega^{p+1}_Y$, and $\gamma^{\xi}= -d(\gamma \lrcorner \xi)$. 
Thus 
\[
  d^t(\beta) \wedge \alpha = \pm \beta \wedge d\alpha    \pm d(\beta \wedge \alpha), 
\]
so the signs must align to give  $d^t =\pm d$. Again, the choice of sign is then forced by $t$ acting as multiplication by $(-1)^j$ on $\gr^F_j \sD_X(\sL)$; since $d$ has order $2$ as a differential operator of $\sO_X$-modules, we must have $d^t=d$.



For the sesquilinear anti-involution $*$ on $\sD(\sL)\brh$ given by $a(\hbar)^*:= -a^t(-\hbar)$, we thus have 
\begin{align*}
 D^* &= 
 -\frac{m}{2} + \deg_{\Omega}^t + (-\hbar)^{-1}f = -\frac{m}{2} + (m- \deg_{\Omega}) -\hbar^{-1}f
 = D,\\
\text{and }\quad  \Delta^*& = -(-\hbar)d^t= \hbar d = \Delta.   \qedhere
 \end{align*}
\end{proof}

\begin{remark}
We can simplify the operations in Proposition \ref{DCritprop} by introducing powers of $\hbar$. If we let
$\sL_{\hbar}':=(\hbar^{m/2-*}\Omega^*_Y\brh,?\wedge \hbar^{-1}df)[m]$, then multiplication by $\hbar^{i-m/2}$ in degree $i$ gives an isomorphism $\sL\brh \to  \sL_{\hbar}'$, with $\Delta$ acting on $\sL_{\hbar}'$ as the de Rham differential, and $D$ acting as the 
$\hbar\pd_{\hbar}$-connection $\hbar\pd_{\hbar} -\hbar^{-1}f$. The $\sO_X$-module structure on $\sL'_{\hbar}$ combines the obvious $\sO_Y$-module structure with  $v \in \sT_Y[1] \subset \sO_X$ acting as contraction with $\hbar v$ (whereas on $\sL\brh$ it acts as contraction with $v$).

\end{remark}

\begin{definition}
 Given a right $\sD_Y$-module  $\sN$ and $f \in \sO_Y$, by analogy with \cite{SabbahSaito} we define the right $\sD_Y(\!(\hbar)\!)$-module $\sN(\!(\hbar)\!)^f$ with underlying $\sO_Y(\!(\hbar)\!)$-module $\sN(\!(\hbar)\!)$ by setting
 \[
  (\sum_i a_i \hbar^i)\xi := (\sum_i (a_i\xi) \hbar^i) - (\sum_i a_{i+1}\xi(f) \hbar^i)
 \]
for $\xi \in \sT_Y \subset \sD_Y$, so $\xi(f) \in \sO_Y$. In other words, we take the connection on $\sN$ and subtract the contraction with $-\hbar^{-1}df$.

The canonical $\hbar\pd_{\hbar}$-connection on $\sN(\!(\hbar)\!)^f$ is given by $D(a):= \hbar\pd_{\hbar}(a) + \hbar^{-1}af$.


Given a left $\sD_Y$-module  $\sN$, we define $\sN(\!(\hbar)\!)^f$ similarly.
\end{definition}

On $X:=\oR\Crit(Y,f)$, the natural square root $i^*\sK_Y$ of the line bundle $\sK_X$ allows us to identify the $\mu_2$-twisted line bundle $\sqrt{\sK}_X$ of square roots of $\sK_X$ with the stack of line bundles of the form $(i^*\sK_Y)\ten_{\Z}\vareps$, for square roots $\vareps$ of the constant \'etale sheaf $\Z$ 
(i.e. $ \vareps\ten_{\Z}\vareps \cong \Z$). 

Substituting  into Corollary \ref{quantDmodcor2} for $(X,\omega, \lambda)$ 
gives a   canonical $\mu_2$-twisted right  $\sD_X\brh$-module $\sE$ on 
$X$
with non-degenerate sesquilinear pairing $\upsilon$ and orthogonal $\hbar^2\pd_{\hbar}$-connection $\phi$, deforming $\sDiff(\sO_X, \sqrt{\sK}_X)$. The twist is the orientation class of $\sK_X$, so we can regard $(\sE,\upsilon,\phi)$ as a functor from the stack of line bundles $(i^*\sK_Y)\ten_{\Z}\vareps$ to the stack of $\sD_X\brh$-modules with those pairings and connections.

The \'etale site of $X$ is equivalent to that of $\pi^0X$, which in turn is equivalent to that of the formal completion $\hat{Y}$ of $Y$ along $\pi^0X$. In particular, their stacks of $\mu_2$-torsors (equivalently, square roots of the \'etale sheaf $\Z$) are equivalent.


\begin{theorem}\label{DCritDmodthm}
For the $\mu_2$-twisted $\sD_X\brh$-module $\sE$ on $X:=\oR\Crit(Y,f)$ associated to the unique self-dual quantisation with $\hbar\pd_{\hbar}$-connection of the standard exact symplectic structure, we have a canonical quasi-isomorphism
\[
 i_!\sE(i^*\sK_Y\ten_{\Z} \vareps)[\hbar^{-1}] \simeq \hbar^{-\dim Y/2}(\sK_{Y})(\!(\hbar)\!)^f\ten_{\Z} \vareps
\]
of 
right $\sD_{Y}(\!(\hbar)\!)$-modules with $\hbar\pd_{\hbar}$-connection.  

  Its sesquilinear pairing is induced via Remark \ref{pushforwardpairingrmk}, for $m= \dim Y$, by the obvious $\hbar$-semilinear isomorphisms
\begin{align*}
  \sHom_{\sO_{Y}(\!(\hbar)\!)}(\sK_{Y}(\!(\hbar)\!), \hbar^{-m/2}\sK_{Y}(\!(\hbar)\!)^f\ten \vareps)&\cong \hbar^{-m} 
  \sHom_{\sO_{Y}(\!(\hbar)\!)}(\hbar^{-m/2}\sK_{Y}(\!(\hbar)\!)^f\ten \vareps,\sK_{Y}(\!(\hbar)\!)) \\
  \hbar^{-m/2}\sO_{Y}(\!(\hbar)\!)^f\ten_{\Z} \vareps &\cong \hbar^{-m/2}\sO_{Y}(\!(\hbar)\!)^{-f}\ten_{\Z} \Hom_{\Z}(\vareps,\Z)\\ 
 a(\hbar)\ten e &\mapsto a(-\hbar)\ten (e,-)
\end{align*}
of 
left  $\sD_{Y}(\!(\hbar)\!)$-modules with $\hbar\pd_{\hbar}$-connection.
\end{theorem}
\begin{proof}
We will simply write $\sE:= \sE(i^*\sK_Y)$; the general case follows by applying $\ten_{\Z}\vareps$ throughout. 
 Substituting Corollary \ref{quantDmodcor2} into Proposition \ref{DCritprop}, we have 
 \[
i_!\sE \cong {}_{\hbar d}\sDiff_{\sO_Y}(\sO_Y, (\Omega^*_Y\brh ,?\wedge df))[m].
 \]
 Now, the left action $l$ of $\sT_Y[1] \subset \sO_X$ on $i_!\sE$ is given by contraction, and then the commutator of $l(\xi)$ with the total differential  agrees with right multiplication by $\hbar \xi + \xi(f) \in \sD_Y\brh$. Consideration of local co-ordinates then shows that the surjective right $\sD_Y(\!(\hbar)\!)$-module morphism
 \[
i_!\sE[\hbar^{-1}]\to \sK_{Y}(\!(\hbar)\!)^f,
 \]
given in degree $0$ by identifying $\alpha \hbar\xi$ with $ -\alpha \xi(f)$ for $\alpha \in \sDiff_{\sO_Y}(\sO_Y, \Omega^m_Y)$, is a quasi-isomorphism.   
 
The $\hbar\pd_{\hbar}$-connection on $i_!\sE[\hbar^{-1}]$ from  Corollary \ref{quantDmodcor2} and Remark \ref{pushforwardpairingrmk}  corresponds under this quasi-isomorphism to $D- \frac{m}{2}$, for the canonical $\hbar\pd_{\hbar}$-connection $D$ on $\sK_{Y}(\!(\hbar)\!)^f$, so the quasi-isomorphism of  $\sD_Y(\!(\hbar)\!)$-modules with  $\hbar\pd_{\hbar}$-connections $D$ is given by
 \[
  i_!\sE[\hbar^{-1}]\simeq \hbar^{-m/2}\sK_{Y}(\!(\hbar)\!)^f.
 \]
 
 The statement about the pairing then follows immediately from the construction in the proof of Lemma \ref{quantDmodlemma}.
\end{proof}

\begin{example}\label{quadDCritex}
 The $(-1)$-shifted symplectic scheme $X:=\oR\Crit(\bA^1,t^2/2)$ is quasi-isomorphic to $\{0\}$, so for the inclusion $i_0 \co \{0\} \to \bA^1$ and the derived closed immersion $i \co \oR\Crit(\bA^1,t^2/2) \to \bA^1$, Theorem \ref{DCritDmodthm} implies that we must have a canonical isomorphism 
\[
  \hbar^{ -\half}\Omega^1_{\bA^1}(\!(\hbar)\!)^{\frac{t^2}{2}} \cong \sDiff_{\bA^1}(\sO_{\bA^1}, i_{0*}i_0^*\Omega^1_{\bA^1})(\!(\hbar)\!)= i_{0!} \sO_{\bA^0}(\!(\hbar)\!)dt
\]
of right $\sD_{\bA^1}(\!(\hbar)\!)$-modules, intertwining the respective $\hbar^2\pd_{\hbar}$ connections and sesquilinear pairings, 
since both modules are models for $i_!\sE(i^*\Omega^1_{\bA^1})[\hbar^{-1}]$.

From the proof of Theorem \ref{DCritDmodthm}, the respective modules are quasi-isomorphic to (and $\sH_0$ of) the complexes
\[
{}_{\hbar d}\sDiff_{\sO_{\bA^1}}(\sO_{\bA^1}, (\Omega^*_Y (\!(\hbar)\!) ,?\wedge tdt))[1]
\quad \text{and} \quad {}_{0}\sDiff_{\sO_{\bA^1}}(\sO_{\bA^1}, (\Omega^*_Y(\!(\hbar)\!)  ,?\wedge tdt))[1]
\] 
with their respective $\hbar\pd_{\hbar}$-connections  $\hbar\pd_{\hbar} + \half -\deg_{\Omega}-\hbar^{-1}\frac{t^2}{2}$ and $\hbar\pd_{\hbar}$,
corresponding to twists by $(\Delta,D)=(\hbar d, \half - \deg_{\Omega} -\hbar^{-1}\frac{t^2}{2}) $ and $(0,0)$, respectively. 

Let  $\tau \in \sT_{\bA^1}$ be the dual vector to $dt$, so $\oR\Crit(\bA^1,t^2/2)\cong R[t,\tau]$, for $\tau$ of degree $1$ with $\delta \tau = \frac{\pd}{\pd t}(\frac{t^2}{2})=t$. We then have $\Delta (a dt)= \hbar \pd_t\pd_{\tau}(a)dt$ and $D(a dt)= (\tau\pd_{\tau} - \half -\hbar^{-1}\frac{t^2}{2})(a)dt$ for all $a \in \sO_X$. Note that $\delta \pd_t =\pd_{\tau}$, with the latter of chain degree $-1$ and necessarily square-zero. The anti-involution on the ring of differential operators  reverses the signs of the operators $adt \mapsto  \pd_t(a)dt$ and $a dt \mapsto \pd_{\tau}(a)dt$ while fixing multiplication by $t$ and $\tau$.

If for all $a \in \sO_X$ we set 
\[
\phi(a dt):= \hbar\pd_{t}^2(a) dt/2,
\]
 then $\delta(\exp(\phi)) \exp(-\phi)= \Delta$. Also, $\phi$ is self-dual so it gives a gauge transformation between $0$ and $\Delta$, in the sense that $\delta  = \exp(-\phi) \circ(\delta + \Delta) \circ \exp(\phi)$. Thus multiplication by $\exp(-\phi)$ gives an isomorphism from the complex with connection  associated to $(\Delta,D)$ to that associated to $(0, \exp(-\phi)D\exp(\phi) + \exp(-\phi)\hbar\pd_{\hbar}(\exp(\phi)))$. At this point, we know that the new connection term must itself be homotopic to $0$, by acyclicity of the relevant complex,  so will  necessarily act trivially  on $\sH_0$.

For completeness, we now construct the homotopy. We have
\begin{align*}
  -\exp(-\phi)\hbar\pd_{\hbar}(\exp(\phi)) + \exp(-\phi)D\exp(\phi)  &= \hbar\pd_{\hbar}(\phi) + \sum_{n \ge 0} (-\ad_{\phi})^n(D)/n!\\
  &= \phi + D - [\phi,D] +\half [\phi,[\phi,D]] + \ldots .
\end{align*}
Omitting the generator $dt$ to simplify the notation, we have
\begin{align*}
 \ad_{\phi}(D)= [  \hbar\pd_{t}^2/2,  -\half +\tau\pd_{\tau} -\hbar^{-1}t^2/2] &=  -t\pd_t -\half, \\
 \ad_{\phi}^2(D)=  [  \hbar\pd_{t}^2/2, - t\pd_t -\half] &=  -\hbar\pd_t^2, \\
 \ad_{\phi}^3(D) = [  \hbar\pd_{t}^2/2,  -\hbar\pd_t^2] &=0.
\end{align*}
Thus
\begin{align*}
&  \phi + D - [\phi,D] +\half [\phi,[\phi,D]] + \ldots\\ &=  \hbar\pd_{t}^2/2 +(-\half +\tau\pd_{\tau} -\hbar^{-1}t^2/2) - (-t\pd_t -\half) +\half(-\hbar\pd_t^2)\\
 &= \tau\pd_{\tau} + t\pd_t -\hbar^{-1}t^2/2,
 \end{align*}
 which  is just $\delta( \tau \pd_t -\hbar^{-1}\tau t/2)$, so
we can take the homotopy 
\begin{align*}
\theta(a dt)&:= (\tau \pd_t -\hbar^{-1}\tau t/2)(a)dt, \quad a \in \sO_X;\\
\theta(b + cdt) &=  (\frac{\pd c}{\pd t} -\hbar^{-1} ct/2)   \quad b,c \in  \sO_{\bA^1},
\end{align*}
between connections to get from $(0, \exp(-\phi)D\exp(\phi) + \exp(-\phi)\hbar\pd_{\hbar}(\exp(\phi)))$ to $(0,0)$. 

\end{example}

\subsubsection{Monodromic perverse sheaves}\label{monodromicsn}

In the $\Cx$-analytic setting, the  Riemann--Hilbert correspondence allows us to associate a sheaf of $\Cx\brh$-modules on $X(\Cx)=\pi^0X(\Cx)$ to each right $\sD_X\brh$-module $\sE$, by sending $\sE$ to $\sD\sR^r_X(\sE):=\sE\hten^{\oL}_{\sD_X\brh}\sO_X\brh$. When $\sE$ carries an $\hbar^2\pd_{\hbar}$-connection, the same is true of $\sD\sR^r_X(\sE)$. When $\sE$ carries a  sesquilinear pairing $ (\sE)^{\ten 2} \to \sDiff_{\sO_X}(\sO_X, \sM)\brh$, then the reasoning of Remark \ref{pushforwardpairingrmk}
gives us a sesquilinear pairing
\[
\sD\sR^r_X(\sE)\by \sD\sR^r_X(\sE) \to \sD\sR^r_X(\sM\brh). 
\]
When $X$ is Gorenstein and $\sM=\sK_X$, the determinant bundle, then the proof of \cite[Lemma \ref{DQm2-BMlemma}]{DQ-2} identifies the target $ \sD\sR^r_X(\sM\brh)$ with the Verdier $\Cx$-dualising complex $\bD_{\pi^0X(\Cx)}\brh$. In particular, if $i \co \pi^0X \to Y$ is a closed immersion to a smooth scheme, we have a sesquilinear pairing $i_*\sD\sR^r_X(\sE)\by i_*\sD\sR^r_X(\sE) \to \bD_{Y(\Cx)}\brh \simeq \Cx\brh[2 \dim Y]$.

Given a line bundle $\sM$ on $X$ with right $\sD_X$-module structure, applying $\sD\sR^r_X$ to  Corollary \ref{quantDmodcor} thus  gives a natural map $(\Delta,D) \mapsto ({}_{\Delta}\sqrt{\sM}\brh, \hbar D)$ 
from the space $Q\cP^D(X,-1)^{sd}$ to the space of triples $(\vv,\upsilon,\phi)$ for $\vv$ a flat $\mu_2$-twisted   $\Cx\brh$-module with $\vv/\hbar \vv \cong \sqrt{\sM}$, equipped with a sesquilinear pairing $\upsilon \co \vv \by \vv \to \sD\sR^r(\sM\brh)$ and an  orthogonal $\hbar^2\pd_{\hbar}$-connection $\phi\co \vv \to \vv$.

In particular, if $X$ is an exact $(-1)$-shifted symplectic $\Cx$-analytic derived DM $n$-stack, then as in Corollary \ref{quantDmodcor2} we can apply this to the essentially unique self-dual deformation quantisation $\Delta$ with $\hbar\pd_{\hbar}$-connection $D$, taking $\sM=\sK_X$.  

For the twisted de Rham complex $\sD\sR_{f,Y} = (\Omega^{\#}_Y(\!(\hbar)\!), d + \hbar^{-1} df)[\dim Y]$, 
we  then have the following immediate consequence of Theorem \ref{DCritDmodthm}. 

\begin{corollary}\label{DCritDiffPervcor}
 For the $\mu_2$-twisted $\sD_X\brh$-module $\sE$ on $X:=\oR\Crit(Y,f)$ associated to the unique self-dual quantisation $\Delta$ with $\hbar\pd_{\hbar}$-connection $D$ of the standard exact symplectic structure, we have a canonical quasi-isomorphism
\[
 i_*\sD\sR^r_X(\sE)[\hbar^{-1}] (i^*\sK_Y\ten \vareps) \simeq (\sD\sR_{f,Y}\ten_{\Z} \vareps,  \hbar\pd_{\hbar}- \hbar^{-1}f - \dim Y/2)
\]
of perverse  $\Cx(\!(\hbar)\!)$-sheaves with $\hbar\pd_{\hbar}$-connection supported on $\pi^0X$.

The sesquilinear pairing 
$i_*\sD\sR^r_X(\sE)[\hbar^{-1}]\by i_*\sD\sR^r_X(\sE)[\hbar^{-1}] \to \Cx(\!(\hbar)\!) [2 \dim Y]$ corresponds to  the  $\Cx(\!(\hbar)\!)$-linear cup product
\begin{align*}
 &(\sD\sR_{-f,Y}\ten_{\Z} \vareps,\hbar\pd_{\hbar}+ \hbar^{-1}f - \dim Y/2 )\ten_{\Cx(\!(\hbar)\!)} (\sD\sR_{f,Y}\ten_{\Z} \vareps, \hbar\pd_{\hbar}- \hbar^{-1}f - \dim Y/2 ) \\
 &\xra{\hbar^{-\dim Y}\smile}  (\Omega^{\bt}_Y(\!(\hbar)\!)[2 \dim Y], \hbar\pd_{\hbar})
\end{align*}
under the $\hbar$-semilinear isomorphism $\sD\sR_{f,Y} \to \sD\sR_{-f,Y}$ sending $a(\hbar)$ to $a(-\hbar)$ and the isomorphism $\vareps\ten_{\Z}\vareps \cong \Z$.
 \end{corollary}

For the Riemann--Hilbert correspondence $\mathrm{RH}^{-1}$ relating differential perverse sheaves and perverse sheaves with an automorphism as in \cite[\S 2.3]{GunninghamSafronov} or \cite[\S 4.2]{SchefersDerivedVFiltrn}, we thus  have the following.

\begin{corollary}\label{DcritAutPervCor}
 For the $\mu_2$-twisted $\sD_X\brh$-module $\sE$ on $X:=\oR\Crit(Y,f)$ associated to the unique self-dual quantisation $\Delta$ with $\hbar\pd_{\hbar}$-connection $D$ of the standard exact symplectic structure, we have a canonical quasi-isomorphism
\[
 \sD\sR^r_X(\sE)[\hbar^{-1}] (i^*\sK_Y\ten \vareps) \simeq \mathrm{RH}^{-1} (\phi_f\ten_{\Z} \vareps, (-1)^{\dim Y} T) 
 \]
of differential perverse sheaves on $\pi^0X=\Crit(Y,f)$, 
assuming $\pi^0X \subset f^{-1}\{0\}$, where $\phi_f:= \phi_f(\uline{\Cx}_Y[\dim Y])$ is the  perverse sheaf of vanishing cycles. 

Under this quasi-isomorphism, the sesquilinear pairing on $\sE$ corresponds, for $i \co \pi^0X \into Y$, to the $\Cx$-linear pairing
\[
(i_*\phi_{-f}\ten_{\Z} \vareps, (-1)^{\dim Y} T)\ten_{\Cx}(i_*\phi_f\ten_{\Z} \vareps, (-1)^{\dim Y} T)\to  (\uline{\Cx}_Y [2 \dim Y], \id)  
\] 
of monodromic sheaves on $Y$ given by combining the Thom--Sebastiani isomorphism $\phi_{(-f)\boxplus f} \cong \phi_{-f}\boxtimes \phi_f$ on $\pi^0(X \by X)$ with 
functoriality along the diagonal 
$
\Delta \co (Y,0) \to (Y\by Y,(-f)\boxplus f)$.  
 \end{corollary}
\begin{proof}
 By \cite[Theorem 8.21]{SchefersDerivedVFiltrn} or \cite[Corollary 5.11]{GunninghamSafronov} (tensoring the latter by $\vareps$), we have an equivalence 
 \[
  \mathrm{RH}^{-1} (\phi_f\ten_{\Z} \vareps,  T) \simeq (\sD\sR_{f,Y}\ten_{\Z} \vareps,  \hbar\pd_{\hbar}- \hbar^{-1}f).
 \]
Since $\mathrm{RH}(\Cx(\!(\hbar)\!), \hbar\pd_{\hbar} +\lambda) \cong  (\Cx, \exp(2 \pi i \lambda))$, the 
		first
result follows by taking the  tensor product with that for $\lambda = - \dim Y/2$,   then applying Corollary \ref{DCritDiffPervcor}. 

The cup product pairing of Corollary \ref{DCritDiffPervcor} can be written as 
\[
\Delta^{-1}(\sD\sR_{(-f)\boxplus f,Y\by Y}\ten_{\Z} (\vareps\boxtimes \vareps),\hbar\pd_{\hbar}- \hbar^{-1}((-f)\boxplus f) - \dim Y ) \xra{\hbar^{-\dim Y}\Delta^{\sharp}}  (\sD\sR_{0,Y},  \hbar\pd_{\hbar})[\dim Y]
\]
since $\Delta^{-1}\sD\sR_{(-f)\boxplus f,Y\by Y} = \sD\sR_{-f,Y}\ten_{\Cx}\sD\sR_{f,Y}$   
and $\Delta^{-1}(\vareps\boxtimes \vareps) =\uline{\Z}_Y$. By compatibility of $\mathrm{RH}^{-1}$ with the Thom--Sebastiani isomorphism \cite[Proposition 5.15]{GunninghamSafronov}  and functoriality, this corresponds (after shifting by $2 \dim Y$ and $\Cx$-linearising) to the  pairing
\[
i_*\phi_{-f}(\vareps) \ten i_*\phi_f(\vareps) = \Delta^{-1}(i,i)_*\phi_{(-f)\boxplus f}(\vareps \boxtimes \vareps)
\xra{\Delta^{\sharp}} 
\phi_0(\Delta^{-1}(\vareps\boxtimes \vareps))
= \vareps\ten \vareps = \uline{\Z}_Y
\]
on $Y$. 
\end{proof}

\begin{remark}\label{pairingmonodromyrmk}
  For any $a \in \Cx$, translating the universal cover $\Cx$ of $\Cx^*$ by $a$ gives an isomorphism $\phi_f \cong \phi_{\exp(2 \pi i a)f}$, with $a=1$ corresponding to the monodromy automorphism $T$. Thus taking $a=\half$ gives  isomorphisms $ T_f^{\half} \co \phi_f \longleftrightarrow \phi_{-f} \co T_{-f}^{\half}$
 which  are not inverses, but instead factorise the monodromy operator $T$ on both sides. 

Under the isomorphism $T_f^{\half}$ (denoted $\tilde{T}_f^{\pi}$ in the proof of \cite[Theorem 3.1]{masseyVanishing}), the pairing between $i_*\phi_{-f}$ and $i_*\phi_f$ in Corollary  \ref{DcritAutPervCor} corresponds to the Verdier self-duality pairing  $\phi_f\ten \phi_f \to i^!\uline{\Cx}_Y[2 \dim Y] = \bD_{\pi^0X} $ in \cite{masseyVanishing}, via the  characterisation in the proofs of \cite[Propositions A.1 and A.2]{BBDJS}. 


\end{remark}

\subsection{Comparison with \tps{\cite{BBDJS}}{BBDJS} and \tps{\cite{HHR1}}{HHR}}

A perverse sheaf of vanishing cycles  on $(-1)$-shifted symplectic derived $\Cx$-schemes $(X,\omega)$ is constructed  in \cite[Corollary 6.11]{BBDJS} by gluing together local charts by derived critical loci. That gluing method is systematised and refined in \cite{HHR1} to induce constructions on $(-1)$-shifted symplectic derived DM $\Cx$-stacks from those on Landau-Ginzburg pairs. Our construction from Corollary \ref{quantDmodcor2} does not depend on  derived critical charts, but the characterisation of \cite{HHR1} isolates the conditions needed to  establish a comparison with \cite{BBDJS}.
 
\cite[Theorem A]{HHR1} shows that the contractible \'etale  sheaf of $\infty$-groupoids on $X$ can be expressed as  a quotient  $\uline{\mathrm{Darb}}_X^{\omega,\bA^1}/\uline{\mathrm{Quad}}_X^{\nabla,\bA^1}$ of a stack of Darboux data by the action of a monoid stack of quadratic data. In \cite[\S 6.3]{HHR1}, this is applied to construct a vanishing cycles sheaf on $X$, proved to be equivalent to the construction of \cite{BBDJS}.

Specifically, \cite[Construction 6.33]{HHR1} gives a functor $\mathup{P} \co \uline{\mathrm{Darb}}_X^{\omega,\bA^1} \to \uline{\mathrm{Perv}}_X$ to the $1$-category of perverse sheaves on $\pi^0X$, with $ \mathup{P}(U,f):=\phi_f(\uline{\Z}_U[\dim U])$, and \cite[Corollary 6.38]{HHR1} refines this to a functor $\uline{\mathrm{Darb}}_X^{\omega,\bA^1}/\uline{\mathrm{Quad}}_X^{\nabla,\bA^1} \to \uline{\mathrm{Perv}}_X/B\mu_2$ to the $2$-category of $\mu_2$-twisted perverse sheaves. Since the source is contractible, this canonically associates a $\mu_2$-twisted perverse sheaf to $(X,\omega)$.

As in \cite[Theorem 6.3.8]{HHR1}, the $\mu$-twist is the orientation class of $\sK_X$, so by taking the fibre over $o(\sK_X) \in \Gamma(X,B^2\mu_2)$  we can phrase this as an object in the $1$-category $(\uline{\mathrm{Perv}}_X/B\mu_2)\by^h_{B^2\mu_2}o(\sK_X)$ of $o(\sK_X)$-twisted perverse sheaves.

For the functors $\sD\sR^r_X$ and $\mathrm{RH}^{-1}$ as in \S \ref{monodromicsn}, and $X_{\an}$ the analytification of $X$, we then have the following. 
\begin{theorem}\label{cfHHRthm}
 Given a $(-1)$-shifted symplectic derived DM $\Cx$-stack $(X,\omega)$, the following $\mu_2$-twisted differential perverse sheaves  
 on $\pi^0X_{\an}$   with $\hbar$-sesquilinear $\bD_{\pi^0X^{\an}}(\!(\hbar)\!)$-valued pairing are canonically isomorphic:
 \begin{itemize}
  \item $\sD\sR^r_{X_{\an}}(\sE)[\hbar^{-1}]$,  
  for $\sE$ the formal deformation of  $\sDiff(\sO_X^{\an}, \sqrt{\sK}_X^{\an})$ given by applying Corollary \ref{quantDmodcor2} 
  to $(X_{\an}, \omega^{\an})$
  with its canonical exact structure, and 
    \item $\mathrm{RH}^{-1}(\mathup{P}_{X,\omega}\ten_{\Z}\Cx)$, for the perverse sheaf $\mathup{P}_{X,\omega}$ 
    of \cite[Theorem 6.3.9]{HHR1} 
 equipped with its canonical monodromy operator
 and  Verdier duality pairing $\mathup{P}_{X,-\omega}\ten_{\Z}\mathup{P}_{X,\omega} \to \bD_{\pi^0X^{\an}}$, where $\mathup{P}_{X,-\omega}\cong \mathup{P}_{X,\omega}$ via $T^{\half}$. 
 \end{itemize}
 \end{theorem}
 In particular, a choice of square root for $\sK_X$ gives a canonical isomorphism between  $\mathrm{RH}(\sD\sR^r_{X_{\an}}(\sE)[\hbar^{-1}])$ and the self-dual monodromic perverse sheaf $(\mathup{P}^{\bt}_{X,\omega}\ten_{\Z}\Cx, T_{X,\omega})$ of  \cite[Theorem 6.9]{BBDJS}, by \cite[Remark 6.3.5]{HHR1}. 
\begin{proof}
Both perverse sheaves have $\mu_2$-twist $o(\sK_X)$, so we can work in the $1$-category of complex $o(\sK_X)$-twisted  differential perverse sheaves on $\pi^0X_{\an}$ equipped with $\hbar$-sesquilinear pairings. 

Since Corollary \ref{DcritAutPervCor} satisfies \'etale functoriality in $Y$, it gives a natural isomorphism $\alpha$ of functors on $\uline{\mathrm{Darb}}_X^{\omega,\bA^1}$ between the constant functor  $\sD\sR^r_{X_{\an}}(\sE)[\hbar^{-1}]$ and the functor induced by the algebraic critical invariant $(U,f) \mapsto \mathrm{RH}^{-1}(\mathup{P}(U,f)\ten_{\Z}\Cx, (-1)^{\dim U}T)$ of \cite[Construction 6.3.3]{HHR1}.
Under this isomorphism, the sesquilinear pairing of Corollary \ref{DcritAutPervCor} corresponds to the canonical Verdier duality pairing between   $\mathup{P}(U,f)$  and $ \mathup{P}(U,-f) =T^{\half}\mathup{P}(U,f)$; as in Remark \ref{pairingmonodromyrmk}, composing this with  $T^{\half}$ gives the  self-duality isomorphism on  $\mathup{P}(U,f)$ from \cite{BBDJS}. 

Since our target is a $1$-category, to show that this descends to a natural isomorphism of functors on $\uline{\mathrm{Darb}}_X^{\omega,\bA^1}/\uline{\mathrm{Quad}}_X^{\nabla,\bA^1} \simeq \ast_X$, it suffices to show locally that for every non-degenerate quadratic bundle $(Q,q)$ of rank $r$ on $U$, with $m=\dim U$,  the diagram
\[
 \begin{CD} \sD\sR^r_{X_{\an}}(\sE)[\hbar^{-1}]|_{\oR\Crit(U,f)}  @>{\alpha_{U,f}}>> \mathrm{RH}^{-1}(\mathup{P}(U,f)\ten_{\Z}\Cx, (-1)^{m}T) \\
 @V{\eta_{Q,q,f}}VV @VV{\gamma_{Q,q,f}}V\\
 \sD\sR^r_{X_{\an}}(\sE)[\hbar^{-1}]|_{\oR\Crit(Q,q+f)}   @>{\alpha_{Q, q+f}}>> \mathrm{RH}^{-1}(\mathup{P}(Q,q+f)\ten_{\Z}\Cx, (-1)^{m+r}T)
   \end{CD}
\]
commutes, where $\eta$ is the equivalence induced by the canonical equivalence $\oR\Crit(U,f) \into \oR\Crit(Q,q+f) $ of $(-1)$-shifted symplectic derived stacks, and $\gamma$ is the equivalence of \cite[Lemma 6.3.6b]{HHR1} or \cite[Theorem 5.4]{BBDJS}.

Since the target of the functor is an \'etale stack, it suffices to work locally, so we may assume that $(Q,q)= (\bA^r_U, \sum_i x_i^2)$. 
By transitivity of the monoid action, this in turn reduces to the case $r=1$.

The isomorphism $\gamma_{\bA^1_U ,x^2 ,f} $ combines the Thom--Sebastiani isomorphism 
\[
  (\mathup{P}(\bA^1 \by U  ,x^2+ f ), (-1)^{m+1}T) \cong   (\mathup{P}(\bA^1 ,x^2), -T) \boxtimes_{\Z}(\mathup{P}(U ,f), (-1)^{m}T)
 \]
with a choice $(\mathup{P}(\bA^1 ,x^2), -T) \cong (\Z,\id)$ of isomorphism as in \cite[Example 2.14]{BBDJS}.

On the other side, the quasi-isomorphism $\eta_{\bA^1_U,x^2,f}$ corresponds under Corollary \ref{DcritAutPervCor} to a quasi-isomorphism
\[
  (\sD\sR_{f,U_{\an}}\ten_{\Z} \vareps,  \hbar\pd_{\hbar}- \hbar^{-1}f - \frac{m}{2}) \to 
(\sD\sR_{x^2+f,\bA^1\by U_{\an}}\ten_{\Z} \vareps,  \hbar\pd_{\hbar}- \hbar^{-1}(x^2+f) - \frac{m+1}{2}).
 \]
 Self-duality of quantisations is preserved by external tensor products, so this corresponds to tensoring with the quasi-isomorphism
 \[
\eta_{\bA^1,x^2,0} \co  (\sD\sR_{0,\Spec \Cx}, \hbar\pd_{\hbar}) \to (\sD\sR_{x^2,\bA^1_{\an}}, \hbar\pd_{\hbar}- \hbar^{-1}(x^2) -\half). 
 \]

 The  Thom--Sebastiani isomorphism also corresponds under the equivalence of  \cite[Theorem 8.21]{SchefersDerivedVFiltrn} or \cite[Corollary 5.11]{GunninghamSafronov} to the K\"unneth isomorphism for twisted de Rham cohomology, by \cite[Proposition 5.15]{GunninghamSafronov}. Our comparison is thus determined by the respective trivialisations of $\mathrm{RH}^{-1}(\mathup{P}(\bA^1 ,x^2)\ten_{\Z}\Cx, -T)$. Since we are then comparing isomorphisms between one-dimensional inner product spaces, the choice is unique up to multiplication by  
 $c \in \{\pm 1\} = \{ c(\hbar) \in \ker(\pd_{\hbar}) \cap \Cx(\!(\hbar)\!) ~:~ c(-\hbar)c(\hbar) =1 \}$. 
 As observed in \cite[Remark 6.10]{GunninghamSafronov}, multiplying $ \alpha_{U,f}$ by $c^{\dim U}$ then gives a natural isomorphism with the required compatibility. 
\end{proof}

\bibliographystyle{alphanum}
\bibliography{references.bib}

\end{document}

%% file: newcommands.tex
\newcommand\la{\leftarrow}

\newcommand\id{\mathrm{id}}

\newcommand\ten{\otimes}
\newcommand\hten{\hat{\otimes}}

\newcommand\vareps{\varepsilon}
\newcommand\eps{\epsilon}

\newcommand\GT{\mathrm{GT}}

\renewcommand\H{\mathrm{H}}
\newcommand\z{\mathrm{Z}}

\newcommand\HH{\mathrm{HH}}

\newcommand\Z{\mathbb{Z}}
\newcommand\Q{\mathbb{Q}}

\newcommand\R{\mathbb{R}}
\newcommand\Cx{\mathbb{C}}

\newcommand\vv{\mathbb{V}}

\newcommand\bA{\mathbb{A}}

\newcommand\bD{\mathbb{D}}

\newcommand\bG{\mathbb{G}}

\newcommand\bK{\mathbb{K}}
\newcommand\bL{\mathbb{L}}

\newcommand\bP{\mathbb{P}}

\newcommand\bS{\mathbb{S}}

\newcommand\C{\mathcal{C}}

\newcommand\cD{\mathcal{D}}

\newcommand\cP{\mathcal{P}}
\newcommand\cQ{\mathcal{Q}}

\newcommand\cU{\mathcal{U}}

\newcommand\cW{\mathcal{W}}

\newcommand\sA{\mathscr{A}}

\newcommand\sD{\mathscr{D}}
\newcommand\sE{\mathscr{E}}
\newcommand\sF{\mathscr{F}}

\newcommand\sH{\mathscr{H}}

\newcommand\sK{\mathscr{K}}
\newcommand\sL{\mathscr{L}}
\newcommand\sM{\mathscr{M}}
\newcommand\sN{\mathscr{N}}
\newcommand\sO{\mathscr{O}}
\newcommand\sP{\mathscr{P}}

\newcommand\sR{\mathscr{R}}

\newcommand\sT{\mathscr{T}}

\newcommand\X{\mathfrak{X}}

\newcommand\fS{\mathfrak{S}}

\newcommand\fX{\mathfrak{X}}
\newcommand\fY{\mathfrak{Y}}

\renewcommand\L{\Lambda}

\newcommand\sHom{\mathscr{H}\!\mathit{om}}
\newcommand\sHHom{\uline{\sHom}}

\newcommand\sDiff{\mathscr{D}\!\mathit{iff}}

\newcommand\Alg{\mathrm{Alg}}

\newcommand\Ass{\mathrm{Ass}}

\newcommand\Mod{\mathrm{Mod}}

\newcommand\Hom{\mathrm{Hom}}

\newcommand\map{\mathrm{map}}

\newcommand\cyc{\mathrm{cyc}}

\newcommand\ctr{\mathrm{ctr}}

\newcommand\HHom{\underline{\mathrm{Hom}}}

\newcommand\DDer{\underline{\mathrm{Der}}}

\newcommand\End{\mathrm{End}}

\newcommand\cocone{\mathrm{cocone}}
\newcommand\dg{\mathrm{dg}}

\newcommand{\brh}{\llbracket \hbar \rrbracket}
\newcommand{\brhh}{\llbracket \hbar^2 \rrbracket}

\newcommand\im{\mathrm{Im\,}}

\newcommand\ex{\mathrm{ex}}

\newcommand\Co{\mathrm{Co}}
\newcommand\CoS{\mathrm{CoS}}

\newcommand\loc{\mathrm{loc}}

\DeclareMathOperator\Spec{\mathrm{Spec}\,}

\newcommand\Sp{\mathrm{Sp}}

\newcommand\BiSp{\mathrm{BiSp}}

\newcommand\Pol{\mathrm{Pol}}

\newcommand\Lag{\mathrm{Lag}}
\newcommand\poly{\mathrm{poly}}
\newcommand\Com{\mathrm{Com}}
\newcommand\Comp{\mathrm{Comp}}

\newcommand\nondeg{\mathrm{nondeg}}

\newcommand\TDO{\mathrm{TDO}}
\newcommand\inv{\mathrm{inv}}

\newcommand\ad{\mathrm{ad}}

\renewcommand\d{\tilde{d}}
\newcommand\<{\langle}
\renewcommand\>{\rangle}
\newcommand\Lim{\varprojlim}

\newcommand\into{\hookrightarrow}

\newcommand\xra{\xrightarrow}
\newcommand\xla{\xleftarrow}

\newcommand\pr{\mathrm{pr}}
\newcommand\pre{\mathrm{pre}}

\newcommand\bt{\bullet}
\newcommand\by{\times}

\DeclareMathOperator\mc{\mathrm{MC}}
\DeclareMathOperator\mmc{\underline{\mathrm{MC}}}

\newcommand\Rees{\mathrm{Rees}}
\newcommand\Symm{\mathrm{Symm}}

\newcommand\et{\acute{\mathrm{e}}\mathrm{t}}

\newcommand\an{\mathrm{an}}

\newcommand\Tot{\mathrm{Tot}\,}

\newcommand\pd{\partial}

\newcommand\half{\frac{1}{2}}
\newcommand\Ad{\mathrm{Ad}}

\newcommand\rig{\mathrm{rig}}

\newcommand\Crit{\mathrm{Crit}}

\newcommand\gr{\mathrm{gr}}

\newcommand\Lie{\mathrm{Lie}}

\newcommand\dR{\mathrm{dR}}
\newcommand\DR{\mathrm{DR}}

\newcommand\op{\mathrm{opp}}

\newcommand\co{\colon\thinspace}

\newcommand\oR{\mathbf{R}}

\newcommand\oL{\mathbf{L}}

\newcommand\uleft\underleftarrow
\newcommand\uline\underline
\newcommand\uright\underrightarrow

\newcommand{\tps}{\texorpdfstring}

\DeclareMathOperator{\hatHHom}{\uline{\mathrm{H}\widehat{\mathrm{o}}\mathrm{m}}}
\DeclareMathOperator{\hatTot}{\mathrm{T}\widehat{\mathrm{o}}\mathrm{t}}

%% file: newthms.tex
\newtheorem{theorem}{Theorem}[section]
\newtheorem{proposition}[theorem]{Proposition}
\newtheorem{corollary}[theorem]{Corollary}

\newtheorem{lemma}[theorem]{Lemma}
\newtheorem*{theorem*}{Theorem}
\newtheorem*{proposition*}{Proposition}
\newtheorem*{corollary*}{Corollary}
\newtheorem*{lemma*}{Lemma}
\newtheorem*{conjecture*}{Conjecture}

\theoremstyle{definition}
\newtheorem{definition}[theorem]{Definition}
\newtheorem*{definition*}{Definition}

\newtheorem*{notation*}{Notation}

\newtheorem{setting}[theorem]{Setting}

\theoremstyle{remark}
\newtheorem{example}[theorem]{Example}
\newtheorem{examples}[theorem]{Examples}
\newtheorem{remark}[theorem]{Remark}
\newtheorem{remarks}[theorem]{Remarks}

\newtheorem*{example*}{Example}
\newtheorem*{examples*}{Examples}
\newtheorem*{remark*}{Remark}
\newtheorem*{remarks*}{Remarks}
\newtheorem*{exercise*}{Exercise}
\newtheorem*{property*}{Property}
\newtheorem*{properties*}{Properties}